\documentclass[reqno]{amsart}
\usepackage[T1]{fontenc}
\usepackage{verbatim}
\let\strokeL\L
\usepackage{enumitem,amssymb}
\usepackage{geometry}
\usepackage[all]{xy}
\usepackage[colorlinks]{hyperref}
\newcommand{\triplearrows}{\begin{smallmatrix} \to \\ \to \\ 
\to \end{smallmatrix} }
\usepackage{cleveref}
\newcommand{\sF}{\mathcal{F}}
\newcommand{\spec}{\mathrm{Spec}}
\newcommand{\et}{\mathrm{\acute{e}t}}
\newcommand{\arc}{\mathrm{arc}}
\renewcommand{\hom}{\mathrm{Hom}}
\newcommand{\cl}{\colon}
\newcommand{\D}{\mathcal{D}}

\newtheorem{theorem}{Theorem}[section]
\newtheorem{proposition}[theorem]{Proposition}
\newtheorem{lemma}[theorem]{Lemma}
\newtheorem{corollary}[theorem]{Corollary}

\theoremstyle{definition}
\newtheorem{definition}[theorem]{Definition}

\newtheorem{remark}[theorem]{Remark}
\newtheorem{example}[theorem]{Example}
\newtheorem{notation}[theorem]{Notation}
\newtheorem{warning}[theorem]{Warning}

\newtheorem{cons}[theorem]{Construction}

\newtheorem*{convention*}{Conventions}

\begin{document}

\title{The $\arc$-topology}
\author{Bhargav Bhatt}
\author{Akhil Mathew}
\maketitle

\begin{abstract}
We study a Grothendieck topology on  schemes
which we call the $\arc$-topology. 
This topology is a 
refinement of the $v$-topology (the pro-version of Voevodsky's $h$-topology)
where 
covers are tested via rank $\leq 1$ valuation rings.
Functors which are $\arc$-sheaves are forced to satisfy a variety of glueing conditions
such as excision in the sense of algebraic $K$-theory. 

We show that \'etale cohomology 
is an $\arc$-sheaf and deduce various pullback squares in \'etale cohomology.  
Using $\arc$-descent, we  reprove
the Gabber--Huber affine analog of proper base change (in
a large class of examples), as well as the Fujiwara--Gabber base change theorem
on the \'etale cohomology of the complement of a henselian pair.
As a final application we prove a rigid analytic version of the
Artin--Grothendieck vanishing
theorem from SGA4, extending results of Hansen. 
\end{abstract}

\tableofcontents

\section{Introduction}

The purpose of this paper is to 
study a Grothendieck topology 
on the category of schemes, which we call the
\emph{$\arc$-topology}. 
This topology is a slight refinement of the $v$-topology of 
\cite{Rydh, BhattScholzeWitt}.  

The benefit of such extremely fine topologies arises when one studies invariants
of schemes which satisfy \emph{descent} with respect to them. In this case, one
can try to work locally. For the $v$-topology, working locally essentially
means that one can reduce many questions about these invariants (on any
scheme) to potentially far simpler questions involving valuation rings, even
ones which have algebraically closed fraction field.
Our strengthening in this paper shows that, for $\arc$-sheaves, one can restrict even further, to rank $\leq
1$ valuation rings. (Recall that the rank of a valuation ring is the same as its Krull dimension.)

We will show that several natural invariants of
schemes, such as \'etale cohomology with torsion coefficients and perfect
complexes on perfect $\mathbb{F}_p$-schemes,   
satisfy descent for the $\arc$-topology. In these cases $v$-descent was previously known. 

The seemingly slight strengthening of topologies (from $v$- to $\arc$-)
turns out to have concrete consequences: 
namely, $\arc$-descent additionally forces several excision-type squares.
For instance, $\arc$-sheaves  satisfy ``excision'' in the
classical $K$-theoretic sense as well as an analog of the Beauville--Laszlo ``formal glueing'' theorem \cite{BeauvilleLaszlo}.
As applications, using the $\arc$-topology, we obtain relatively soft proofs of some classical results in \'etale
cohomology, such as the Gabber--Huber affine analog of proper base change and the
Fujiwara--Gabber theorem on the \'etale cohomology of punctured henselian pairs.
We also prove new general results including a version of Artin--Grothendieck vanishing in rigid geometry (which improves on recent work of Hansen \cite{HansenArtin}).

\subsection{The $\arc$-topology}

The starting point for us is the so-called \emph{$v$-topology} or
\emph{universally subtrusive topology}, studied by \cite{Rydh, BhattScholzeWitt}: 
\begin{definition}[The $v$-topology] 
\label{vcover}
\begin{enumerate}
\item An \emph{extension} of valuation rings is a faithfully flat map $V \to W$ of
valuation rings (equivalently, an injective local homomorphism).

\item A map of quasi-compact and quasi-separated (qcqs) schemes $Y \to X$ is called a \emph{$v$-cover} if 
for any valuation ring $V$ and any map $\mathrm{Spec} (V) \to X$, there is an extension of
valuation rings $V \to W$ and a map $\mathrm{Spec} (W) \to Y$ that fits into a commutative
square
\begin{equation}  
\label{commsquarev} \xymatrix{
\mathrm{Spec} (W) \ar[d] \ar[r] &   Y \ar[d]  \\
\mathrm{Spec} (V) \ar[r] &  X.}
\end{equation}

\item The {\em $v$-topology} on the category of schemes is the Grothendieck
topology where the covering families $\{f_i \colon Y_i \to X\}_{i \in I}$ are those
families with the following property: for any affine open $V \subset X$, there
exists a map $t \colon K \to I$ of sets with $K$ finite and affine opens $U_k \subset f_{t(k)}^{-1}(V)$ for each $k \in K$ such that the induced map $\sqcup_k U_k \to V$ is a $v$-cover in the sense of (2).
\end{enumerate}
\end{definition}

For finite type maps of noetherian schemes, the $v$-topology 
coincides with Voevodsky's $h$-topology \cite[Sec. 3]{VHomology},
i.e., the one generated by \'etale covers and proper surjections. 
In general, every $v$-cover is a limit of $h$-covers. 
In this paper we study the following definition, which has also been explored by Rydh in the forthcoming work \cite{RydhII}. 

\begin{definition}[The $\arc$-topology]
\label{DefArcTop}
\begin{enumerate}
\item A map $f\colon Y \to X$ of qcqs schemes is an \emph{$\arc$-cover}\footnote{The name ``$\arc$'' was chosen in view of the natural analogy between rank $1$ valuations and arcs. Thus, an $\arc$ cover is a map of schemes along which every arc lifts.} if for any rank
$\leq 1$ valuation ring $V$ and a map $\mathrm{Spec}(V) \to X$, there is an extension $V \to
W$ of rank $\leq 1$ valuation rings and a map $\mathrm{Spec}(W) \to Y$ lifting the
composition $\mathrm{Spec}(W) \to \mathrm{Spec}(V) \to X$ to a commutative square as in \eqref{commsquarev}. 

\item The $\arc$-topology on the category of all schemes is defined analogously to Definition~\ref{vcover} (3).
\end{enumerate}
\end{definition}

We shall show (Proposition~\ref{UnivSubArc}) that the $\arc$-topology coincides with a finer variant of the topology of universal submersions (or universal topological quotient maps) of \cite{Pic, Rydh} that behaves better under filtered inverse limits.  This result had been independently observed by Rydh in the forthcoming  \cite{RydhII}. 

For noetherian targets, there is no distinction between $v$-covers and $\arc$-covers (Proposition~\ref{Noetherianv}). On the other hand, they do not coincide in general.
For the purposes of this paper,
the fundamental example  capturing this discrepancy is the following. 

\begin{example}
\label{ExValRingArc}
Let $V$ be a valuation ring of rank $2$. If $\mathfrak{p} \subset V$ denotes the
unique height $1$ prime, then both $V_{\mathfrak{p}}$ and $V/\mathfrak{p}$ are
rank $1$ valuation rings, and the map $V \to V_{\mathfrak{p}} \times
V/\mathfrak{p}$ is an $\arc$-cover (\Cref{v1val}) but not a $v$-cover. In fact, if $f \in V - \mathfrak{p}$ is not invertible, then $V \to V_f \times V/fV$ is a finitely presented $\arc$-cover that is not a $v$-cover.
\end{example}

The existence of this example illustrates one of the subtleties in working with the $\arc$-topology. Namely, even though every $\arc$-cover is a limit of finitely presented $\arc$-covers, a finitely presented $\arc$-cover cannot be obtained as a base change of an $\arc$-cover of noetherian schemes. In particular, noetherian approximation arguments do not work as well as they do in the $v$-topology. 

The main goals of this paper are two fold: we show that certain naturally defined functors are $\arc$-sheaves (as summarized in \S \ref{ss:InstantArcDesc}), and we exhibit some remarkably nice structural properties of any $\arc$-sheaves (as summarized in \S \ref{ss:ConseqArcDesct}).

\begin{remark}[Restriction to qcqs schemes]
A contravariant functor $F$ on the category $\mathrm{Sch}$ of all schemes that is a sheaf for the Zariski topology is automatically determined by its restriction to the subcategory $\mathrm{Sch}_{qcqs}$ of qcqs schemes. Conversely, any Zariski sheaf on $\mathrm{Sch}_{qcqs}$ comes from a unique Zariski sheaf on $\mathrm{Sch}$. As the $\arc$-topology is finer than the Zariski topology, and because all interesting functors that we consider are Zariski sheaves for essentially formal reasons, we typically restrict attention to qcqs schemes in the rest of the paper. This simplifies and shortens the exposition as the $\arc$-topology on $\mathrm{Sch}_{qcqs}$ is finitary (in the sense of \cite[\S A.3.1]{SAG}), making it slightly faster to check the sheaf property on $\mathrm{Sch}_{qcqs}$ (see Definition~\ref{DefDescent}). 
\end{remark}

\subsection{Instances of $\arc$-descent}
\label{ss:InstantArcDesc}

Given a Grothendieck topology $\tau$, we can ask when a functor satisfies descent with
respect to it (or equivalently is a sheaf). Let us spell this out next in an $\infty$-categorical context\footnote{The theory of
$\infty$-categories does not play a crucial role in this paper. However, it is
convenient to use this language to formulate clean statements. Our main example
of a target $\infty$-category $\mathcal{C}$ shall be the derived
$\infty$-category $\D(\Lambda)$ of a ring $\Lambda$ (or variants), though we
occasionally also use the $\infty$-category $\mathrm{Cat}_\infty$ of all
$\infty$-categories when discussing ``stacky'' phenomenon.} for presheaves on $\mathrm{Sch}_{qcqs}$. 
We will only consider Grothendieck topologies on 
$\mathrm{Sch}_{qcqs}$ which are finitary 
(i.e., every cover admits a finite subcover) and such that if $X, Y \in \mathrm{Sch}_{qcqs}$, then $\left\{X, Y \to X
\sqcup Y\right\}$ forms a covering family, cf.~\cite[Sec.~A.3.2 and A.3.3]{SAG}. 
\begin{definition}
\label{DefDescent}
Let $F\colon  \mathrm{Sch}_{qcqs}^{op} \to \mathcal{C}$ be a presheaf valued in an $\infty$-category $\mathcal{C}$.  We will say that $F$ satisfies \emph{descent} for a morphism $Y \to X$ of qcqs schemes if it satisfies the $\infty$-categorical sheaf axiom with respect to $Y \to X$, i.e., if the natural map 
\[F(X) \to \varprojlim ( F(Y) \rightrightarrows F(Y \times_X Y) \triplearrows
\dots ).\]
is an equivalence. If this property holds for all maps $f\colon Y \to X$
that are covers for a Grothendieck topology $\tau$ (satisfying the
finiteness conditions
above)
and further if $F$ carries finite disjoint unions to finite products, then we
say that $F$ {\em satisfies $\tau$-descent} or is a {\em $\tau$-sheaf.}
\end{definition}

It is relatively straightforward to see that notion of a $\tau$-sheaf in Definition~\ref{DefDescent} is equivalent to the notion of a $\mathcal{C}$-valued sheaf with respect to $\tau$ (see \cite[Proposition A.3.3.1]{SAG}). An important source of examples for our purposes is the following:

\begin{example}
Let $\tau$ be a Grothendieck topology on the category of schemes (of suitably
bounded cardinality if $\tau$ is ``large''). For any
scheme $X$ and a coefficient ring $\Lambda$, write $F(X) := R\Gamma(X_{\tau},
\Lambda)$ for the $\tau$-cohomology of $X$, viewed as an object of the derived
$\infty$-category $\D(\Lambda)$. The resulting functor $F\colon \mathrm{Sch}_{qcqs}^{op} \to
\D(\Lambda)$ is a $\tau$-sheaf. 
\end{example}

In this paper, we will be interested in invariants which satisfy $\arc$-descent. As $v$-covers are $\arc$-covers, it is clear that $\arc$-sheaves are $v$-sheaves. Conversely, we prove the following general criterion for a $v$-sheaf to be an $\arc$-sheaf; roughly speaking, it says that a $v$-sheaf (satisfying a mild finite presentation constraint) which also satisfies descent with respect to the covers from Example~\ref{ExValRingArc} (and slight variants) is automatically an $\arc$-sheaf. 
The idea of studying functors which satisfy the excisiveness condition in
\Cref{ExValRingArc}, and recovering them from valuation rings of rank $\leq 1$,
also appears in the work of Huber--Kelly \cite{HK18}. 

\begin{theorem}[Criteria for $\arc$-descent, Theorem~\ref{mainthm}] 
Let $F\colon  \mathrm{Sch}_{qcqs}^{op} \to \D(\mathbb{Z})^{\geq 0}$ be a functor on qcqs
schemes which is finitary\footnote{For sheaves of sets, this property has been dubbed ``local finite presentation''  by Grothendieck and is pervasive in the classical literature \cite[Tag 049J]{stacks-project}; however, we avoid this terminology to avoid clashes with other similarly named notions, such as that of finitely presented objects in an $\infty$-category.}, i.e. $F$ takes filtered limits with affine transition maps into filtered
colimits. Then the following are equivalent: 
\begin{enumerate}
\item $F$ is an $\arc$-sheaf.  
\item $F$ is a $v$-sheaf and for every valuation
ring $V$ with algebraically closed fraction field and prime ideal $\mathfrak{p} \subset V$, the square
\[ \xymatrix{ F( \spec(V)) \ar[d] \ar[r] &  F( \spec( V/\mathfrak{p})) \ar[d]
\\
F( \spec( V_{\mathfrak{p}})) \ar[r] &  F( \spec(\kappa( \mathfrak{p})))
}\]
is cartesian. 
\end{enumerate}
\end{theorem} 

As an application, we give some examples of $\arc$-sheaves. Our first example
is \'etale cohomology with torsion coefficients; note that this invariant was
known to be a $v$-sheaf, essentially because of the proper base change theorem,
cf.~Proposition~\ref{vDescentEtaleCoh}. More generally, we show that constructible complexes on a scheme $X$ can be constructed $\arc$-locally.

\begin{theorem}[$\arc$-descent for \'etale cohomology,
Theorems~\ref{EtaleCohExcisive} and \ref{arcConsDer}] 
\label{v1descent}
For a finite ring $\Lambda$, the assignment $X \mapsto
\D^b_{\mathrm{cons}}(X, \Lambda)$ sending $X$ to the bounded derived
$\infty$-category of sheaves of $\Lambda$-modules on $X$ with constructible
cohomology satisfies $\arc$-descent.  In particular, the functor $X \mapsto R \Gamma(X_{\et}, \Lambda)$ is an $\arc$-sheaf.
\end{theorem} 

Our second example is that of perfect complexes on perfect schemes of
characteristic $p$; again, this invariant was already known to be a $v$-sheaf
\cite[Sec.~11]{BhattScholzeWitt}.

\begin{theorem}[$\arc$-descent for perfect complexes, Theorem~\ref{archyp}] 
On qcqs $\mathbb{F}_p$-schemes, the functor $X \mapsto \mathrm{Perf}( X_{\mathrm{perf}})$
satisfies $\arc$-descent. 
\end{theorem}

As an application, we obtain a completely categorical description of $\arc$-covers on perfect schemes.

\begin{theorem}[Characterization of $\arc$-covers of perfect schemes, Theorem~\ref{PerfRepArcSheaf}] 
On qcqs perfect $\mathbb{F}_p$-schemes, a map $Y \to X$ is an $\arc$-cover if
and only if it is a universally effective epimorphism \cite[Tag
00WP]{stacks-project} in the category of qcqs perfect $\mathbb{F}_p$-schemes. 
\end{theorem}

\subsection{Consequences of $\arc$-descent}
\label{ss:ConseqArcDesct}

It is sometimes easy to show (thanks largely to the relatively simple nature of
rank $\leq 1$ valuation rings) that certain squares of schemes that ``look
like'' they ought to be pushouts do in fact yield pushout squares on associated
$\arc$-sheaves. In particular, any $\arc$-sheaf carries such a square of
schemes to a pullback square. This leads to concrete consequences for
$\arc$-sheaves, such as excision squares or Mayer--Vietoris sequences, which we describe next.

Let us begin with the classical formulation for excision for functors on rings.

\begin{definition}
\label{excdatum}
An {\em excision datum} is given by a map $f\colon (A,I) \to (B,J)$ where $A$ and $B$ are commutative rings, $I \subset A$ and $J \subset B$ are ideals, and $f\colon A \to B$ is a map that carries $I \subset A$ isomorphically onto $J \subset B$. In this situation, we obtain a commutative square 
of rings
\begin{equation} \label{milnorsquare} \xymatrix{ A \ar[r] \ar[d] & A/I \ar[d] \\
		   B \ar[r] & B/J } \end{equation}
that is both cocartesian and cartesian.  Such diagrams are also called \emph{Milnor
squares} (after \cite[\S 2]{Milnor}).  We say that a $\D(\mathbb{Z})$-valued
functor $F$ on commutative rings is \emph{excisive} if for any excision datum as above,
the square obtained by applying $F$ to \eqref{milnorsquare} is  cartesian;
equivalently, the natural map from the fiber of $F(A) \to F(A/I)$ to the fiber of $F(B) \to F(B/J)$ is an equivalence.
\end{definition}

The question of excision has played a crucial role in algebraic $K$-theory, see
for instance \cite{SW, Cortinas, GHexc, LT19}. We prove the following result relating excisiveness to the $\arc$-topology.

\begin{theorem}[$\arc$-sheaves satisfy excision, Theorem~\ref{mainthm} and
Corollary~\ref{ExcArcSheafPull}] 
\label{ArcExcIntro}
Let $\mathcal{C}$ be an $\infty$-category that has all small limits. 
Any $\arc$-sheaf $F\colon \mathrm{Sch}_{qcqs}^{op} \to \mathcal{C}$  satisfies excision. 
\end{theorem} 

\begin{remark}
\label{ArcNeccExc}
The hypothesis that $F$ be an $\arc$-sheaf, and not just a $v$-sheaf, is essential to Theorem~\ref{ArcExcIntro}. Indeed, general $v$-sheaves can fail to be excisive. This distinction is explained by the following observation (which was also the basis of our discovery of the $\arc$-topology):  if $(A,I) \to (B,J)$ is an excision datum, then $\spec(B) \sqcup \spec(A/I) \to \spec(A)$ is always an $\arc$-cover (Lemma~\ref{excisionimpliesvleq1}) but not in general a $v$-cover; in fact, with notation as in Example~\ref{ExValRingArc}, the map $(V,\mathfrak{p}) \to (V_{\mathfrak{p}}, \mathfrak{p} V_{\mathfrak{p}})$ gives an example.
\end{remark}

Next, we study ``formal glueing'' results for $\arc$-sheaves. The formal glueing property of a functor captures whether its value on a variety can be reconstructed from its value on a formal neighbourhood of a subvariety and its value on the complement. We formulate this precisely as follows.

\begin{definition}
\label{DefFormalGlueing}
A {\em formal glueing datum} is given by a pair $(R \to S, I)$ where $R \to S$
is a map of commutative rings and $I \subset R$ is a finitely generated ideal
such that $R/I^n \simeq S/I^n S$ for all $n \geq 0$. We consider the corresponding square
\begin{equation}
\label{MVsquare}
\xymatrix{ \spec(S) \setminus V(IS) \ar[r] \ar[d] & \spec(S) \ar[d] \\
		\spec(R) \setminus V(I) \ar[r] & \spec(R) }
		\end{equation}
of schemes. We say that a contravariant functor $F$ on schemes satisfies {\em formal glueing} if for every formal glueing datum $(R \to S, I)$ as above, the functor $F$ carries the square \eqref{MVsquare} to a cartesian square. 
\end{definition}

This property has been studied frequently in algebraic geometry. For example, the functor that assigns to a scheme its category of vector bundles satisfies formal glueing (at least when restricted to noetherian schemes, see \cite[Tag 05E5]{stacks-project} and the references therein). We prove the same holds for $\arc$-sheaves without any noetherianness constraints:

\begin{theorem}[Formal glueing squares for $\arc$-sheaves, Theorem~\ref{arcarithm}
 and 
Corollary~\ref{FormalGlueArcSheafGeneral}]
\label{FormalGlueing}
Let $\mathcal{C}$ be an $\infty$-category that has all small limits. 
Any $\arc$-sheaf $F\colon \mathrm{Sch}_{qcqs}^{op} \to \mathcal{C}$ satisfies formal glueing.
\end{theorem}

\begin{remark}
Just as in Theorem~\ref{ArcExcIntro}, the hypothesis that $F$ should be an
$\arc$-sheaf, and not merely a $v$-sheaf, is essential
for Theorem~\ref{FormalGlueing}. In fact, just as in Remark~\ref{ArcNeccExc},
this can be explained by the following observation: if $(R \to S, I)$ is a
formal glueing datum, then the corresponding map $f\colon \spec(S) \sqcup
(\spec(R) \setminus V(I)) \to \spec(R)$ is an $\arc$-cover (Proposition~\ref{completeisarccover}) but not in general a $v$-cover; see Example~\ref{ExFG} to see what can go wrong.
\end{remark}

Specializing the previous theorems to \'etale cohomology, we obtain:

\begin{corollary} 
\label{etaleexc}
For any torsion abelian group $\Lambda$, the functor $X \mapsto R\Gamma(X_{\et}, \Lambda)$ is excisive (on rings, in the sense of Definition~\ref{excdatum}) and satisfies formal glueing (in the sense of Definition~\ref{DefFormalGlueing}).
\end{corollary}

As applications, we give quick proofs of two foundational results in the
\'etale cohomology of rings and schemes: the Gabber--Huber affine analog of
proper base change\footnote{At its core, our proof has some similarity with Gabber's proof:
both proofs involve reduction to the absolutely integrally closed case. Once
this reduction is made, Gabber proceeds to study the topology of closed sets
inside absolutely integrally closed integral schemes directly. However, in our
proof, we actually reduce to absolutely integrally closed valuation rings, where
there is no nontrivial topology at all.}, at least when working over a henselian
local ring, and the Fuijwara-Gabber theorem, generalized to the non-noetherian setting.

\begin{corollary}
\label{GabberLocal}
Let $R$ be a commutative ring that is henselian with respect to an ideal $I \subset R$. Let $\mathcal{F}$ be a torsion \'etale sheaf on $\spec(R)$, viewed as a sheaf on all schemes over $\spec(R)$ via pullback. 
\begin{enumerate}
\item (Gabber \cite{Gabber}, Huber \cite{Huber}, \S \ref{ss:EtaleCohArcSheaf} below) Assume that $R$ admits
the structure of an algebra
over a  
henselian local ring.  Then\footnote{Note that \cite{Gabber, Huber} prove the
result without the assumption that $R$ is an algebra over a henselian local
ring.}
$R\Gamma(\mathrm{Spec}(R)_{\et},\mathcal{F}) \xrightarrow{\sim}
R\Gamma(\mathrm{Spec}(R/I)_{\et},\mathcal{F})$.
\item (Fujiwara--Gabber \cite{Fujiwara}, Corollary~\ref{GabberFujiwara} below)  Assume that $I$ is finitely generated. Then 
\[ R \Gamma( (\spec(R) \setminus V(I))_{\et}, \mathcal{F}) \xrightarrow{\sim}  R \Gamma(
(\spec ( \hat{R}_I) \setminus V( I \hat{R}_I))_{\et},  \mathcal{F}) .\]
\end{enumerate}
\end{corollary}

Unlike the proof of Corollary~\ref{GabberLocal} (2) in \cite{Fujiwara}, our proof is relatively soft and does not rely on Elkik's approximation theorem or its variants. Similar techniques also allow us to prove a descent property for the \'etale cohomology (Corollary~\ref{BerkDescent}) that essentially amounts to the equality of the algebraic and analytic \'etale cohomology of affinoid spaces in rigid analytic geometry (which, under noetherian hypotheses, is Huber's affinoid comparison theorem \cite[Corollary 3.2.2]{HuberBook} for constant coefficients and Hansen's \cite[Theorem 1.9]{HansenArtin} for general coefficients).

As a final application, we apply these descent-theoretic techniques to sharpen
recent results of Hansen \cite{HansenArtin} and prove the following version of the
classical Artin--Grothendieck vanishing theorem for rigid geometry. 
In {\em loc.\ cit.}, this is proved when 
$K$ has characteristic zero 
and when $A$ arises as the base change of an affinoid algebra over a discretely
valued nonarchimedean field. 

\begin{theorem}[Artin--Grothendieck vanishing for the cohomology of affinoids, Theorem~\ref{rigidartin}]
\label{rigidartinintro}
Fix a complete and algebraically closed nonarchimedean field $K$. Let $A$ be a classical affinoid $K$-algebra of dimension $d$.  Fix a prime $\ell$ and let $\sF$ be an $\ell$-power torsion \'etale sheaf on $\spec(A)$. If $A$ is smooth or if $\ell$ is not the residue characteristic of $K$, then $H^i( \spec(A)_{\et}, \sF) = 0$ for $i > d$. In general, we at least have $H^i(\spec(A)_{\et},\sF) = 0$ for $i > d+1$.
\end{theorem} 

Our proof of Theorem~\ref{rigidartinintro} is independent of Hansen's work \cite{HansenArtin}. Moreover, with the exception of the affinoid comparison theorem mentioned above, our arguments can be formulated purely in terms of the \'etale cohomology of rigid spaces (though we do not do so in our expositon).

\begin{convention*}
For a commutative ring $R$, let $\mathrm{Ring}_R$
(resp.~$\mathrm{Sch}_{qcqs,R}$) denote the category of commutative $R$-algebras
(resp.~qcqs $R$-schemes). We simply write $\mathrm{Sch}_{qcqs} := \mathrm{Sch}_{qcqs,\mathbb{Z}}$ for the category of all qcqs schemes. Given a presheaf $F$ on $\mathrm{Sch}_{qcqs,R}$, we often write $F(A) = F(\spec(A))$ for an $R$-algebra $A$ if there is no confusion.

For a torsion \'etale sheaf $\mathcal{F}$ on a scheme $X$ and a morphism $f\colon Y \to X$, we often write $R\Gamma(Y,\mathcal{F})$ as shorthand for $R\Gamma(Y_{\et}, f^* \mathcal{F})$ if there is no confusion. In particular, cohomology with constant coefficients is always computed with respect to the \'etale topology unless otherwise specified.

For a spectral space $S$ (such as the space underlying a qcqs scheme), we write
$S^{cons}$ for the set $S$ equipped with the constructible topology inherited
from the spectral topology on $S$ (so the open subsets of $S^{cons}$ are exactly
the ind-constructible subsets of $S$); recall that $S^{cons}$ is a profinite set
whose clopen subsets are exactly the constructible subsets of $S$. We refer to 
\cite[Tag 08YF]{stacks-project} for an account of the theory of spectral spaces
and the constructible topology, and \cite[Sec.~1.9]{EGAIV} for an account of the
constructible topology in the
case of schemes (for instance, cf.~\cite[Prop.~1.9.11]{EGAIV} for the characterization of
clopen subsets). 

Given a qcqs scheme $X$ and a specialization $x \rightsquigarrow y$ of points in
$X$, we say that this specialization is {\em witnessed by} a map $f\colon
\mathrm{Spec}(V) \to X$ if $V$ is a valuation ring, and $f$ carries the generic
point (resp.~the closed point) to $x$ (resp.~$y$); every specialization can be witnessed by a map from the spectrum of a valuation ring \cite[Tag 00IA]{stacks-project}.

Following terminology from the theory of adic spaces, we shall refer to a nonzero nonunit in a rank $\leq 1$ valuation ring $V$ as {\em pseudouniformizer}.

Given a commutative ring $R$ and a finitely generated ideal $I \subset R$, we
denote by $\hat{R}_I$ the $I$-adic completion of $R$; note that this is
well-behaved even when $R$ is non-noetherian, \cite[Tag 00M9]{stacks-project}. 

We let $\mathcal{S}$ denote the $\infty$-category of spaces, and
$\mathcal{S}_{\leq n}$ denote the full subcategory of $n$-truncated spaces.
\end{convention*}

\subsection*{Acknowledgments} 
The first author was supported by NSF grants \#1501461 and \#1801689 as well as grants from the Packard and Simons Foundations. This work was done while the second author was a
Clay Research Fellow. The second author would also like to thank the University of Michigan
at Ann Arbor for its hospitality during a weeklong visit where this work was
started. We thank Piotr Achinger, Benjamin Antieau, Dustin Clausen, David
Hansen, Marc Hoyois, Johan de Jong,
Matthew Morrow, David Rydh, and Peter Scholze for helpful conversations. We also
thank Joseph Ayoub, K\k{e}stutis \v{C}esnavi\v{c}ius, Shane Kelly, and Jacob
Lurie for comments on an earlier version of this
paper. 
Finally, we heartily thank the referees for a very careful reading of the
manuscript and for many helpful comments and
corrections. 

\newpage
\section{The $\arc$-topology}

The goal of this section is twofold: in \S \ref{ArcProperties}, we collect some basic properties and examples of the $\arc$-topology, while in \S \ref{SpecSubm}, we explain why the $\arc$-topology can be defined topologically via the notion of ``universal spectral submersions''. 

\subsection{Properties of $\arc$-covers}
\label{ArcProperties}

In this subsection, we collect various results on the $\arc$-topology. We will work freely with the theory of valuation rings; see  \cite[Tag 0018, Tag 0ASF]{stacks-project} or \cite[Ch. 6]{BourbakiCA} for an introduction. In particular, we use the basic properties that valuation rings are closed under localization and taking quotients by prime ideals, as well as the notion of a rank $1$ valuation. 

\begin{proposition} 
\label{visv1}
Let $f\colon  Y \to X$ be a $v$-cover of qcqs schemes (\Cref{vcover}).   Then $f$ is an $\arc$-cover. 
\end{proposition} 
\begin{proof} 
Let $V$ be a rank $\leq 1$ valuation ring 
with a map $\mathrm{Spec} (V) \to X$; by assumption there exists a 
valuation ring $W$ with 
a faithfully flat map $g\colon  V \to W$ and a map $\mathrm{Spec} (W) \to Y$ making the diagram
\eqref{commsquarev} commutative. 
The only issue is that $W$ need not be rank $\leq 1$.  To remedy this, we  proceed  as follows. 

Let $\mathfrak{m}$ be the maximal ideal of $V$. The collection of prime ideals in $W$ is totally ordered, so any
intersection of prime ideals remains prime; thus, there exists a minimal prime
$\mathfrak{q} \subset W$ such that $g^{-1}(\mathfrak{q}) = \mathfrak{m}$. 
Similarly, there exists a maximal prime ideal $\mathfrak{q}_0 \subset W$ such
that $g^{-1}( \mathfrak{q}_0) = 0$.  The map $V \to W \to W'
\stackrel{\mathrm{def}}{=} 
(W/\mathfrak{q}_0)_{\mathfrak{q}}$ is faithfully flat because it is an injective
local homomorphism, and $W'$ is a 
rank $\leq 1$ valuation ring. 
In view of the map $\mathrm{Spec} (W') \to \mathrm{Spec} (W) \to Y$, we
see that $f$ is an $\arc$-cover, as desired. 
\end{proof} 

\begin{remark}
The primes $\mathfrak{q}$ and $\mathfrak{q}_0$ appearing in the preceding proof
can be described more explicitly as follows. Recall that for any valuation ring
$W$ equipped with an element $f \in W$ which is not a unit, the ideal $\mathfrak{q} := \sqrt{fW}$ is the minimal prime ideal containing $f$, and the ideal $\mathfrak{q}_0 := \cap_n f^n W$ is the maximal prime ideal contained in $fW$; one proves this using the primality of radical ideals in valuation rings. In the situation of the proof above, one simply applies this to $f \in W$ being the image of any element $t \in V$ such that $\sqrt{tV}$ is the maximal ideal.

In the preceding construction, if $f \in W$ is nonzero and not invertible, then the resulting specialization $\mathfrak{q}_0 \rightsquigarrow \mathfrak{q}$ is an immediate one, i.e., $(W/\mathfrak{q}_0)_{\mathfrak{q}}$ has rank $1$ (and $f$ gives a pseudouniformizer in this valuation ring). Consequently, for any non-trivial specialization $\mathfrak{p} \rightsquigarrow \mathfrak{p}'$ in a valuation ring $W$, there exist specializations $\mathfrak{p} \rightsquigarrow \mathfrak{q}_0 \rightsquigarrow \mathfrak{q} \rightsquigarrow \mathfrak{p}'$ where the middle one is an immediate specialization realized by applying the construction in the previous paragraph to some $f \in \mathfrak{p}' - \mathfrak{p}$. In particular, the totally ordered sets arising as spectra of valuation rings cannot be arbitrary totally ordered sets admitting minimal and maximal elements. 
This observation (for all commutative rings more generally) already appears in \cite[Sec. 1-1, Theorem 11] {Kaplansky}. 
\end{remark} 

\begin{lemma} 
\label{compositemaparccover}
Consider a composite map $Y' \stackrel{g}{\to} Y \stackrel{f}{\to} X$ of
qcqs schemes.  If $f \circ g$ is an $\arc$-cover, then so is $f$. 
\end{lemma} 
\begin{proof} 
Clear. 
\end{proof}

\begin{lemma} 
\label{arccovercrit}
Let $(V, \mathfrak{m})$ be a rank $1$ valuation ring and let $A$ be a
$V$-algebra with the
structure map $f\colon  V \to A$. 
Then the following are equivalent: 
\begin{enumerate}
\item $\mathrm{Spec} (A) \to \mathrm{Spec} (V) $ is a $v$-cover.  
\item $\mathrm{Spec} (A) \to \mathrm{Spec} (V) $ is an $\arc$-cover.
\item 
There exists an inclusion $\mathfrak{p}_1 \subset \mathfrak{p}_2 \subset A$ of prime ideals with $0 = f^{-1}( \mathfrak{p}_1)$  and $f^{-1}(\mathfrak{p}_2) = \mathfrak{m}$. That is, the specialization $(0) \rightsquigarrow \mathfrak{m}$ in $\mathrm{Spec}(V)$ lifts to a specialization $\mathfrak{p}_1 \rightsquigarrow \mathfrak{p}_2$ in $\mathrm{Spec}(A)$.
\end{enumerate}
\end{lemma} 
\begin{proof} 
(1) implies (2) in view of Proposition~\ref{visv1}. If $\mathrm{Spec} (A )\to
\mathrm{Spec} (V)$ is an
$\arc$-cover, then there exists a rank $1$ valuation ring $W$, faithfully
flat over $V$, such that the map $\mathrm{Spec} (W) \to \mathrm{Spec} (V)$ factors through $\mathrm{Spec}
(A)$. 
The images of the generic and special points in $\mathrm{Spec} (A)$ then verify (3).
Finally, (3) implies (1), 
we use \cite[Prop.~2.7]{Rydh}. 
\end{proof} 

\begin{corollary} 
\label{v1detectionr1}
Let $f\colon  Y \to X$ be a map of qcqs schemes. Then $f$ is an $\arc$-cover if and only if every base change of $f$ along a map $\mathrm{Spec}(V) \to X$, for $V$ a rank $\leq 1$
valuation ring, is a $v$-cover. \qed
\end{corollary} 

In the noetherian setting, there is no distinction between the $v$- and
$\arc$-topologies; we record this result here, though it is not essential to
the sequel. In particular, the results of \cite{Rydh} show that for a map of
qcqs schemes with the target noetherian, one can check whether it is a
$v$-cover after base change to discrete (in particular, rank $\leq 1$) valuation rings. 

\begin{proposition}[{\cite[Theorem 2.8]{Rydh}}]
\label{Noetherianv}
Let $f\colon  Y \to X$ be a map of qcqs schemes with $X$ noetherian. Then $f$ is
a $v$-cover if and only if it is an $\arc$-cover.  \qed
\end{proposition} 

To deduce excision, we must necessarily work with non-noetherian rings. In this setting, the $\arc$-topology differs from the $v$-topology. To see this, let us first explain how to extract $\arc$-covers from excision data; we shall later see that these covers are typically not $v$-covers.

\begin{lemma} 
\label{excisionimpliesvleq1}
Consider an excision datum $f\colon  (A, I)  \to (B, J)$ as in
Definition~\ref{excdatum}. 
Then the map $A \to A/I \times B$ is an $\arc$-cover.
More precisely, any map from $A$ to a rank $\leq 1$ valuation ring factors
through $A/I \times B$. 
\end{lemma} 
\begin{proof} 
Let $W$ be a rank $\leq 1$ valuation ring and consider a map $g\colon  A \to W$.
If $g(I) = 0$, then $g$ factors over $A \to A/I$. It thus suffices to assume
that $g(I) \neq 0$, and show that $g$ factors over $A \to B$. Choose $x \in I$
such that $g(x) \neq 0 \in W$. Inverting $x$ yields a map $A[1/x]
\xrightarrow{g[1/x]} W[1/g(x)] \subset K := \mathrm{Frac}(W)$. As $A[1/x] \to
B[1/x]$ is an isomorphism thanks to the pullback expression $A \simeq B \times_{B/J}
A/I$, we obtain a map $B \to B[1/x] \to K$. Let $W'$ be the $W$-subalgebra of $K$ generated by the image of $B$, i.e., $W' = \mathrm{im}(W \otimes_A B \to K)$. But $x$ kills the cokernel of $A \to B$: this cokernel is identified with that of $A/I \to B/J$ (by hypothesis) which is trivially killed by $x \in I$. This implies $g(x) \cdot W' \subset W$. But then $W = W'$ as any subring between $W$ and $K$ with bounded denominators is necessarily equal to $W$ as $W$ has rank $\leq 1$. The resulting map $A \to B \to W' \simeq W$ is the desired factorization of $g$.
\end{proof}

We now give the basic example (for this note) of an excision datum that will
yield an $\arc$-cover. 

\begin{proposition}[Excision data attached to valuation rings]
\label{AICValRingPrime}

Let $V$ be a valuation ring. Fix  $\mathfrak{p} \in \mathrm{Spec}(V)$. 
Fix an inclusion $\mathfrak{p}\subset \mathfrak{q}$ of prime ideals in $V$. Then $(V_{\mathfrak{q}},\mathfrak{p}V_{\mathfrak{q}}) \to (V_{\mathfrak{p}},\mathfrak{p}V_{\mathfrak{p}})$ is an excision datum.

\end{proposition}
In particular, the map $(V, \mathfrak{p}) \to (V_{\mathfrak{p}}, \mathfrak{p}
V_{\mathfrak{p}})$ is an excision datum.

\begin{proof}
We may replace $V$ with $V_{\mathfrak{q}}$ to assume $\mathfrak{q}$
is the maximal ideal. 
That is, 
we need to show that $\mathfrak{p} \simeq \mathfrak{p} V_{\mathfrak{p}}$. The
map is clearly an injection (both sides are contained in the quotient field of
$V$). Conversely, given a fraction $a/s$ with $a \in \mathfrak{p}$ and $s \in V
\setminus \mathfrak{p}$, we have necessarily $s \mid a$ and $a =sb$ for $b \in
\mathfrak{p}$. Thus, $a/s = b\in \mathfrak{p}$ as desired. 
\end{proof}

\begin{corollary}
\label{v1val}
Let $V$ be a valuation ring, and let $\mathfrak{p} \subset V$ be a prime ideal.
Then $V \to V_{\mathfrak{p}} \times V/\mathfrak{p}$ is an $\arc$-cover. \qed
\end{corollary}

Thus, the $\arc$-topology is strictly finer than the $v$-topology: if $V$ has rank $\geq 2$ and $\mathfrak{p}$ is a nonzero nonmaximal prime, then the map from Corollary~\ref{v1val} is clearly not a $v$-cover.

\begin{remark}
Using the easier ``only if'' part of Proposition~\ref{UnivSubArc} below, one may also deduce Corollary~\ref{v1val} from \cite[Corollary 33]{Pic} (see also \cite[Ex. 4.5]{GooLi}, and \cite[Ex. 4.3]{Rydh}).
\end{remark}

Finally, we observe that the condition of being an $\arc$-cover is essentially
module-theoretic. 
\begin{proposition} 
Let $f\colon  \spec(B) \to \spec(A)$ be a map of affine schemes.   Then the following are
equivalent: 
\begin{enumerate}
\item  $f$ is an $\arc$-cover. 
\item The map $A \to B$ has the property that after
every base change along a map $A \to V$ with $V$ a rank $\leq 1$ valuation ring, 
it is universally injective \cite[Tag 058I]{stacks-project}. 
\end{enumerate}
\end{proposition} 
\begin{proof} 
We immediately reduce to the case where $A = V$ is a rank $1$ valuation
ring with a pseudouniformizer $t$. 
Suppose that $A \to B$ is universally injective. Let $B'$ be the quotient of $B$ by the ideal of $t$-power torsion elements.  
We claim that the map $A \to B \to B'$ is faithfully flat, which is
sufficient to prove that $A \to B$ is a $v$-cover (cf.~\cite[Rem.~2.5]{Rydh}). 
Since $B'$ is flat over $A$ as it is $t$-torsionfree, it suffices to show that 
$B'[1/t]$ and $B'/(t)$ are both nonzero. First, $B'[1/t] = B[1/t] \neq 0$. 
Second, if $B'/t = 0$, then we can write in $1 = tx + y$ in $B$ for $x,y \in B$ with $y$ being $t$-power torsion. Multiplying by some $t^m$ for $m
\gg 0$, we get the equality $t^{m} = t^{m+1} x$  in $B$; however, this contradicts 
our assumption that the map  $A/t^{m+1} \to B/t^{m+1}$ is injective. 

Conversely, 
if $A \to B$ is an $\arc$-cover with $A$ a rank $1$ valuation ring, choose an inclusion of prime ideals
$\mathfrak{p}_1 \subset \mathfrak{p}_2 \subset B$ which pull back to the zero
and maximal ideal of $A$. Then $B/\mathfrak{p}_1$ is an integral domain which
is faithfully flat over $A$, so that $A \to B/\mathfrak{p}_1$ is universally
injective. Thus $A \to B$ is also universally injective. 
\end{proof} 
\subsection{Relation to submersions}
\label{SpecSubm}

In this subsection, we relate the $\arc$-topology to the topology of universal
submersions from \cite{Pic,Rydh}. 
Recall that coverings in the latter are given
by maps of schemes that are universally submersive
in the following sense. 

\begin{definition}[Submersions]
A map $f \cl X \to Y$ of qcqs schemes is called \emph{submersive} or a
\emph{submersion}
if it induces a quotient map on underlying topological spaces, i.e., $f$ is
surjective and $Y$ has the quotient topology. We say that $f$
is \emph{universally submersive} if 
for all maps $Y' \to Y$ of qcqs schemes, the base change 
$X \times_Y Y' \to Y'$
is a submersion. 
\end{definition} 

\begin{proposition} 
Let $f \colon X \to Y$ be a $v$-cover. 
Then $f$ is universally submersive. 
\end{proposition} 
\begin{proof} 
Since $v$-covers are stable under base change, it suffices to show that $f$
induces a quotient map. 
Indeed, if $U \subset Y$ is a subset such that $f^{-1}(U)$ is open, then the
hypothesis that $f$ is a $v$-cover implies that $U$ is stable under
generization (since valuation rings detect specializations, thanks to \cite[Tag 00IA]{stacks-project}). 
Now $f$ is a surjection, hence induces a quotient map in the constructible
topologies 
(as any continuous surjection of profinite sets is a quotient map); therefore, $f^{-1}(U)$ is also open in the constructible topology
and hence so is $U$. 
But any subset of $X$ which is both open in the constructible topology and
stable under generization is actually open, cf.~\cite[Tag 0903]{stacks-project}. 
If follows that $U$ is open. 
\end{proof}

In general, not every universal submersion is a $v$-cover; for example, 
if $A$ is a rank two valuation ring with 
$\spec(A) = \left\{0 , \mathfrak{p}, \mathfrak{m}\right\}$ with $0 \subsetneq
\mathfrak{p} \subsetneq \mathfrak{m}$, then $A \to A/\mathfrak{p} \times
A_{\mathfrak{p}}$ is a universal submersion \cite[Prop.~32]{Pic} which is
not a $v$-cover. 
By contrast, we will see that the notion of an $\arc$-cover is a slight
\emph{weakening} of that of a universal submersion. 
 Since universally submersive maps are not stable under limits (Remark~\ref{LargeExample}), we shall use instead the following variant where the quotient property is tested only using quasi-compact open subsets:

\begin{definition}[Spectral submersions] 
\label{SpectralSubmDef}
A map $f\colon X \to Y$ of qcqs schemes is called {\em spectrally submersive} or a {\em spectral submersion} if it satisfies the following two conditions:
\begin{enumerate}
\item $f$ is surjective; and 
\item given a subset $U \subset Y$, if the preimage $f^{-1}(U)$ is
quasi-compact and open, then $U$ is  open (and automatically quasi-compact).
\end{enumerate}
We say that $f$ is {\em universally spectrally submersive} if for all maps $Y'
\to Y$ of qcqs schemes, the base change $X \times_Y Y' \to Y'$ is a spectral submersion. 
\end{definition}

\begin{remark}
\label{AltDefSpecSubm}
Fix a surjective map $f\colon X \to Y$ of qcqs schemes. Condition (2) in Definition~\ref{SpectralSubmDef} can be relaxed to the following:
\begin{enumerate}
\item[$(2')$] Given a constructible subset $U \subset Y$, if the preimage $f^{-1}(U)$ is open, then $U$ is open. 
\end{enumerate}
To see why, fix $U \subset Y$ with $f^{-1}(U)$ quasi-compact open in $X$. We must
show that $U$ is a quasi-compact open if $f$ is surjective and satisfies $(2')$ above. As $f$ is surjective, the quasi-compactness
of $U$ is clear. To show openness, thanks to $(2')$, it is enough to show that
$U$ is constructible. 
The map $f \cl X \to Y$ is a surjective spectral map, and hence induces a quotient
map in the constructible topology. 
Now $f^{-1}(U)$ is constructible, hence clopen for the constructible topology;
thus $U$ is also clopen for the constructible topology and hence constructible. 
\end{remark}
\begin{remark}
As universally spectrally submersive maps are closed under composition and base
change, the property of being universally spectrally submersive can be checked
Zariski locally (or even $v$-locally) on the source and the target. 
That is, given a diagram
\[ \xymatrix{
X' \ar[d]^{f'} \ar[r] &  X \ar[d]^f  \\
Y' \ar[r] &  Y
}\]
of qcqs schemes with both maps $X' \to X$ and $Y' \to Y$ being $v$-covers, 
the map $f$ is universally spectrally submersive if $f'$ is
universally spectrally submersive. 
This follows easily from \Cref{AltDefSpecSubm}. 
\end{remark}

The following stability property of universal spectral submersions is crucial for our application.
In Remark~\ref{LargeExample}, we give a counterexample showing that this
property fails if we drop the adjective ``spectral.''

\begin{lemma}
\label{UnivSpecSubmLimit}
Let $f\colon X \to Y$ be a map of qcqs schemes that can be written as a cofiltered
inverse limit of morphisms $f_i\colon X_i \to Y_i$ of qcqs schemes along affine
transition maps. If each $f_i$ is a spectral submersion (resp.~a universal
spectral submersion), then  $f$ is a spectral submersion (resp.~a universal
spectral submersion).
\end{lemma}
\begin{proof}
Let us first show that $f$ is a spectral submersion when the $f_i$ are
spectrally submersive. As each $f_i$ is surjective, the map $f$ is also
surjective: if $f_i$ is surjective, then $f_i^{cons}$ is surjective, whence
$f^{cons} := \varprojlim_i f_i^{cons}$ is surjective by Tychonoff's theorem,
and thus $f$ is surjective. For the rest, we use the criterion in
Remark~\ref{AltDefSpecSubm}. Choose a constructible subset $U \subset Y$ such
that $f^{-1}(U)$ is open. As $Y \simeq \varprojlim_i Y_i$, we also have
$Y^{cons} \simeq \varprojlim_i Y_i^{cons}$, so the constructible set $U$ arises
as the pullback of some constructible set $U_i \subset Y_i$ for some index $i$.
For $j \geq i$, write $U_j$ for the preimage of $U_i$ in $Y_j$. As each $f_j$
is spectrally submersive, it is enough to show that the constructible set
$f_j^{-1}(U_j)$ is open for $j \gg i$. As $f^{-1}(U)$ is a quasi-compact open in
$X$ and $X \simeq \varprojlim_i X_i$, we can realize $f^{-1}(U)$ as the preimage of some quasi-compact open $V_k \subset X_k$ for some index $k$; reindexing, we may assume $k = i$. Write $V_j \subset X_j$ for the preimage of $V_i$. Then $\varprojlim_{j \geq i} V_j = f^{-1}(U) = \varprojlim_{j \geq i} f_j^{-1}(U_j)$ as constructible subsets of $X$.  But then we must have $V_j = f_j^{-1}(U_j)$ for $j$ sufficiently large: by passage to the constructible topology, this reduces to the observation that in a cofiltered inverse limit $S := \varprojlim_i S_i$ of profinite sets $S_i$, if we have clopen subsets $W, W' \subset S_i$ for some $i$ with the same preimage in $S$, then the preimages of $W$ and $W'$ in $S_j$ for $j \gg i$ must coincide (as the collection of clopen subsets of $S$ is the direct limit of the collection of clopen subsets of the $S_i$'s). In particular, we learn that $f_j^{-1}(U_j)$ is open for $j \gg i$, as wanted.

It remains to check that every base change of $f$ is spectrally submersive if the $f_i$'s are universally spectrally submersive. But a base change of $f$ can be realized as the cofiltered inverse limit of base changes of the $f_i$'s. As the $f_i$'s are universally spectrally submersive, the same holds for their base changes, so we conclude using the previous paragraph. 
\end{proof}

\begin{remark}
\label{LargeExample}
We give an example of a universal spectral submersion that is not a quotient
map. 
Write $\mathbf{N}$ for the set of positive natural numbers. Let $V$ be a
valuation ring with $\mathrm{Spec}(V)$ given by the totally ordered set $T :=
\{0\} \sqcup \{\frac{1}{n} \mid n \in \mathbf{N}\} \subset [0,1]$, so $0 \in T$
corresponds to the generic point; the existence of such a valuation ring
follows from \cite[Theorem 3.9 (b), (e)]{SarussiToset}. Write $\mathfrak{p}_n
\in \mathrm{Spec}(V)$ for the prime corresponding to $\frac{1}{n} \in T$ for $n
\geq 1$; set $V_n := (V/\mathfrak{p}_n)_{\mathfrak{p}_{n-1}}$ for any $n > 1$
and set $V_1 = \kappa(\mathfrak{p}_1)$ to be the residue field of $V$. Then
each $V_n$ is a rank $\leq 1$ valuation ring (with rank exactly $1$ unless $n =
1$) with residue field $\kappa(V_n)$ identified with
$\kappa(\mathfrak{p}_{n-1})$ if $n > 1$ and $\kappa(\mathfrak{p}_1)$ when
$n=1$. Moreover, we have $\cap_n \mathfrak{p}_n = (0)$. In particular, the map
$V \to B := \prod_{n} V_n$ is injective, whence $f\colon  X:=\mathrm{Spec}(B)
\to \mathrm{Spec}(V)$ hits the generic point and is thus surjective (as
non-generic points are obviously in the image). Moreover, the map $f$ is also
an $\arc$-cover: the specializations $\mathfrak{p}_n \rightsquigarrow
\mathfrak{p}_{n-1}$ exhaust all the immediate specializations in
$\mathrm{Spec}(V)$, so any map $V \to W$ where $W$ is a rank $1$ valuation ring
either factors over some $V \to V_n$ or factors over a residue field of $V$
(we include the fraction field of $V$ as a possibility). It follows from Proposition~\ref{UnivSubArc} that $f$ is a universal spectral submersion. 

Let us show that $f$ is not a quotient map. First, one checks that $\pi_0(X) = \beta
(\mathbf{N})$ is the Stone--\v{C}ech compactification of $\mathbf{N}$
(cf.~\Cref{SCcompact} below for a review). The
connected component of $X$ corresponding to an ultrafilter $\mathfrak{U} \in
\beta(\mathbf{N})$ (i.e., the fiber of $X \to \pi_0(X) \simeq
\beta(\mathbf{N})$ over $\mathfrak{U} \in \beta (\mathbf{N})$) is given by the
spectrum of the ultraproduct $\prod_{\mathfrak{U}} V_n$ (see also
Remark~\ref{UltraprodSpec} and Lemma~\ref{ultraaic} below). We claim that the
preimage $f^{-1}((0))$ of the generic point of $\mathrm{Spec}(V)$ coincides
with the preimage in $X$ of the closed set $\beta \mathbf{N} - \mathbf{N}
\subset \beta\mathbf{N}$. Granting this claim, it is clear that $f$ is not a
quotient map: if it were, then the generic point of $\mathrm{Spec}(V)$ is
closed, which is absurd. To identify $f^{-1}((0))$ with $\beta(\mathbf{N}) -
\mathbf{N}$, note that the containment $f^{-1}((0)) \subset \beta(\mathbf{N}) -
\mathbf{N}$ is clear: the points in the preimage of $\mathbf{N}$ map down to
the $\mathfrak{p}_n$'s in $\mathrm{Spec}(V)$. Conversely, given a non-principal
ultrafilter $\mathfrak{U}$ on $\mathbf{N}$, we must show that the image on
$\mathrm{Spec}(-)$ of the map $V \to \prod_{\mathfrak{U}} V_n$ is the generic
point. Unwinding the definitions, this amounts to checking that if $\mathfrak{U}$ is not principal, then the map $V \to \prod_{\mathfrak{U}} \kappa(V_n)$ is injective. If this map were not injective, then the kernel would be nonzero. But, by definition of the ultraproduct, this kernel is exactly those $a \in V$ such that $\{n \in \mathbf{N} \mid a = 0 \in \kappa(V_n) \} \in \mathfrak{U}$; this set is finite if $a \neq 0$ since $a \notin \mathfrak{p}_n$ for $n \gg 0$, whence $a \in \kappa(V_n)^*$ for $n \gg 0$. Taking a nonzero $a$ in the kernel then shows that $\mathfrak{U}$ contains a finite set, so it must be principal, as wanted.

In this example, if we write $V$ as a filtered direct limit $\varinjlim_i W_i$ of finite rank valuation rings $W_i$ along faithfully flat maps as in Lemma~\ref{filtcolimitaic}, then the induced maps $\mathrm{Spec}(B) \to \mathrm{Spec}(W_i)$ are all (universal) submersions, but the inverse limit $\mathrm{Spec}(B) \to \mathrm{Spec}(V) \simeq \varprojlim_i \mathrm{Spec}(W_i)$ is not a quotient map. Thus, we also obtain an example of the failure of (universal) submersions to be stable under filtered inverse limits; this phenomenon cannot happen with spectral submersions by Lemma~\ref{UnivSpecSubmLimit}.
\end{remark}

We now arrive at our main topological characterization of $\arc$-covers; this was independently observed by Rydh in the forthcoming \cite{RydhII}.

\begin{proposition}[$\arc$-covers and universal spectral submersions]
\label{UnivSubArc}
Let $f\colon X \to Y$ be a map of qcqs schemes. Then $f$ is universally spectrally submersive if and only if $f$ is an $\arc$-cover.
\end{proposition}

We note the following immediate corollary of Proposition~\ref{UnivSubArc} and
Lemma~\ref{UnivSpecSubmLimit}. 
\begin{corollary} 
The collection of $\arc$-covers
of qcqs schemes
is closed
under filtered inverse limits with affine transition maps.  \qed 
\end{corollary}

 The proof of Proposition~\ref{UnivSubArc} will rely on a few (standard) facts about valuation rings that we recall next; these will also be useful later. 
For future reference, we also include assertions about absolutely integrally
closed valuation rings (see Definition~\ref{aic}), i.e., those valuation
rings with algebraically closed fraction field.

\begin{lemma} 
Let $W$ be a valuation ring with fraction field $K$. Let $K_0 \subset K$ be a
subfield. Then $W_0 := K_0 \cap W$ is a valuation ring, and
the map $W_0 \to W$ is faithfully flat. 
If, in addition, $K_0$ is algebraically closed, then $W_0$ is absolutely integrally closed. 
\label{aicflatlemma}
\end{lemma} 
\begin{proof} 
We see immediately that $W_0$ is a valuation ring in the field $K_0$. 
Moreover, the map $W_0 \to W$ is a local injective homomorphism 
of valuation rings, so it is faithfully flat. 

Suppose now that $K_0$ is algebraically closed. 
Then $W_0$ is absolutely integrally closed, since it is integrally closed and has algebraically closed
fraction field.  
\end{proof} 

\begin{lemma} 
\label{filtcolimitaic}
Let $V$ be  a valuation ring. Then $V$ can be written as a filtered colimit of a family of valuation subrings
$V_i \subset V$ such that: 
\begin{enumerate}
\item Each $V_i$ has finite rank.  
\item Each map $V_i \to V_j$ is faithfully flat. 
\end{enumerate}
If in addition $V$ is absolutely integrally closed, we may also choose each $V_i$ to be absolutely integrally closed. 
\end{lemma} 
\begin{proof} 
Consider first the case of valuation rings (not assumed absolutely integrally closed). 
Let $K$ be the  fraction field of $V$. 
For each finitely generated subfield $K_i \subset K$, we let $V_i = K_i \cap V$;
this is a valuation subring of $K_i$, and is finite rank since the
transcendence degree of $K_i$ is finite, cf.~\cite[Theorem
VI.10.3.1]{BourbakiCA}; furthermore $V_i \to V$ is an extension of valuation
rings by \Cref{aicflatlemma}. 
This gives the desired expression of $V$ as a filtered colimit. When $V$ is absolutely integrally closed,
we run a similar argument but take the $\left\{K_i\right\}$ to be the family of
algebraically closed subfields of $K$ of finite transcendence degree over the
prime field. 
\end{proof} 

\begin{proof}[Proof of Proposition~\ref{UnivSubArc}]
Assume first that $f$ is universally spectrally submersive. We must show $f$ is
an $\arc$-cover. By base change and Lemma~\ref{arccovercrit}, it is enough to
show that the following holds: if $Y := \mathrm{Spec}(V)$ for a valuation ring
$V$ of rank $1$, the nontrivial specialization $(0) \rightsquigarrow
\mathfrak{m}$ in $\mathrm{Spec}(V)$ lifts to a specialization in $X$.  Now
$f^{-1}( (0)) \subset X$ is a nonempty quasi-compact open subset that is not
closed: if it were closed, its complement $f^{-1}(\mathfrak{m}) \subset X$
would be a quasi-compact open, which, by the spectral submersiveness of $f$, would imply the absurd conclusion that the closed point of $\mathrm{Spec}(V)$ is open.  As points in the closure of a quasi-compact open of a spectral space are given by specializations \cite[Tag 0903]{stacks-project}, there must exist a specialization $x_0 \rightsquigarrow x_{\mathfrak{m}}$ in $X$ such that $x_0 \in f^{-1}( (0))$ and $x_{\mathfrak{m}} \in X - f^{-1}( (0))$. But then $x_{\mathfrak{m}} \in f^{-1}(\mathfrak{m})$, so we have found the desired specialization.

Conversely, assume that $f$ is an $\arc$-cover. We must show $f$ is universally
spectrally submersive. In fact, since $\arc$-covers are stable under base
change, it is enough to show that $f$ is spectrally submersive. By
Remark~\ref{AltDefSpecSubm}, we must show the following: a constructible subset
$U \subset Y$ is open if $f^{-1}(U) \subset X$ is open. We may assume that $U$ is non-empty. 

First, assume $Y = \mathrm{Spec}(V)$ is the spectrum of a finite rank valuation ring. Then $\mathrm{Spec}(V) = \{\mathfrak{p}_0
\rightsquigarrow \mathfrak{p}_1 \rightsquigarrow ... \rightsquigarrow
\mathfrak{p}_n\}$ is a finite totally ordered set of prime ideals $\mathfrak{p}_i \subset V$. Each immediate specialization
$\mathfrak{p}_i \rightsquigarrow \mathfrak{p}_{i+1}$ gives rise to a map $V \to
V_i := (V/\mathfrak{p}_i)_{\mathfrak{p}_{i+1}}$ where the target is a rank $1$
valuation ring. As $f$ is an $\arc$-cover, these maps lift to $X$ (after extending $V_i$ if necessary). It is thus enough to check the claim when $f$ equals the
map $\sqcup_{i=0}^n \mathrm{Spec}( V_i) \to \mathrm{Spec}(V)$. But this is an
elementary fact about finite totally ordered sets. Our hypothesis on $U$ implies that each $U_i := U \cap \mathrm{Spec}(V_i)$ is open in $\mathrm{Spec}(V_i)$ for all $i$. As $U = \cup_i U_i$ is non-empty, we may choose $j$ maximal such that $U_j \neq \emptyset$. If $j=0$, the claim is clear. If not, then $U_j \subset \mathrm{Spec}(V_j)$ is either just the generic point $\{\mathfrak{p}_j\} \subset \mathrm{Spec}(V_j)$ or the whole space $\{\mathfrak{p}_j,\mathfrak{p}_{j+1}\} = \mathrm{Spec}(V_j)$. In either case, we have $\mathfrak{p}_j \in U$, which then implies $\mathfrak{p}_{j-1} \in U$ since $U_{j-1} = U \cap \mathrm{Spec}(V_{j-1})$ is open in $\mathrm{Spec}(V_{j-1}) = \{\mathfrak{p}_{j-1},\mathfrak{p}_j\}$. Continuing this way, we find $U = \{\mathfrak{p}_0 \rightsquigarrow ... \rightsquigarrow \mathfrak{p}_k\}$ for $k \in \{j,j+1\}$, so $U$ is open. 

Next, assume $Y = \mathrm{Spec}(V)$ is the spectrum of an arbitrary valuation ring $V$. Write $V$ as a filtered colimit $\varinjlim V_i$ of finite rank valuation rings $V_i \subset V$ as in Lemma~\ref{filtcolimitaic}, and let $g_i\colon \mathrm{Spec}(V) \to \mathrm{Spec}(V_i)$ be the projection. The induced map $f_i\colon X \to \mathrm{Spec}(V_i)$ is then an $\arc$-cover by Proposition~\ref{visv1}, and thus a spectral submersion by the previous paragraph.  It follows from Lemma~\ref{UnivSpecSubmLimit} that the inverse limit $f\colon X \to \mathrm{Spec}(V) \simeq \varprojlim_i \mathrm{Spec}(V_i)$ of the $f_i$'s is also a spectral submersion.

Finally, we tackle the general case of any  qcqs $Y$. As a constructible set in a spectral space is open exactly when it is stable under generalizations \cite[Tag 0903]{stacks-project}, it is enough to show that $U$ is stable under generalizations, i.e., if $u \in U$ and $y \in Y$ is a generalization of $u$ in $Y$, then $y \in U$. Choose a valuation ring $V$ and a map $g\colon \mathrm{Spec}(V) \to Y$ such that the specialization $y \rightsquigarrow u$ lifts to a specialization $y' \rightsquigarrow u'$ in $\mathrm{Spec}(V)$ along $g$. Assume that the base change $f_V\colon X \times_Y \mathrm{Spec}(V) \to \mathrm{Spec}(V)$ of $f$ along $g$ is already known to be spectrally submersive. Now $f_V^{-1}(g^{-1}(U))$ is a quasi-compact open in $X \times_Y \mathrm{Spec}(V)$ by hypothesis, so $g^{-1}(U)$ must be open by spectral submersiveness of $f_V$. In particular, since $u' \in g^{-1}(U)$, we must have $y' \in g^{-1}(U)$, whence $y = g(y') \in U$, as wanted. In other words, we have made the promised reduction. 
\end{proof}

\newpage
\section{Studying $v$-sheaves via valuation rings}

In this section, we shall describe how to control $v$-sheaves in terms of their behaviour on absolutely integrally closed valuation rings. We begin in \S \ref{Finvsheaf} by isolating a class of $v$-sheaves that we call {\em finitary}; these are the ones that commute with filtered colimits of rings, and are thus determined by their values on finitely presented rings. In \S \ref{ultraprod}, we collect some tools for describing the behaviour of finitary $v$-sheaves on (infinite) products of rings in terms of ultraproducts. In \S \ref{DetectUnivDescent}, we prove some general results concerning morphisms of universal $\mathcal{F}$-descent with respect to certain functors $\mathcal{F}$. Finally, these ingredients are put together in \S \ref{vSheafAICVR} to explain why finitary $v$-sheaves can be controlled by their behaviour on absolutely integrally closed valuation rings.

\subsection{Finitary $v$-sheaves}
\label{Finvsheaf}

In the sequel, we will need to work with sheaves for the $v$- and
$\arc$-topology.
Throughout, we let $\mathcal{C}$ be an $\infty$-category that has all small
limits; we will later impose additional assumptions on $\mathcal{C}$. 
\begin{definition} 
\label{DefFinitary}
Let $R$ be a ring, and let $\mathcal{C}$ be an $\infty$-category that has
small limits. A functor $\mathcal{F}\colon  \mathrm{Sch}_{qcqs,R}^{op} \to \mathcal{C}$ is
 a \emph{$v$-sheaf} (resp.~\emph{$\arc$-sheaf}) if $\mathcal{F}$
carries finite disjoint unions of schemes to products and for every $v$-cover
(resp.~$\arc$-cover) $Y \to X$ in $\mathrm{Sch}_{qcqs,R}$, the natural map  
\[\mathcal{F}(X) \to \varprojlim( \mathcal{F}(Y) \rightrightarrows \mathcal{F}(Y \times_X Y) \triplearrows \dots )\]
is an equivalence.
Note that the limit is indexed over the simplex category $\Delta$, i.e., it 
is a totalization. 

Suppose that $\mathcal{C}$ has filtered colimits.  We will say that
$\mathcal{F}$ is \emph{finitary} if whenever $\left\{Y_\alpha\right\}_{\alpha
\in A}$ is a tower of qcqs $R$-schemes (indexed over some cofiltered partially
ordered set $A$) with affine transition maps, then $
\varinjlim_{A} \mathcal{F}(Y_\alpha) \xrightarrow{\simeq}
\mathcal{F}( \varprojlim_\alpha
Y_\alpha) $. In this case, by ``relative approximation'' \cite[Tag 09MU]{stacks-project}, $\mathcal{F}$ is determined by its restriction to $R$-schemes of finite presentation. 

For functors defined on $\mathrm{Ring}_R$ instead of $\mathrm{Sch}_{qcqs,R}^{op}$, we have analogous definitions. 
\end{definition} 

We next review generalities on descent. 
To abstract the sheaf property, we use the following definition. 
\begin{definition}[Universal $\mathcal{F}$-descent] 
Let $\mathcal{F}\colon  \mathrm{Sch}_{qcqs,R}^{op} \to \mathcal{C}$ be a functor, with
$\mathcal{C}$ as in \Cref{DefFinitary}. 
We say that a map $f\colon  Y \to X $ in $\mathrm{Sch}_{qcqs,R}$ is of
\emph{$\mathcal{F}$-descent} if the
natural map
\[ \mathcal{F}(X) \to \varprojlim( \mathcal{F}(Y) \rightrightarrows \mathcal{F}(Y \times_{X}
Y) \triplearrows \dots )  \]
is an equivalence in $\mathcal{C}$. 
We say that $f$ is of \emph{universal $\mathcal{F}$-descent} if all base changes $f'\colon  Y' \to X'$ of $f$ (along maps $X' \to X$ in
$\mathrm{Sch}_{qcqs,R}$) are of $\mathcal{F}$-descent. 
\end{definition}  

\begin{example}
For an abelian group $A$, let $\mathcal{F}(X) = R\Gamma(X_{\et},A)$ valued in $\mathcal{C} = \mathcal{D}(\mathbf{Z})$. Then any \'etale surjection $Y \to X$ is of universal $\mathcal{F}$-descent: this is a general fact about cohomology on a site. More interesting, for $A$ torsion, any proper surjection is of universal $\mathcal{F}$-descent by proper cohomological descent for torsion \'etale cohomology; see Proposition~\ref{vDescentEtaleCoh}  for a more general assertion.
\end{example}

We recall the following basic result (cf. \cite[Lemma 3.1.2]{LZ}), which gives
basic closure properties for the class of maps of universal $\mathcal{F}$-descent. 
Compare also 
\cite[10.10 and 10.11]{Gir64} and \cite[Vbis, 3.3.1]{SGA4} for closely related results. 

\begin{lemma} 
\label{Fsorite}
Fix a functor $\mathcal{F}\colon  \mathrm{Sch}_{qcqs,R}^{op} \to \mathcal{C}$ as before. 
Let $f\colon  Y \to X $ and $g\colon  Z \to Y$ be maps in $\mathrm{Sch}_{qcqs,R}$. \begin{enumerate}
\item If $f$ has a section, then $f$ is of universal $\mathcal{F}$-descent. 
\item  
If $f, g$ are of universal $\mathcal{F}$-descent, 
then $f \circ g$ is of universal $\mathcal{F}$-descent. 
\item  If $f \circ g$ is of universal $\mathcal{F}$-descent, then $f$ is of universal
$\mathcal{F}$-descent. 
\item The collection of maps of universal $\mathcal{F}$-descent is closed under
base change. 
\end{enumerate}
\end{lemma} 

We now specialize to cases such as $\mathcal{C} = \D( \Lambda)^{\geq 0}$,
for $\Lambda$ a ring, where we have additional properties. 
To begin with, we isolate the relevant desired feature, following \cite{EHIK}. 

\begin{definition}[{\cite[Def.~3.1.4]{EHIK}}]  
We say that an $\infty$-category $\mathcal{C}$ is
\emph{compactly generated by cotruncated objects} if: 
\begin{enumerate}
\item $\mathcal{C}$ is compactly generated 
(i.e., $\mathcal{C}$ is presentable and is generated under colimits by its compact objects). Compare
\cite[Sec.~5.5.7]{HTT} for an account of the theory of
compactly generated $\infty$-categories. 

\item 
Every compact object $x \in \mathcal{C}$ is $n$-cotruncated for some $n$
(which may depend on $x$), i.e.,  $\hom_{\mathcal{C}}(x, y)$ is $n$-truncated for all $y \in \mathcal{C}$. 
\end{enumerate}
\end{definition}

\begin{example} 
The following give examples of $\infty$-categories which are compactly generated by cotruncated objects. 
\begin{enumerate}
\item Let $\Lambda$ be any ring. Then the $\infty$-category $\D( \Lambda)^{\geq
0}$, the coconnective part of the derived $\infty$-category $\D( \Lambda)$, is
compactly generated by cotruncated objects. In fact, the compact objects have
the form $\tau^{\geq 0} P$ for a perfect complex $P \in \D(\Lambda)$ (since
perfect complexes compactly generate $\D(\Lambda)$), and this
object is $n$-cotruncated if $P \in \D(\Lambda)^{\leq n}$. 

\item Let $n \geq 0$ be an integer. 
Then the $\infty$-category $\mathcal{S}_{\leq n}$ of $n$-truncated spaces is
compactly generated by cotruncated objects. This follows because $\mathcal{S}$
is compactly generated by $\ast$, which becomes cotruncated in
$\mathcal{S}_{\leq n}$. 
\item 
Let $n \geq 0$ be an integer. 
Then the $\infty$-category $\mathrm{Cat}_{ n}$ of $(n, 1)$-categories is
compactly generated by cotruncated objects. This follows because
the $\infty$-category $\mathrm{Cat}_\infty$ of $(\infty, 1)$-categories is compactly generated by $\ast, \Delta^1$. 
\end{enumerate}
In contrast, the ``unbounded'' versions of each of these examples (i.e., the
full derived $\infty$-category $\mathcal{D}(\Lambda)$, the $\infty$-category
$\mathcal{C}$ of spaces, and the $\infty$-category $\mathrm{Cat}_{\infty}$) are not compactly generated by cotruncated objects. 
\end{example} 

The following lemma captures some essential features of this notion for our purposes.

\begin{lemma}
\label{univFclosedfiltlimit}
Let $\mathcal{C}$ be an $\infty$-category compactly generated cotruncated objects. 
\begin{enumerate}
\item Totalizations in $\mathcal{C}$ commute with filtered colimits. 
\item Suppose that $\mathcal{F}\colon  \mathrm{Sch}_{qcqs,R}^{op} \to
\mathcal{C}$ is finitary. Then the collection of maps which are of universal
$\mathcal{F}$-descent (resp.~$\mathcal{F}$-descent) is closed under filtered limits with affine transition maps. 
\end{enumerate}
\end{lemma}
\begin{proof}
For part (1), observe that totalizations and filtered colimits in $\mathcal{C}$ can both be detected by applying $\hom_{\mathcal{C}}(x, \cdot)$ for $x$ compact.  Now we use that $\mathcal{S}_{\leq n}$ is an $(n+1, 1)$-category, so $\mathrm{Tot} \simeq \mathrm{Tot}^{n+2}$, and finite limits commute with filtered colimits. 
Part (2) follows from part (1) as $\mathcal{F}$ is finitary.
\end{proof}

In addition,  the  $v$-topology has the following basic finiteness property.  The analogous result is \emph{not} true with the $\arc$-topology.

\begin{proposition} 
\label{finitevtopology}
Let $R$ be a ring. Let $\mathcal{C}$ be an $\infty$-category that is compactly generated by cotruncated objects. 
A finitary functor $\mathcal{F}\colon  \mathrm{Sch}_{qcqs,R}^{op} \to \mathcal{C}$ is a $v$-sheaf if and only
if it satisfies the following conditions:
\begin{enumerate}
\item $\mathcal{F}$ carries finite disjoint unions of finitely presented $R$-schemes to
products. 
\item For every $v$-cover of finitely presented $R$-schemes $Y \to X$,  the map 
\[ \mathcal{F}(X) \to \varprojlim ( \mathcal{F}(Y) \rightrightarrows \mathcal{F}(Y \times_X Y) \triplearrows \dots )\]
is an equivalence. 
\end{enumerate}
\end{proposition} 
\begin{proof} 
Let $Y' \to X'$ be a $v$-cover between qcqs $R$-schemes. 
We want to show that this map is of universal $\mathcal{F}$-descent. 
Suppose that $Y', X'$ are finitely presented over $R$. Because any qcqs $X'$-scheme is a filtered
limit of finitely presented $X'$-schemes under affine transition maps \cite[Tag
09MU]{stacks-project}, it follows that  any base change of $Y'
\to X'$ in this case is a filtered limit of $v$-covers between finitely
presented $R$-schemes. 
Our hypotheses thus show that $Y' \to X'$ is of universal $\mathcal{F}$-descent. 

A similar argument as in the preceding 
paragraph shows that $\sF$ is a Zariski sheaf. 
Indeed, any Zariski cover $Y' \to X'$ of qcqs $R$-schemes can be written by
general limit formalism (cf.~\cite[Tag 01YT]{stacks-project}) as a
filtered limit with affine transition maps of Zariski covers of finitely
presented $R$-schemes; by assumption, these latter Zariski covers are of
$\sF$-descent, and hence $Y' \to X'$ is of $\sF$-descent. 

Now suppose that $Y' \to X'$ is an arbitrary $v$-cover between qcqs
$R$-schemes, which we need to show is of universal $\sF$-descent. 
Without loss of generality, we can assume that $X'$ is affine since $\sF$ is
a Zariski sheaf (and by \Cref{Fsorite}, since Zariski covers are of universal
$\sF$-descent). 
Since $Y' \to X'$ is a filtered limit of finitely presented maps $Y'_\alpha \to
X'$ (each of which is forced to be a $v$-cover) with affine transition maps, we
can assume that $Y' \to X'$ is finitely presented. Then using \cite[Th.~3.12]{Rydh}, $Y' \to X'$ admits a refinement $Y'' \to X'$ which is a composite of a
quasi-compact open cover and a proper finitely presented surjection. It suffices
(by \Cref{Fsorite}) to show that this refinement is of universal $\mathcal{F}$-descent. 
By general results (see, e.g., \cite[Tag 01YT]{stacks-project}) we can 
descend quasi-compact open covers and proper finitely presented surjections to a
base
which is finitely presented 
over $R$. Therefore, $ Y'' \to X' $ is a filtered limit (with affine transition
maps) of $v$-covers between finitely presented $R$-schemes. Since we saw earlier
that these maps are of universal $\mathcal{F}$-descent, it follows that $Y'' \to X'$ is
of universal $\mathcal{F}$-descent, as desired. The statement that
$\mathcal{F}$ carries finite disjoint unions of qcqs $R$-schemes to 
finite products is similar and easier. 
\end{proof} 

\begin{corollary} 
\label{finitaryv}
Let $\mathcal{C}$ be an $\infty$-category that is compactly generated by cotruncated objects. 
The $\infty$-category of finitary $\mathcal{C}$-valued $v$-sheaves on $\mathrm{Sch}_{qcqs,R}$ is
equivalent to the $\infty$-category of $\mathcal{C}$-valued $v$-sheaves on finitely
presented $R$-schemes. 
\qed
\end{corollary}

\subsection{Ultraproducts and sheaves}
\label{ultraprod}

To proceed further, we will need to show that  equivalences of $v$-sheaves can
be detected by their values on absolutely integrally closed valuation rings (at
least in the noetherian case, this is a classical result). In order to do so, we
first review some facts about ultraproducts and sheaves; our goal is to prove
Corollary~\ref{ultradetectsequivalence}, explaining how certain functors defined
on all commutative rings can be controlled when evaluated on an infinite product
of rings. We also refer to \cite{Jardine} for a more detailed treatment.  For the rest of this section, we fix a base ring $R$ and an $\infty$-category $\mathcal{C}$ admitting small limits and filtered colimits.

Recall that under Stone duality (see \cite{Johnstone} for a general reference),
there is a duality between Boolean algebras and profinite sets given by sending such a space to its collection of clopen subsets. One can also
describe sheaves on profinite sets in terms of the corresponding Boolean algebra:

\begin{proposition}
\label{SheafStone}
Let $X$ be a profinite set, and let $\mathcal{B}$ be the poset of quasi-compact
open (or equivalently clopen) subsets of $X$. The $\infty$-category of
$\mathcal{C}$-valued sheaves on $X$ is identified via restriction with the
$\infty$-category of functors $\mathcal{B}^{op} \to \mathcal{C}$ that carry finite disjoint unions to products.
\end{proposition}
See \cite[Corollary 1.1.4.5]{SAG} for a more general statement. 

\begin{proof}
Given a $\mathcal{C}$-valued sheaf $f$ on $X$, restriction certainly gives a
functor $\overline{f}\colon \mathcal{B}^{op} \to \mathcal{C}$ carrying finite
disjoint unions to coproducts. Conversely, given such a functor $g\colon
\mathcal{B}^{op} \to \mathcal{C}$, we can define a presheaf $\widetilde{g}$ on
$X$ by setting $\widetilde{g}(U) = \varprojlim g(V)$ where the limit runs over
all quasi-compact open subsets $V \subset U$. To see that $\widetilde{g}$ is a
sheaf, it is enough to check that its restriction $g$ to quasi-compact open
subsets forms a sheaf; but this follows from the assumption on $g$ together
with the observation that if $V \subset U$ is an inclusion of quasi-compact open
subsets of $X$, then both $U$ and $V$ are clopen, and $U = V \sqcup (U-V)$ is
the disjoint union of $V$ with its complement in $U$. It is straightforward to
check that these constructions give inverse equivalences.  \end{proof}

The following definition is partially motivated by Proposition~\ref{SheafStone}.

\begin{definition} 
\label{SheafPowerSet}
Let $T$ be a set. We let $\mathcal{P}(T)$ denote the poset of all subsets of $T$.  We say that a functor $f\colon  \mathcal{P}(T)^{op} \to \mathcal{C}$ is a \emph{sheaf} if it carries finite disjoint unions to products.
\end{definition}
\begin{example} 
\label{prodfunctor}
Suppose that we have a functor $u\colon  T \to \mathcal{C}$ (where $T$ is considered as a
discrete category). Then we have a sheaf $f\colon  \mathcal{P}(T)^{op} \to
\mathcal{C} $ defined via $f(T') = \prod_{t \in T'} u(t)$. 
\end{example}

This notion of a sheaf turns out to be equivalent to the classical notion of a sheaf on the profinite set corresponding to the Boolean algebra $\mathcal{P}(T)$.

\begin{cons}[Stone--\v{C}ech compactification] 
\label{SCcompact}
\label{StoneCech}
Let $T$ be a set. An \emph{ultrafilter} $\mathfrak{U}$ on $T$ is a collection
of subsets such that $\emptyset \notin \mathfrak{U}$, the
collection $\mathfrak{U}$ is closed
under finite intersections, and for every $T' \subset T$, either $T'$ or $T
\setminus T'$ belongs to $\mathfrak{U}$. Each element $t \in T$ defines the {\em
principal ultrafilter} $\mathfrak{U}_t$ of all subsets of $T$ containing $t$.

The collection of all ultrafilters on $T$ is naturally a profinite set $\beta
T$, called the {\em Stone--\v{C}ech compactification} of $T$ (considered as a
topological space with the discrete topology), cf.~\cite[Ch.~6]{GJ60} for an
account of the Stone--\v{C}ech compactification more generally. The assignment $t
\mapsto \mathfrak{U}_t$ gives an open embedding $T \hookrightarrow \beta T$ with
dense image, where $T$ again has the discrete topology. The Boolean algebra of all
clopen subsets of $\beta T$ is identified with the power set $\mathcal{P}(T)$
via $(V \subset \beta T) \mapsto (V \cap T \subset T)$. Under the equivalence of
Proposition~\ref{SheafStone}, the $\infty$-category of sheaves on $\mathcal{P}(T)$ in the sense of Definition~\ref{SheafPowerSet} is equivalent to the category of sheaves on $\beta T$ in the classical sense. In particular, given a sheaf $f\colon  \mathcal{P}(T)^{op} \to \mathcal{C}$ and an ultrafilter
$\mathfrak{U}$ on $T$, we may define {\em the stalk of $f$ at $\mathfrak{U}$} by the formula  
\[ f_{\mathfrak{U}} = \varinjlim_{T' \in \mathfrak{U}} f(T'),\] 
where the colimit is taken over the (filtered) partially ordered set of all $T' \in \mathfrak{U}$.
\end{cons}

\begin{lemma} 
\label{checkequivalence}
Assume that $\mathcal{C} $ is compactly generated. Fix a map $\eta\colon f \to
g$ of sheaves in $\mathrm{Fun}( \mathcal{P}(T)^{op}, \mathcal{C})$. Then $\eta$
is an equivalence if and only if $\eta_{\mathfrak{U}} \colon f_{\mathfrak{U}} \to g_{\mathfrak{U}}$ is an equivalence  for each ultrafilter $\mathfrak{U}$ on $T$.
\end{lemma} 
\begin{proof} 
The ``only if'' direction is clear. Thus, assume $\eta_{\mathfrak{U}}$ is an equivalence for each ultrafilter $\mathfrak{U}$ on $T$. 
Applying $\mathrm{Hom}_{\mathcal{C}}(x, -)$ for each compact object $x \in
\mathcal{C}$, we can reduce to the case 
where $\mathcal{C} = \mathcal{S}$. 
In this case, the result follows because sheaves of spaces on the 
Stone--\v{C}ech compactification $\beta T$ (which has covering dimension zero
in the sense of \cite[Def.~7.2.3.1]{HTT}) are automatically hypercomplete, by
\cite[Th.~7.2.3.6]{HTT}, whence equivalences can be tested on stalks. 
\end{proof} 

To proceed further, recall the following classical definition.
Fix a base ring $R$. 
\begin{definition}[Ultraproducts of rings]
Given a set $\left\{A_t\right\}_{t \in T}$ of $R$-algebras and an ultrafilter  $\mathfrak{U}$ on $T$, we define the \emph{ultraproduct} $\prod_{\mathfrak{U}} A_t$ via the formula
\[ \prod_{\mathfrak{U}} A_t = \varinjlim_{T' \in \mathfrak{U}} \left(  \prod_{t
\in T'} A_t \right). \]
Note that the colimit appearing above is filtered.
\end{definition} 

\begin{remark}[Ultraproducts via the spectrum of a product]
\label{UltraprodSpec}
Given a set $\{A_t\}_{t \in T}$ of commutative rings, the space $\mathrm{Spec}(\prod_{t \in T} A_t)$ comes equipped with a natural projection map $\pi\colon \mathrm{Spec}(\prod_{t \in T} A_t) \to \beta T$ determined by requiring that the preimage $\pi^{-1}(U)$ of a quasi-compact open $U \subset \beta T$ corresponding to a subset $T' \subset T$ is the clopen subscheme $\mathrm{Spec}(\prod_{t \in T'} A_t) \subset \mathrm{Spec}(\prod_{t \in T} A_t)$. The ultraproduct $\prod_{\mathfrak{U}} A_t$ is then simply the coordinate ring of the closed (and pro-open) subscheme $\pi^{-1}(\mathfrak{U}) \subset \mathrm{Spec}(\prod_{t \in T} A_t)$. 
Explicitly, the map 
carries a prime ideal $\mathfrak{p} \subset \prod_{t \in T} A_t$ to the
following ultrafilter on $T$: we consider all subsets $T' \subset T$ such that
the characteristic function $1_{T'} \in \prod_{t \in T} A_t$ does not belong to
$\mathfrak{p}$.\footnote{For more details, see
\url{https://math.stackexchange.com/questions/1533237/spectrum-of-mathbbz-mathbbn}.} 
\end{remark}

Next, given a set $\left\{A_t\right\}_{t \in T}$ of $R$-algebras, let us explain how certain functors defined on all rings give sheaves on $\mathcal{P}(T)$.

\begin{cons}
\label{varcons}
Let $\left\{A_t\right\}_{t \in T}$ be a  set of $R$-algebras. Let
$\mathcal{F}\colon  \mathrm{Ring}_R \to \mathcal{C}$ be a functor which preserves finite products and filtered colimits.  Then we obtain a functor (using Example~\ref{prodfunctor})
\[ f\colon  \mathcal{P}(T)^{op}  \to \mathcal{C} \]
sending 
\[ T' \subset T \mapsto \mathcal{F}( \prod_{T'} A_t) .  \]
The hypothesis that $\mathcal{F}$ preserves finite products implies that $f$ is a sheaf.  Moreover, for any ultrafilter $\mathfrak{U}$ on $T$, we have an identification $f_{\mathfrak{U}} \simeq \mathcal{F}( \prod_{\mathfrak{U}} A_t)$ since $\mathcal{F}$ commutes with filtered colimits. That is, the stalks of $f$ can be identified with the values of $\mathcal{F}$ on the ultraproducts. 

For future reference, we record a slight variant of the preceding paragraph. Fix an $A$-algebra $B$.  Then we obtain a functor 
$f_B\colon  \mathcal{P}(T)^{op} \to \mathcal{C}$ defined by 
\[ T' \subset T \mapsto \mathcal{F}( B \otimes_A \prod_{t \in T'} A_t ).  \]
Then $f_B$ is again a sheaf on $\mathcal{P}(T)$, and $(f_B)_{\mathfrak{U}} \simeq \mathcal{F}( B  \otimes_A \prod_{\mathfrak{U}} A_t)$.  That is, the stalk of $f_B$ at an ultrafilter $\mathfrak{U}$ is $\mathcal{F}$ applied to the base change of $B$ along the map from $A$ to the $\mathfrak{U}$-ultraproduct. 
\end{cons}

We can now prove the promised result, explaining how to control the behaviour on infinite products of certain functors defined on commutative rings.

\begin{corollary} 
\label{ultradetectsequivalence}
Suppose that $\mathcal{C}$ is a compactly generated, presentable $\infty$-category. 
Fix a base ring $R$. Let $\mathcal{F}_a, \mathcal{F}_b\colon  \mathrm{Ring}_R \to
\mathcal{C}$ be two functors which  commute with filtered colimits and finite
products. Suppose given a map $\mathcal{F}_a \to \mathcal{F}_b$ of functors.
Fix a set $\left\{A_t\right\}_{t \in T}$ of $R$-algebras.  Suppose that for each ultrafilter $\mathfrak{U}$ on $T$, the map $\mathcal{F}_a(\prod_{\mathfrak{U}} A_t ) \to \mathcal{F}_b( \prod_{\mathfrak{U}} A_t )$ is an equivalence.  Then the map $\mathcal{F}_a( \prod_T A_t) \to \mathcal{F}_b( \prod_T A_t )$ is an equivalence. 
\end{corollary} 
\begin{proof} 
This follows from Lemma~\ref{checkequivalence} in light of the above
constructions. 
\end{proof} 

\begin{corollary} 
\label{truncatedultra}
Fix a base ring $R$. 
Let $\sF\colon \mathrm{Ring}_R \to \mathcal{S}$ be a functor which commutes with
filtered colimits and finite products. 
Let $\left\{A_t\right\}_{t \in T}$ be a set of $T$-algebras. 
Suppose that for each ultrafilter $\mathfrak{U}$ on $T$, 
$\sF( \prod_{\mathfrak{U}} A_t)$ is $n$-truncated for some fixed $n$. 
Then $\sF( \prod_T A_t)$ is $n$-truncated. 
\end{corollary} 
\begin{proof} 
This also follows in view of the above constructions. 
Here we use that a sheaf of spaces $\mathcal{P}(T) \to \mathcal{S}$ with $n$-truncated
stalks at all ultrafilters $\mathfrak{U}$ is automatically $n$-truncated; 
this follows for instance by comparison with sheaves on the topological space
$\beta T$, which are automatically hypercomplete,
cf.~\cite[Cor.~7.2.1.12 and Th.~7.2.3.6]{HTT}.
\end{proof}

\subsection{Detection of universal $\mathcal{F}$-descent}
\label{DetectUnivDescent}
Throughout this section, we let $\mathcal{C}$ be an $\infty$-category
that is compactly generated by cotruncated objects. 
Let $\mathcal{F}\colon  \mathrm{Sch}_{qcqs,R}^{op} \to \mathcal{C}$ be a
functor. The goal of this section is prove a result (\Cref{ultraproductdetects}) that explains how to detect $\mathcal{F}$-descent properties of morphisms of schemes by base changing to ultraproducts instead of products. To formulate this precisely, it is convenient to make the following definition. 

\begin{definition}[Detection of universal $\mathcal{F}$-descent]
Let $X \in \mathrm{Sch}_{qcqs,R}$. 
Consider a family of maps $X_i \to X $ of qcqs $R$-schemes for $ i \in I$. We say that  
the family $\left\{X_i \to X \right\}_{i \in I}$
\emph{detects universal
$\mathcal{F}$}-descent if a map $f\colon  Y \to X$ in $\mathrm{Sch}_{qcqs,R}$ is of universal $\mathcal{F}$-descent
if and only if the base change $f_i\colon  Y \times_X X_i \to X_i$ is of universal
$\mathcal{F}$-descent for each $i \in I$. 
\end{definition}

\begin{example} 
\label{F1ex}
Suppose a map $Y \to X$ is of universal $\mathcal{F}$-descent. 
Then the map $Y \to X$ detects universal $\mathcal{F}$-descent. This follows in view of Lemma~\ref{Fsorite}. 
\end{example}

\begin{example} 
Suppose we have a family $\left\{X_i \to X\right\}_{i \in I}$ which detects
universal $\mathcal{F}$-descent. Suppose for each
$i \in I$, we have a set $J_i$ and a family of maps $\{Y_j \to X_i\}_{ j \in
J_i}$ which detects universal $\mathcal{F}$-descent. 
Then the family $\left\{Y_j \to X\right\}_{j \in J_i, \  \mathrm{some} \ i}$
detects universal $\mathcal{F}$-descent. 
\label{transitivityofdetect}
\end{example} 

\begin{example} 
\label{F2ex}
\label{detectsimpliesproductF}
Suppose that $\left\{\mathrm{Spec}(A_i)\to \mathrm{Spec}(A)\right\}_{i \in I}$ detects universal $\mathcal{F}$-descent. 
Then the map $f\colon  \mathrm{Spec}( \prod_{i \in I} A_i) \to \mathrm{Spec}(A)$ is of universal $\mathcal{F}$-descent. In fact,
the map admits a section after base change to $\mathrm{Spec}(A_j)$ for each $j$ and
each such base change is therefore of universal $\mathcal{F}$-descent. By
assumption, this implies that $f$ is of universal $\mathcal{F}$-descent. As a
partial converse, under mild hypotheses on $\mathcal{F}$, we shall show in
\Cref{ultraproductdetects} that if the singleton family $\{\mathrm{Spec}(\prod_i A_i) \to \mathrm{Spec}(A)\}$ detects universal $\mathcal{F}$-descent, then the family $\{\mathrm{Spec}(\prod_{\mathfrak{U}} A_i) \to \mathrm{Spec}(A)\}_{\mathfrak{U} \in \beta I}$ (obtained by localizing $\prod_i A_i$ at all possible ultrafilters on $I$) detects universal $\mathcal{F}$-descent. 
\end{example}

\begin{lemma}[Detection and filtered colimits] 
\label{detectfiltcolimit}
Let $\mathcal{C}$ be an $\infty$-category
that is compactly generated by cotruncated objects. 
Let $\mathcal{F}\colon  \mathrm{Sch}_{qcqs,R}^{op} \to \mathcal{C}$ be a finitary functor. 
Let $X \in \mathrm{Sch}_{qcqs,R}$ be a qcqs scheme and let $\left\{Y_i\right\}_{i \in I}$ be a family of
qcqs $X$-schemes. Suppose we can write $X$ as a filtered limit $X = \varprojlim_{j
\in J} X_j$ in $\mathrm{Sch}_{qcqs,R}$ with affine transition maps. Suppose that 
for each $j$, the family of maps $\{Y_i \to X_j\}_{i \in I}$ detects
universal $\mathcal{F}$-descent. 
Then the family of maps $\left\{Y_i \to X\right\}$ detects universal $\mathcal{F}$-descent. 
\end{lemma} 

\begin{proof} 
Let $X' \to X $ be a map. Suppose that each base change $Y_i \times_X X' \to Y_i$ is  of universal $\mathcal{F}$-descent. 
For each $i, j$, we have a factorization
\[  Y_i \times_X X' \to  
Y_i \times_{X_j} X' \to Y_i
\]
and since the composition is of universal $\mathcal{F}$-descent, it follows that 
$Y_i \times_{X_j} X' \to Y_i$ is of universal $\mathcal{F}$-descent (Lemma~\ref{Fsorite}). 
Since $\{Y_i \to X_j\}_{i \in I}$ detects universal $\mathcal{F}$-descent, we find that
each map 
$X' \to X_j$ is of universal $\mathcal{F}$-descent. Taking the limit over $j$,
we find that $X' \to X$ is of universal $\mathcal{F}$-descent
(Lemma~\ref{univFclosedfiltlimit}). 
\end{proof}

\begin{lemma} 
\label{ultraproductdetects}
Let $\mathcal{C}$ be 
an $\infty$-category that is compactly generated by cotruncated objects. 
Let $\mathcal{F}\colon  \mathrm{Ring}_R \to \mathcal{C}$ be a finitary functor which
preserves finite products. Let $\left\{A_t\right\}_{t \in T}$ be a set of $R$-algebras and let $A =
\prod_{t \in T} A_t$. 
Then the maps $\left\{A \to \prod_{\mathfrak{U}} A_t\right\}$,
where $\mathfrak{U}$
ranges over the ultrafilters on $T$, detect universal $\mathcal{F}$-descent. 
\end{lemma} 

\begin{proof} 

Let $B \to C$ be a map of $A$-algebras. 
Suppose that $B \otimes_A \prod_{\mathfrak{U}} A_t \to C \otimes_A
\prod_{\mathfrak{U}} A_t$ is of $\mathcal{F}$-descent for each $\mathfrak{U}$ on $T$. Then we claim that 
$B \to C$ is of $\mathcal{F}$-descent, which will imply the result. 
In fact, 
consider the augmented cosimplicial object of $\mathcal{C}$
\begin{equation}  \mathcal{F}(B) \to  \mathcal{F}(C) \rightrightarrows \mathcal{F}( C \otimes_B C)
\triplearrows \dots   \end{equation}
which we want to show to be a limit diagram. 
By Construction~\ref{varcons}, 
the augmented cosimplicial diagram upgrades to a diagram of sheaves (with values
in $\mathcal{C}$) on
$\mathcal{P}(T)$. Furthermore, by assumption, the diagram of
$\mathfrak{U}$-stalks is a limit diagram for each ultrafilter $\mathfrak{U}$. 
Since totalizations and filtered colimits commute in $\mathcal{C}$ by
assumption, it follows by Lemma~\ref{checkequivalence} that 
we have a limit diagram of sheaves, and thus $B \to C$ is of
$\mathcal{F}$-descent, as
desired. 
\end{proof}

\subsection{Detection in the $v$-topology}
\label{vSheafAICVR}

The goal of this section is to prove that finitary $v$-sheaves are controlled by their behaviour on (certain) valuation rings (Propositions~\ref{vcoverbyvaluation} and \ref{aicvaldetect}). To formulate these, we first recall the following basic definition:

\begin{definition} 
\label{aic}
A ring $R$ is called \emph{absolutely integrally closed} if every monic polynomial in $R[x]$ admits a root in $R$. 
\end{definition}

We will only use this definition in the case where $R$ is an integral domain, in
which case it is equivalent to the condition that $\mathrm{Frac}(R)$ is
algebraically closed and that $R$ is normal. Note that the class of absolutely
integrally closed domains is preserved 
by localizations and quotients by prime ideals. 

\begin{lemma} 
\label{ultraaic}
An ultraproduct of absolutely integrally closed valuation rings is an absolutely integrally closed valuation ring. 
\end{lemma} 
\begin{proof} 
The condition that a ring should be an absolutely integrally closed valuation ring is a first-order
property in the language of commutative rings. Therefore, the result follows
from \strokeL o\'s's  theorem (cf.~\cite[Th.~4.1.9]{CKmodel} for an account).  

Alternately, one can argue directly as in \cite[Lemma 6.2]{BhattScholzeWitt} to show that an ultraproduct $\prod_{\mathfrak{U}} V_t$ of a collection $\{V_t\}_{t \in T}$ of valuation rings is a valuation ring. If each $V_t$ is absolutely integrally closed, the same holds for $\prod_{t \in T} V_t$, and thus also for any localization such as $\prod_{\mathfrak{U}} V_t$.
\end{proof}

Our next goal is to show that every ring admits a $v$-cover by a product of
absolutely integrally closed valuation rings
(Proposition~\ref{vcoverbyvaluation}); this is a variant of \cite[Lemma
6.2]{BhattScholzeWitt}, and is closely related to the fact that absolutely
integrally closed valuation rings give a conservative system of points for the
$h$-topology (at least for noetherian rings, see \cite[Prop.~2.2, Cor.~3.8]{GooLi} and \cite{GabberKelly}).  To get there, let us first study absolute integral closures of valuation rings.

\begin{lemma} 
\label{aicvcover}
Let $V$ be a valuation ring. Then there exists an absolutely integrally closed valuation ring $V'$ and a
map $V \to V'$ which is faithfully flat (hence a $v$-cover). 
\end{lemma} 
\begin{proof} 
Consider  
an extension of the valuation to $\overline{\mathrm{Frac}(V)}$, and take $V'$
to be the
valuation ring associated to that valuation. 
By Lemma~\ref{aicflatlemma}, 
we see that 
$V'$ satisfies the desired claims. 
\end{proof}

To avoid set-theoretic inconsistencies, we bound the size of the absolutely integrally closed valuation rings required to probe a given ring. 

\begin{lemma} 
\label{boundcardaic}
Let $A$ be an arbitrary commutative ring and let $\kappa = \max(
\mathrm{card}(A), \aleph_0)$. 
Then for every absolutely integrally closed valuation ring $W$ with a map $A \to W$, there exists a
factorization $A \to W' \to W$ such that: 
\begin{enumerate}
\item $W'$ is an absolutely integrally closed valuation ring with $\mathrm{card}(W') \leq \kappa$.  
\item  $W' \to W$ is faithfully flat. 
\end{enumerate}
\end{lemma} 
\begin{proof} 
Fix a map $f\colon  A \to W$ and consider the image $\mathrm{im}(f) \subset W$. 
Let $ K_0$ be the fraction field of $\mathrm{im}(f)$ inside $K :=
\mathrm{Frac}(W)$, and let $ K_1 \subset K$ be the algebraic closure of $K_0$. 
Then $\mathrm{card}(K_1) \leq \kappa$, so $W' := K_1  \cap W$ has cardinality at
most $\kappa$. The rest of the result follows from
Lemma~\ref{aicflatlemma}. 
\end{proof}

We can now prove the promised result.

\begin{proposition} 
\label{vcoverbyvaluation}
Let $\mathcal{C}$ be an $\infty$-category that is compactly generated by cotruncated objects. 
Let $\mathcal{F}\colon  \mathrm{Sch}_{qcqs,R}^{op} \to \mathcal{C}$ be a finitary $v$-sheaf. Let $A$ be an $R$-algebra. Then there exists a set of maps $\{A \to V_i\}_{i \in S}$
such that: 
\begin{enumerate}
\item Each $V_i$ is an absolutely integrally closed valuation ring.  
\item The maps $\left\{A \to V_i\right\}_{i \in S}$ detect universal $\mathcal{F}$-descent. 
\item The map $A \to \prod_{i \in S} V_i$ is a $v$-cover. 
\end{enumerate}
\end{proposition} 
\begin{proof} 
Let $\kappa
= \max( \mathrm{card}(A), \aleph_0)$. 
Let $\{W_t\}_{t \in T}$ be  a set of representatives of isomorphism classes of absolutely integrally closed
valuation rings of cardinality $\leq \kappa$ receiving maps $A \to W_t$. 
We first claim that the map $A \to A' :=\prod_{t \in T} W_t$ is a $v$-cover. 
This follows from Lemma \ref{aicvcover} and Lemma \ref{boundcardaic}. 
Explicitly, if $A \to V$ is a map to any valuation ring $V$, we can enlarge $V$
and assume $V$ absolutely integrally closed (by Lemma~\ref{aicvcover}). Then the map $A \to V$
actually factors through $A \to A'$ (by Lemma~\ref{boundcardaic});
together, these imply that $A\to A'$ is a $v$-cover. 

This produces a collection of absolutely integrally closed valuation rings satisfying (3), but we have not yet shown
detection of universal $\sF$-descent; for this we enlarge the family further. 
We construct the family $\left\{V_i\right\}$ of $A$-algebras as the
collection of ultraproducts of the $W_t$. 
For each ultrafilter $\mathfrak{U}$ on $T$, we consider the $A'$-algebra
$A'_{\mathfrak{U}} := \prod_{\mathfrak{U}} W_t$. Then each $A'_{\mathfrak{U}}$
is an absolutely integrally closed valuation ring (Lemma~\ref{ultraaic}). Moreover the family of
maps $\left\{A' \to A'_{\mathfrak{U}}\right\}$ (as $\mathfrak{U}$ ranges
over all ultrafilters on $T$) detects universal $\mathcal{F}$-descent
thanks to \Cref{ultraproductdetects}. 
The assertion (3) now follows because $A \to A'$ is a $v$-cover,
as verified in the previous paragraph. 
\end{proof}

\begin{remark}
Proposition~\ref{boundcardaic} and Lemma~\ref{ultraaic} imply that for every
commutative ring $A$, there is a $v$-cover $A \to B$ such that each connected
component of $B$ is a valuation ring. By contrast, already for $A=k[x,y]$ over a field $k$, there does not exist an $\arc$-cover $A \to B$ such that that every connected component of $B$ is a rank $\leq 1$ valuation ring. Indeed, if such an $\arc$-cover $A \to B$ existed, then this map would be a $v$-cover by Proposition~\ref{Noetherianv}, so every valuation on $A$ would extend to $B$, which is impossible: $A$ admits a rank $2$ valuation, while every valuation on $B$ has rank $\leq 1$ by the assumption on connected components.
\end{remark}

Using the preceding constructions, we can control finitary $v$-sheaves in terms of their behaviour on absolutely integrally closed valuation rings.

\begin{proposition} 
\label{aicvaldetect}
Let $\mathcal{C}$ be an $\infty$-category that is compactly generated by cotruncated objects. 
Let $\mathcal{F}, \mathcal{G}\colon  \mathrm{Sch}_{qcqs,R}^{op} \to \mathcal{C}$ be  finitary
$v$-sheaves and fix a map $f\colon  \mathcal{F} \to \mathcal{G}$. 
Then $f$ is an equivalence of $v$-sheaves if and only if 
for every absolutely integrally closed valuation ring $V \in
\mathrm{Ring}_R$, $\mathcal{F}(\mathrm{Spec}(V)) \to \mathcal{G}( \spec(V)) $
is an equivalence.  
\end{proposition} 
\begin{proof} 
Suppose that for every 
absolutely integrally closed valuation ring $V \in
\mathrm{Ring}_R$, the map $\mathcal{F}(\mathrm{Spec}(V)) \to \mathcal{G}( \spec(V)) $ is
an equivalence. 
By Corollary~\ref{ultradetectsequivalence} and \Cref{ultraaic}, we have
$\mathcal{F}(\mathrm{Spec}(S)) \xrightarrow{\sim} \mathcal{F}( \spec(S))$
whenever $S$ is a product of absolutely integrally closed valuation rings with
the structure of $R$-algebra. 
Now for any $R$-algebra $A$, we have a $v$-cover $\spec(S) \to \spec(A)$ for $S$ an appropriate
product of absolutely integrally closed valuation rings, by Proposition~\ref{vcoverbyvaluation}. 
Continuing, we can construct a 
$v$-hypercover of $\spec(A)$ where all of the terms are spectra of products of absolutely integrally closed valuation
rings. But we have $v$-hyperdescent of $\mathcal{F}, \mathcal{G}$: this can be checked after
applying $\hom_{\mathcal{C}}(x, \cdot)$ for $x \in \mathcal{C}$ compact, and
then the result  follows since sheaves of $n$-truncated spaces are automatically
hypercomplete (for any $n$), cf.~\cite[Sec.~6.5.3]{HTT} for a treatment. \end{proof}

\newpage
\section{The main result for $\arc$-descent}

In this section, we prove our main result, explaining the relationship of
$\arc$-descent to excision. Fix a base ring $R$.
\begin{theorem}[Equivalence of $\arc$-descent and excision]
\label{mainthm}
Let $\mathcal{C}$ be an $\infty$-category that is compactly generated by cotruncated objects. 
Let $\mathcal{F}\colon  \mathrm{Sch}_{qcqs,R}^{op} \to \mathcal{C}$ be a finitary functor (see Definition~\ref{DefFinitary}) which satisfies $v$-descent. Then the following are equivalent: 
\begin{enumerate}
\item $\mathcal{F}$ satisfies $\arc$-descent. 
\item $\mathcal{F}$ satisfies excision. 
\item $\mathcal{F}$ satisfies {\em aic-$v$-excision}, i.e., for every
absolutely integrally closed valuation ring $V$ with the structure of an $R$-algebra  and prime ideal
$\mathfrak{p} \subset V$, the square
in $\mathcal{C}$
\begin{equation} \label{faic}  \begin{aligned}\xymatrix{
\mathcal{F}(V) \ar[d]  \ar[r] &  \mathcal{F}( V/\mathfrak{p}) \ar[d]  \\
\mathcal{F}(V_{\mathfrak{p}}) \ar[r] &  \mathcal{F}(\kappa(\mathfrak{p}))
} \end{aligned} \end{equation}
is  cartesian.  (Recall that the corresponding square of rings is a Milnor square, cf. Proposition~\ref{AICValRingPrime}.)
\end{enumerate}
\end{theorem} 

Clearly (2) implies (3). We will structure the proof of the theorem as follows. In subsection~\ref{aicexcsec}, we show that (1) and (3) are equivalent. In subsection~\ref{excgeneral}, we show that (3) implies (2). Together these will complete the proof. 
Note that the condition of $v$-excision has been considered by Huber--Kelly
\cite{HK18}.

\subsection{Aic-$v$-excision and $\arc$-descent}
\label{aicexcsec}
In this subsection, we will prove that (1) and (3) in Theorem~\ref{mainthm} are equivalent. First we show that (1) implies (3). 
Let $\mathcal{F}$ be as in Theorem~\ref{mainthm}. 
\begin{proposition} 
\label{1implies3}
Suppose that $\mathcal{F}$ satisfies $\arc$-descent. Then $\mathcal{F}$
satisfies aic-$v$-excision. 
\end{proposition} 
\begin{proof} 
Let $V$ be an absolutely integrally closed valuation ring that is an
$R$-algebra, and let
$\mathfrak{p} \in \spec(V)$. 
We consider the map 
$V \to \widetilde{V} := V_{\mathfrak{p}} \times V/\mathfrak{p}$. 
By our assumptions and Corollary~\ref{v1val}, $\mathcal{F}$
satisfies descent for this morphism. 
Unwinding the definitions, and using the assumption that $\mathcal{F}$ preserves finite
products, we see that 
the cosimplicial object $\mathcal{F}( \widetilde{V}^{\otimes \bullet + 1})$ computes
precisely $\mathcal{F}( V_{\mathfrak{p}}) \times_{\mathcal{F}( \kappa( \mathfrak{p}))} \mathcal{F}(
V/\mathfrak{p})$. In particular, the statement that $\mathcal{F}$ satisfies descent
for $V \to \widetilde{V}$ is equivalent 
to the statement that $\mathcal{F}$ satisfies excision for the excision datum 
$(V, \mathfrak{p}) \to (V_{\mathfrak{p}}, \mathfrak{p} V_{\mathfrak{p}})$. 
\end{proof}

We will next show that for a $v$-sheaf which satisfies aic-$v$-excision, there are
enough maps to rank $\leq 1$ valuation rings which detect universal
$\mathcal{F}$-descent. 
To proceed, we will use the following lemma about maps between absolutely integrally closed valuation
rings. 

\begin{lemma}
\label{MilnorAICValRing}
Let $V \to W$ be a map between absolutely integrally closed valuation rings.
Fix a prime $\mathfrak{p} \subset V$. Suppose that $\mathfrak{p} W \neq W$.
Then:
\begin{enumerate}
\item The ideal $\mathfrak{p}W \subset W$ is prime and pulls back to
$\mathfrak{p}$ in $V$. 
\item There exists a largest prime $\mathfrak{q} \subset W$ whose pullback to $V$ equals $\mathfrak{p}$. For such $\mathfrak{q}$, we get $W \otimes_V V_{\mathfrak{p}} \simeq W_{\mathfrak{q}}$ and $W \otimes_V \kappa(\mathfrak{p}) \simeq W_{\mathfrak{q}}/\mathfrak{p}W_{\mathfrak{q}}$. 
\end{enumerate}
\end{lemma}

As the proof below shows, we do not need the full strength of absolute integral closedness above: it suffices to assume that the value groups of $V$ and $W$ are $n$-divisible for some integer $n > 1$. 

\begin{proof}
By replacing $V$ by its image, we may assume that $V$ is a subring of $W$. 

For the first assertion of (1), since radical ideals in a valuation ring are prime, it is enough to
show that $\mathfrak{p}W$ is radical. Say $y \in W$ with $y^2 \in
\mathfrak{p}W$. Then we can write $y^2 = \sum_{i=1}^n a_i x_i$ with $a_i \in
\mathfrak{p}$ and $x_i \in W$. As $V$ is a valuation ring, we may assume after
rearrangement that $a_1 \mid a_i$ for all $i$, so we can rewrite $y^2 = a_1
\cdot z$ for some $z \in W$. As both $V$ and $W$ have $2$-divisible value
groups, we can choose a $c_1 \in \mathfrak{p}$ such that $c_1^2V = a_1V$ and $w
\in W$ such that $w^2W = zW$. It follows that $y^2 W = c_1^2 w^2W$, and hence
$yW = c_1w W$ as $W$ is a valuation ring. But this implies that $y \in c_1W \subset \mathfrak{p}W$, as wanted.

Next we verify the second assertion of (1), i.e., that $\mathfrak{p} W \subset W$ pulls back to $\mathfrak{p}
\subset V$. 
Suppose that  there is an element $x \in V\setminus \mathfrak{p}$ such that $x \in \mathfrak{p}
W$, so $x = \sum y_i z_i$ for some $y_i \in \mathfrak{p}, z_i \in W$. Using
divisibility to collect terms, we may assume that there is only one term in the
sum, i.e., $x = yz$ for $y \in \mathfrak{p}, z \in W$. 
Since $V$ is a valuation ring,  we can write $y = x y'$ for some
$y'  \in \mathfrak{p}$. 
We get $x = xy' z$ in $W$, and canceling we find that $y' z = 1$ in $W$. This
means that $\mathfrak{p} W = W$, a contradiction.

For the existence of $\mathfrak{q}$  in (2): note that the collection of primes
$\mathfrak{q}' \subset W$ whose pullback to $V$ equals $\mathfrak{p}$ is
directed: as $W$ is a valuation ring, any set of prime ideals of $W$ is even
totally ordered. The existence of $\mathfrak{q}$ follows as a filtered colimit
of prime ideals is prime (and because forming filtered colimits of primes in $W$
commutes with pulling back to $V$). Having constructed $\mathfrak{q}$, we
immediately get a map $W \otimes_V V_{\mathfrak{p}} \to W_{\mathfrak{q}}$. This
is an injective map of two valuation rings (as both sides are localizations of
$W$). To prove bijectivity, it is thus enough to show that their spectra match
up. But $\mathrm{Spec}(W \otimes_V V_{\mathfrak{p}})$ is the preimage of
$\mathrm{Spec}(V_{\mathfrak{p}}) \subset \mathrm{Spec}(V)$, and hence is
identified with the collection of primes of $W$ whose pullback to $V$
lies in $\mathfrak{p}$. As $\mathfrak{q}$ was the maximal element in this collection, it follows that $W \otimes_V V_{\mathfrak{p}} \to W_{\mathfrak{q}}$ is a bijection on $\mathrm{Spec}(-)$ and thus an isomorphism. The final assertion follows by comparing residue fields at the closed point in the preceding isomorphism.
\end{proof}

We can now prove a crucial stability property of aic-$v$-excision: it passes up to algebras. 

\begin{lemma} 
\label{excifaicexciv}
Let $V$ be an absolutely integrally closed valuation ring, and fix a
$V$-algebra $A$. 
Let $\mathcal{F}\colon  \mathrm{Sch}_{qcqs,V}^{op} \to \mathcal{C}$ be a finitary
$v$-sheaf  which satisfies aic-$v$-excision, where $\mathcal{C}$ is compactly generated by cotruncated objects. Then for every 
prime ideal $\mathfrak{p} \subset V$, the map
\[ \mathcal{F}( A) \to \mathcal{F}( A \otimes_V V_{\mathfrak{p}})
\times_{\mathcal{F}( A \otimes_V \kappa(\mathfrak{p}))}\mathcal{F}( A/\mathfrak{p})\]
 is an equivalence. 
\end{lemma} 
\begin{proof} 
Since both sides are finitary $v$-sheaves in $A$, it suffices (by Proposition~\ref{aicvaldetect}) to show that  the map is an equivalence for $A = W $ an absolutely integrally closed valuation ring. 

If $\mathfrak{p} W  = W$, then $W \simeq W \otimes_V V_{\mathfrak{p}}$ and
$W/\mathfrak{p}W = 0$, so the result is clear.  If $\mathfrak{p} W \neq W$,
then \Cref{MilnorAICValRing} shows that $\mathfrak{p} W$ is a prime ideal  in $W$.  Since $\mathfrak{p} W$ pulls back to $\mathfrak{p}$, we obtain a natural map $V_\mathfrak{p} \to W_{\mathfrak{p}W}$.  We consider the diagram
\[ \xymatrix{
W \ar[d]  \ar[r] &  W/\mathfrak{p}W \ar[d] \\
W \otimes_V V_{\mathfrak{p}} \ar[r] \ar[d]  &  W \otimes_V \kappa( \mathfrak{p}) \ar[d]  \\
W_{\mathfrak{p}W} \ar[r] &   \kappa( \mathfrak{p} W),
}\]
where all rings are absolutely integrally closed valuation rings. 
There are three squares we can extract naturally from this diagram. 
The outer one is a Milnor square as in \eqref{faic} (i.e., belongs to the
setting of aic-$v$-excision), and so is the
bottom square (Proposition~\ref{AICValRingPrime}). 
Therefore, $\mathcal{F}$ carries the outer and bottom squares to  pullback
squares. It follows that $\mathcal{F}$ carries the top square to a  pullback,
i.e., $\mathcal{F}(W) \simeq \mathcal{F}( W/\mathfrak{p}W) \times_{\mathcal{F}( W \otimes_V
\kappa( \mathfrak{p}))} \mathcal{F}( W \otimes_V V_{\mathfrak{p}})$. 
\end{proof}

\begin{lemma} 
\label{Descinvlimv}
Let $\{A_i \to B_i\}_{i \in I}$ be a diagram of finitely presented ring maps indexed by a filtered category $I$. Assume that for each map $i \to j$ in $I$, the map $A_j \otimes_{A_i} B_i \to B_j$ is an isomorphism.  If $\varinjlim_I A_i \to \varinjlim_I B_i$ is a   $v$-cover, then so is $A_j \to B_j$ for some $j \in I$ (and thus for all $j$  large enough). 
\end{lemma} 

\begin{proof} 
Let $A_\infty = \varinjlim_I A_i, B_\infty = \varinjlim_I B_i$. The map
$\spec(B_\infty) \to \spec(A_\infty)$ is finitely presented and
a $v$-cover and therefore \cite[Theorem 3.12]{Rydh} admits a refinement which factors as a composite of a quasi-compact open
covering and a proper finitely presented surjection. By general results of
noetherian approximation \cite[Tag 09MV]{stacks-project}, some map $\spec(B_i)
\to \spec(A_i)$ admits a refinement which factors as a composite of a quasi-compact open covering and a
proper finitely presented surjection, and is therefore a $v$-cover.  
\end{proof} 

To proceed further, we will need a general lemma which will aid us in reducing 
valuation rings to rank $\leq 1$ valuation rings. 
Let $V$ be a valuation ring. 
Consider the totally ordered set $\spec(V)$ of prime ideals of $V$. 

\begin{definition}[Intervals] 
An \emph{interval} $I = [\mathfrak{p}, \mathfrak{q}]$ (for $\mathfrak{p}
\subset \mathfrak{q}$ prime ideals in $V$) will denote the set of
prime ideals of $V$ which 
are contained between $\mathfrak{p}$ and $\mathfrak{q}$; note that this is also
$\spec( (V/\mathfrak{p})_{\mathfrak{q}})$, so to each interval $I$ we have an
associated absolutely integrally closed valuation ring $V_I  :=
(V/\mathfrak{p})_{\mathfrak{q}}$.  

The collection of intervals of $\spec(V)$ is partially ordered under inclusion:
we have $[\mathfrak{p}, \mathfrak{q}] \subset [ \mathfrak{p}', \mathfrak{q}']$
if and only if $\mathfrak{p} \supset \mathfrak{p}'$ and $\mathfrak{q} \subset
\mathfrak{q}'$ as prime ideals.
Given an inclusion of intervals $I \subset J$, we have a 
map $V_J \to V_I$, so we have a contravariant functor from the poset of intervals
to absolutely integrally closed valuation rings. 
Finally, any chain $\mathcal{C}$ (i.e., totally ordered subset) in the
poset of intervals admits a supremum and an infimum. 
The infimum is 
given by the intersection, and the functor $I \mapsto V_I$ sends an
intersection of  intervals along a chain to the associated filtered colimit
of the $V_I, I \in \mathcal{C}$. 
\end{definition}

\begin{lemma} 
\label{intervallemma}
Let $V$ be a valuation ring. 
Let $\mathcal{P}$ be a collection of intervals in $\spec(V)$ such that: 
\begin{enumerate}
\item If $I$ is an interval of \emph{length $\leq 1$} (so $I$ either consists
of one point or $I = [\mathfrak{p}, \mathfrak{q}]$ where $\mathfrak{p}
\subsetneq \mathfrak{q}$ is an inclusion that cannot be refined further), then
$I \in \mathcal{P}$.\footnote{Strictly speaking, (1) is a consequence of (3) and (4),
but it is convenient to state it separately.} 
\item
If $I \in \mathcal{P}$ and $J \subset I$ is a subinterval, then $ J \in
\mathcal{P}$ as well. 
\item Suppose $\mathfrak{p} \in \spec(V)$ is not maximal. Then there exists 
$\mathfrak{q} \supsetneq \mathfrak{p}$ with $[\mathfrak{p}, \mathfrak{q}] \in
\mathcal{P}$. 
\item Suppose $\mathfrak{p} \in \spec(V)$ is nonzero. Then there exists 
$\mathfrak{r} \subsetneq \mathfrak{p}$ with $[\mathfrak{r}, \mathfrak{p}] \in
\mathcal{P}$. 

\item 
Suppose that $I, J$ are overlapping intervals in $\mathcal{P}$, so that $I \cup J$ is an
interval. Then $I \cup J \in \mathcal{P}$.  
\end{enumerate}
Then $\mathcal{P}$ consists of all intervals in $\spec(V)$. 
\end{lemma} 
\begin{proof} 
By (2), it suffices to show that the interval $\spec(V)$ belongs to $\mathcal{P}$. 
The proof is a straightforward Zorn's lemma argument. 
Suppose $\mathcal{D}$ is a chain of intervals $[\mathfrak{p}_i,
\mathfrak{q}_i]$ in $\mathcal{P}$. 
We let $\mathfrak{p} = \bigcap \mathfrak{p}_i$ and $\mathfrak{q} = \bigcup
\mathfrak{q}_i$. We want to show that $[\mathfrak{p}, \mathfrak{q}]$ is also
in $\mathcal{P}$, so $\mathcal{D}$ has an upper bound. 
By assumption and (2) any interval $[\mathfrak{p}',\mathfrak{q}']$ that is \emph{properly} contained in $[\mathfrak{p},
\mathfrak{q}]$ belongs to $\mathcal{P}$: if $\mathfrak{q}' \subsetneq \cup_i \mathfrak{q}_i$, then $\mathfrak{q}' \subset \mathfrak{q}_i$ for $i \gg 0$ (for otherwise we would have $\mathfrak{q}_i \subset \mathfrak{q}'$ for all $i$, whence $\cup_i \mathfrak{q}_i \subset \mathfrak{q}'$, contradicting our assumption), and similarly at the other end.  Now  there exists an interval $I_0 \in \mathcal{P}$ containing $\mathfrak{p}$ and
a strictly larger prime ideal $\mathfrak{p}'$, and an interval $I_1 \in \mathcal{P}$ containing $\mathfrak{q}$
and a strictly smaller prime ideal $\mathfrak{q}'$ by (3, 4). 
But then $I_2 := [\mathfrak{p}',\mathfrak{q}']$ is properly contained in $[\mathfrak{p},\mathfrak{q}]$, so $I_2 \in \mathcal{P}$. Since $I_2 \cap I_j \neq \emptyset$ for $j={0,1}$, we conclude using (5) that $I_0 \cup I_1 \cup I_2 = [\mathfrak{p}, \mathfrak{q}] \in \mathcal{P}$, so $[\mathfrak{p}, \mathfrak{q}] \in
\mathcal{P}$. 

Thus the collection of intervals in $\mathcal{P}$ contains a maximal element by Zorn's
lemma. 
Let $[\mathfrak{p}_\infty, \mathfrak{q}_\infty]$ be such a maximal element; we
claim that $\mathfrak{p}_\infty = 0$ and $\mathfrak{q}_\infty $ is the maximal
ideal. 
Suppose the former fails. 
Then, by (4), we can find $\mathfrak{p}' \subsetneq \mathfrak{p}$ with 
$[\mathfrak{p}' , \mathfrak{p}] \in \mathcal{P}$. 
Using (5), we find that $[\mathfrak{p}', \mathfrak{q}_\infty] \in \mathcal{P}$,
contradicting maximality. 
Thus, we find that $\mathfrak{p}_\infty = 0$. Similarly, $ \mathfrak{q}_\infty 
$ is the maximal ideal, whence the result. 
\end{proof} 

\begin{proposition} 
\label{aicvexcimpliesarc}
Let $\mathcal{C}$ be an $\infty$-category that is compactly generated by cotruncated objects.
Let $\mathcal{F}\colon  \mathrm{Sch}_{qcqs,R}^{op}\to \mathcal{C}$ be a finitary
$v$-sheaf which satisfies aic-$v$-excision. 
Then $\sF$ is an $\arc$-sheaf. 
\end{proposition} 

\begin{proof} 
Let $f\colon  Y \to X$ be an $\arc$-cover of $R$-schemes. We want to show that it is
of universal $\sF$-descent. Without loss of generality, we can assume $X, Y$ are
affine. 
By Proposition~\ref{vcoverbyvaluation}, we 
can further reduce to the case where $X = \spec(V)$ is the spectrum of an
absolutely integrally closed valuation ring $V$. Let $Y = \spec(A)$, for a
$V$-algebra $A$. 

Now $A$ is a filtered colimit of finitely presented $V$-algebras, each of
which is also an $\arc$-cover of $V$ (\Cref{compositemaparccover}). Since morphisms of universal
$\sF$-descent are closed under filtered colimits of rings
(\Cref{univFclosedfiltlimit}), it suffices to
assume that $A$ is finitely presented over $V$.  

We consider now those intervals $I$ such that the map $V_I \to A \otimes_{V}
V_I$ (a finitely presented map of $R$-algebras) is of universal 
$\sF$-descent. We let $\mathcal{P}$ be the collection of such intervals; our goal is to show that
the interval $\spec(V)$ belongs to $\mathcal{P}$. 
To check this, we verify the hypotheses of \Cref{intervallemma} in order. 

\begin{enumerate}
\item If $I$ is an interval of \emph{length $\leq 1$},  then
$V_I$ is a rank $\leq 1$ valuation ring. Therefore,  
for such $I$, the map $V_I \to A \otimes_V V_I$ is actually a $v$-cover 
(cf.~\Cref{v1detectionr1})
and
therefore of universal $\sF$-descent. 
Thus, $I \in \mathcal{P}$. 
\item  If $I \in \mathcal{P}$ and $J \subset I$, then $J \in \mathcal{P}$. This is evident since morphisms of universal $\sF$-descent are stable under
base change.

\item
Suppose that $\mathfrak{p} \in \spec(V)$ is not maximal. Then we claim that there exists 
$\mathfrak{q} \supsetneq \mathfrak{p}$ such that $[\mathfrak{p}, \mathfrak{q}]$
belongs to $\mathcal{P}$. 
Indeed, if $\mathfrak{p}$ has an immediate successor, then we can take that
successor as $\mathfrak{q}$, in view of (1) above. 
Suppose $\mathfrak{p}$ has no immediate successor. 
Then the interval $\left\{\mathfrak{p}\right\}$ is the intersection of the
intervals $[\mathfrak{p}, \mathfrak{q}']$ for $\mathfrak{q}' \supsetneq
\mathfrak{p}$. 
As $I \mapsto V_I$ carries such intersection to filtered colimits, we learn that 
$$\kappa( \mathfrak{p}) = \varinjlim_{I = [\mathfrak{p}, \mathfrak{q}'],
\mathfrak{q}' \supsetneq \mathfrak{p}} V_I.$$ 

Now $\kappa( \mathfrak{p}) \to A \otimes_V \kappa( \mathfrak{p}) $ is  a 
$v$-cover and is finitely presented. 
Therefore, by Lemma~\ref{Descinvlimv}, there exists $I = [\mathfrak{p}, \mathfrak{q}]$ for
$\mathfrak{q} \supsetneq \mathfrak{p}$ such that $V_I \to A \otimes_V V_I$ is a
$v$-cover and hence of universal $\sF$-descent. Thus $I \in \mathcal{P}$. 

\item 
Suppose that $\mathfrak{q} \in \spec(V)$ is nonzero. Then there exists 
$\mathfrak{p} \subsetneq \mathfrak{q}$ such that $[\mathfrak{p}, \mathfrak{q}]$
belongs to $\mathcal{P}$. This is proved similarly.

\item 
Suppose that $I, J \in \mathcal{P}$ are overlapping intervals. Then the interval $I
\cup J$ belongs to $\mathcal{P}$. 
To see this, we may first assume without loss of generality that $I \cup J =
\spec(V)$, by base change. Suppose $I = [0, \mathfrak{p}], J = [\mathfrak{q},
\mathfrak{m}]$ for $\mathfrak{p} \supset \mathfrak{q}$ and $\mathfrak{m}$ the
maximal ideal. We claim that for any $V$-algebra $B$, the diagram
\begin{equation} \label{VIdiag1} \begin{aligned}\xymatrix{
\sF( B ) \ar[d] \ar[r] &  \sF( B \otimes_V V_I) \ar[d]  \\
\sF( B \otimes_V V_J) \ar[r] &  \sF( B \otimes_V V_{I \cap J})
} \end{aligned} \end{equation}
is cartesian. 

If $I \cap J$ is a single point, 
then this assertion follows from Lemma~\ref{excifaicexciv}. 
To reduce to this case, 
let $J' = [\mathfrak{p}, \mathfrak{m}]$ so that $J' \subset J$ and $I \cap J'
= \left\{\mathfrak{p}\right\}$. 
We 
can extend the diagram to 
\begin{equation} \label{VIdiag2} \begin{aligned} \xymatrix{
\sF( B \otimes_V V_J) \ar[r] \ar[d]  &  \sF( B \otimes_V V_{I \cap J}) \ar[d]  \\
\sF( B \otimes_V V_{J'}) \ar[r] &  \sF( B \otimes_V V_{I \cap J'}). 
}\end{aligned} \end{equation}
This square is cartesian, since we can consider the intervals $J', I \cap J$
whose union is $J$ and which intersect at a single point. 
Moreover, if we paste together the diagrams \eqref{VIdiag1}, \eqref{VIdiag2}
the outer square is cartesian (via the intervals $J', I$). Combining, we get
that \eqref{VIdiag1} is cartesian. 

Since \eqref{VIdiag1} is cartesian, it follows easily that the collection of maps $V \to V_I,
V \to V_J, V \to  V_{I \cap J}$ detect universal $\sF$-descent. This proves the claim. 
\end{enumerate}

Now observations (1) through (5) and \Cref{intervallemma} imply
that $\spec(V) \in \mathcal{P}$, and hence $A \to V$ is of universal
$\sF$-descent as desired. 
\end{proof}

\begin{proof}[{Proof that (1) is equivalent to (3) in Theorem~\ref{mainthm}}] 
We already saw that (1) implies (3) in 
Proposition~\ref{1implies3}.
In Proposition~\ref{aicvexcimpliesarc}, we saw (3) implies (1). 
\end{proof} 

For future reference we record the following corollary
(\Cref{detectequivfinitaryarc}): one can test
equivalences of finitary $\arc$-sheaves on absolutely integrally closed valuation rings of rank $\leq
1$. 
This is a slight strengthening of Proposition~\ref{aicvaldetect} in this
case. To begin with, we need two lemmas. 

\begin{lemma} 
Let $V$ be an absolutely integrally closed valuation ring. 
Let $\sF \colon \mathrm{Ring}_V \to \mathcal{S}$ be a finitary $\arc$-sheaf. 
For each interval $I \subset \spec(V)$, suppose that: 
\begin{enumerate}
\item  The space $\sF(V_I)$ is discrete. 
\item If furthermore $I$ has length $\leq 1$, then 
$\sF(V_I)$ is contractible. 
\end{enumerate}
Then $\sF(V)$ is contractible. 
\label{xequalsy}
\end{lemma} 
\begin{proof} 
By assumption, $\sF(V)$ is discrete. 
First, we show that $\pi_0(\sF(V))$ consists of at most one element. 
Let $x, y \in \pi_0(\sF(V))$. 
Consider the collection $\mathcal{P}$ of intervals $I \subset \spec(V)$ such that the images of $x,y $ in
$\pi_0( \sF(V_I))$ agree. 
Since $\sF$ is a finitary $\arc$-sheaf, we find that if $I, J \subset \spec(V)$
are \emph{overlapping} intervals, then the natural square
\[ \xymatrix{
\sF(V_{I \cup J}) \ar[d]  \ar[r] &  \sF(V_J) \ar[d]  \\
\sF(V_I) \ar[r] &  \sF(V_{I \cap J})
}\]
is cartesian (e.g., because $V_{I \cup J} \to V_I \times V_J$ is an
$\arc$-cover). 
By our discreteness hypotheses, this is simply a pullback diagram of sets. 
Consequently, our assumptions imply that the collection $\mathcal{P}$
satisfies the conditions of
\Cref{intervallemma} and hence consists of all intervals
$I$. In particular, we can take $I = \spec(V)$ as desired. 

Applying the above argument to each $V_I$ replacing $V$, we find that 
$\pi_0(\sF(V_I))$
has cardinality $\leq 1$ for all intervals $I \subset \spec(V)$. 
Consider the collection of intervals $I \subset \spec(V)$ such that
$\sF(V_I)$ is nonempty (and hence contractible). 
By what we have shown, 
this collection of intervals also satisfies the conditions of
\Cref{intervallemma} and therefore consists of all intervals, as desired. 
\end{proof} 

\begin{lemma} 
Let $R$ be a ring. 
Let $\sF \colon \mathrm{Ring}_R \to \mathcal{S}_{\leq n}$ be a finitary $\arc$-sheaf. 
Suppose that 
for every absolutely integrally closed rank $\leq 1$ valuation ring $W$ with the
structure of $R$-algebra, the space $\sF(W)$ is $k$-truncated for some $-1 \leq k \leq n$. 
Then $\sF(R')$ is $k$-truncated for all $R$-algebras $R'$. 
\label{ntruncatedsheaf2}
\end{lemma} 
\begin{proof} 
We prove the result by descending induction, as follows. 
Given $k + 1 \leq r \leq n$, we suppose that 
$\sF$ is $r$-truncated (as a presheaf of spaces), and we will conclude that
$\sF$ is actually $(r-1)$-truncated.
Resolving any $R$ by a hypercover of products of absolutely integrally
closed valuation rings and using \Cref{truncatedultra}, we see that it
suffices to show that $\sF$ takes $(r-1)$-truncated values
on arbitrary absolutely integrally closed valuation rings with the structure of
$R$-algebra.

Let $W$ be an absolutely integrally closed valuation ring under $R$. 
Suppose that $\ast \in \sF(W)$ is an arbitrary basepoint; then 
we need to show that $\pi_r( \sF(W), \ast) =0 $. 
The construction 
$\Omega^r \sF $ (where the loop space is taken with respect to the chosen
basepoint $\ast$) forms a finitary $\arc$-sheaf on $\mathrm{Ring}_W$, whose
value on $W_I$ is discrete for any interval $I \subset \spec(W)$ and is contractible
when $I$ has length $\leq 1$. 
By \Cref{xequalsy}, it follows that $\pi_0(\Omega^r \sF(W)) = \ast$ vanishes, as
desired. 
Thus, $\sF(W)$ is $(r-1)$-truncated as desired, whence the result. 
\end{proof} 

\begin{corollary} 
\label{detectequivfinitaryarc}
Let $\mathcal{C}$ be compactly generated by cotruncated objects. 
Let $\sF_1, \sF_2\colon  \mathrm{Sch}_{qcqs,R}^{op} \to \mathcal{C}$ be finitary
$\arc$-sheaves. 
Suppose given a map $\sF_1 \to \sF_2$ such that for every absolutely integrally closed valuation ring
$W$ of
rank $\leq 1$ that is an $R$-algebra, the map $\sF_1(W) \to \sF_2( W)$ is an
equivalence. Then $\sF_1\to \sF_2$ is an equivalence. 
\end{corollary} 
\begin{proof} 
By the Yoneda lemma, it suffices to assume that $\mathcal{C} = \mathcal{S}_{\leq
n}$ for some $n$. 
By Proposition~\ref{aicvaldetect}, 
it suffices to show that for any absolutely integrally closed valuation ring $V$
with the structure of $R$-algebra,
we have $\sF_1(V) \xrightarrow{\sim} \sF_2(V)$. 
Taking homotopy fibers over a basepoint, we may assume $\sF_2$ is contractible. 
Then \Cref{ntruncatedsheaf2}, we see that $\sF_1$ 
is $(-1)$-truncated, so the values of $\sF_1$ are either empty or contractible
and, by assumption, contractible on absolutely integrally closed rank $\leq 1$ valuation rings under $R$. 
By \Cref{xequalsy}, we find that $\sF_1$ takes contractible values on
absolutely integrally closed rank $\leq 1$ valuation rings under $R$, hence on
products of them, and hence in general on all $R$-algebras by hypercover
resolutions. 
\end{proof} 

For future reference, we record the following 
variant of \Cref{detectequivfinitaryarc}. 
\begin{corollary} 
Let $\mathcal{C}$ be compactly generated by cotruncated objects. 
Let $\sF_1, \sF_2\colon  \mathrm{Sch}_{qcqs,R}^{op} \to \mathcal{C}$ be finitary
$\arc$-sheaves. 
Suppose given a map $\sF_1 \to \sF_2$ such that for every absolutely integrally closed valuation ring
$W$ of
rank $\leq 1$ that is an $R$-algebra, 
there exists an extension of rank $\leq 1$ valuation rings $W \to W'$ such that 
the map $\sF_1(W') \to \sF_2( W')$ is an
equivalence. Then $\sF_1\to \sF_2$ is an equivalence. 
\label{completevariantrnk1}
\end{corollary} 
\begin{proof} 
As before, we can assume $\mathcal{C} = \mathcal{S}_{\leq n}$ and that $\sF_2$
is contractible. 
We claim that $\sF_1$ is $(-1)$-truncated. 
To see this, we show by descending induction on $k \geq -1$ that $\sF_1$ is 
$k$-truncated; the induction starts with $k = n$. 

Suppose we know that $\sF_1$ is $k$-truncated for some $k \geq 0$; we need to
show that $\sF_1$ is actually $(k-1)$-truncated. 
By \Cref{ntruncatedsheaf2}, it suffices to show that 
$\sF_1(V)$ is $(k-1)$-truncated whenever $V$ is an absolutely integrally closed
valuation ring of rank $\leq 1$ that is an $R$-algebra.  
For such $V$, we have a $V \to V'$ which is faithfully flat 
such that 
$\sF_1(V')$ is contractible. 
We can
form the \v{C}ech nerve, yielding 
$\sF_1(V) \simeq \varprojlim \left( \sF_1( V') \rightrightarrows
\sF_1(V' \otimes_V V') \dots \right)$.
This is a $\Delta$-indexed limit 
where the $0$th term is $(k-1)$-truncated and where every other term is
$k$-truncated; it follows that the limit $\sF_1(V)$ is also $(k-1)$-truncated as
desired. 
We thus conclude that $\sF_1$ is $(k-1)$-truncated and, by descending induction on $k$,
$(-1)$-truncated. 
In particular, the values of $\sF_1$ are either empty or contractible. 

To see that 
$\sF_1$ is contractible, 
it suffices (by Corollary~\ref{detectequivfinitaryarc}) to show that 
$\sF_1(V)$ is contractible for any absolutely integrally closed, rank $\leq 1$
valuation ring $V$ with the structure of $R$-algebra. 
Choose an extension $V \to V'$ such that $\sF_1(V')$ is contractible. 
But $\sF_1(V) 
\simeq \varprojlim \left( \sF_1( V') \rightrightarrows
\sF_1(V' \otimes_V V') \dots \right)
$, and the zeroth term in the totalization is contractible while all the other
terms are $(-1)$-truncated, by what we have shown above. It follows easily that $\sF_1(V)$ is contractible as
desired, whence the result by \Cref{detectequivfinitaryarc}. 
\end{proof} 

\subsection{Aic-$v$-excision and excision}
\label{excgeneral}
Finally, we show that a functor which satisfies arc-descent also satisfies excision.  That is, we show that (1) implies (2) in Theorem~\ref{mainthm}.
We follow a general argument going back to Voevodsky, though we find it more
convenient to argue directly instead of quoting an axiomatization (such as
\cite[Th.~3.2.5]{AHW}).

\begin{lemma} 
\label{milnorbase}
Consider a Milnor square 
\begin{equation} 
\xymatrix{
A \ar[d]  \ar[r] & A/I \ar[d]  \\
B \ar[r] &  B/J
}. \label{milnorsquare2}
\end{equation} 
Then for any map $A \to
V$ where $V$ is an integral domain, the base change of the square
\eqref{milnorsquare2} along
$A \to V$ is also a Milnor square. 
\end{lemma} 
\begin{proof} 
In fact, the square
\[ \xymatrix{
V \ar[d]  \ar[r] &  V/IV \ar[d]  \\
B \otimes_A V \ar[r] &  B/J \otimes_A V
}\]
has at least the property that the map 
$f\colon  V \to V/IV \times_{(B /J) \otimes_A V} B \otimes_A V $ is surjective, by
right-exactness of the tensor product. 
To see that the map is injective, it suffices to show that 
after base change to the fraction field $\mathrm{Frac}(V)$, the target is
nonzero. This holds because the map $A  \to \mathrm{Frac}(V)$ factors through
$A/I \times B$
(Lemma~\ref{excisionimpliesvleq1}), so $(A/I \times B) \otimes_A \mathrm{Frac}(V) \neq 0$. 
\end{proof}

\begin{lemma} 
\label{surjmilnorsquare}
Let $\mathcal{C}$ be compactly generated by cotruncated objects. 
Let $\mathcal{F}\colon  \mathrm{Sch}_{qcqs,R}^{op} \to \mathcal{C}$ be a finitary  $v$-sheaf.
Given a Milnor square as in \eqref{milnorsquare2} such that 
$A \to B$ is surjective, $\mathcal{F}$ carries every base change of it to a  pullback square. 
\end{lemma}
\begin{proof} 
Since everything is local on $A$ (and $\mathcal{F}$ is a $v$-sheaf),
it suffices to check that $\mathcal{F}$ carries every base change of
\eqref{milnorsquare2} to an absolutely integrally closed valuation ring to a
pullback square. By Lemma~\ref{milnorbase},
base change to an absolutely integrally closed valuation ring preserves Milnor squares. 
Thus, it suffices to show that $\mathcal{F}$ carries any Milnor square
\eqref{milnorsquare2}
with $A$ an absolutely integrally closed valuation ring and $A
\twoheadrightarrow B$ surjective to a 
pullback square. 
In this case, 
we must have that one of the maps $A \to A/I$ or $A \to B$ is an isomorphism;
otherwise, we could not have a Milnor square as the ideals of $A$ are totally
ordered. In either case, it is clear that $\mathcal{F}$ carries the diagram to a 
pullback. 
 \end{proof} 

The following result shows that in the statement of \Cref{mainthm}, (1) implies (2),  
and therefore completes the proof of Theorem~\ref{mainthm}. 
\begin{proposition} 
\label{arcsheafisexcisive}
Let $\mathcal{C}$ be compactly generated by cotruncated objects. 
Let $\mathcal{F}\colon  \mathrm{Sch}_{qcqs,R}^{op} \to \mathcal{C}$ be a finitary  $\arc$-sheaf. Then $\mathcal{F}$ is excisive. 
\end{proposition} 
\begin{proof} 
Consider a Milnor square \eqref{milnorsquare2}; we show that $\mathcal{F}$
carries every base change of it to 
a pullback square. 
Using \Cref{detectequivfinitaryarc}, 
the observation that everything is local on $A$, and that $\mathcal{F}$ is an
$\arc$-sheaf, we reduce to the case 
of a base-change to a rank $\leq 1$ valuation ring. 
Using \Cref{milnorbase}, we see that it suffices to show that 
$\mathcal{F}$ carries any Milnor square \eqref{milnorsquare2}
where $A$ is a rank $\leq 1$ valuation ring to a pullback. 

In this case, either the map $A \to A/I$ or the map $A \to B$ admits a section
thanks to Lemma~\ref{excisionimpliesvleq1}. 
If the first map admits a section, then $I = 0$ and  
it is clear that $\mathcal{F}$ carries \eqref{milnorsquare2}
to a  pullback square. 
Suppose that $A \to B$ admits a section. Then we can form a new Milnor square
\begin{equation}  \label{surjsquare} \xymatrix{
B \ar[d]  \ar[r] &  B/J \ar[d]  \\
A \ar[r] &  A/I
}, \end{equation}
where the section $B \to A$ is surjective. 
Now it suffices to show that $\mathcal{F}$ carries \eqref{surjsquare} to a fiber square,
by a two-out-of-three argument. However, $\mathcal{F}$ carries \eqref{surjsquare} to a
fiber square thanks to  Lemma~\ref{surjmilnorsquare}. 
\end{proof} 

\subsection{Excision via $\arc$-sheafification}
\label{ss:ExcSheaf}

In this subsection, we give a slightly different formulation of the relation
between excision and $\arc$-descent. Namely, we shall prove that the square of
schemes attached to an excision datum gives a pushout square of $\arc$-sheaves
of spaces on $\arc$-sheafification; this implies
Proposition~\ref{arcsheafisexcisive} (without the hypothesis that
$\mathcal{C}$ should be compactly generated by cotruncated objects, or that
$\sF$ should be finitary) by the universal property of pushouts. To obtain a topos, we first bound the size of all relevant schemes by a cutoff cardinal. 

\begin{definition}[The $\arc$-topos on schemes of size $<\kappa$]
\label{ArcCovKappa}
We fix an uncountable strong limit cardinal $\kappa$ (i.e., $\kappa$ is
uncountable and $\lambda < \kappa$ implies $2^\lambda < \kappa$) and a base
ring $R$ of cardinality $< \kappa$. Let $\mathrm{Sch}_{qcqs, R, < \kappa}$
denote the category of qcqs $R$-schemes
which can be written as a finite union of affine schemes whose coordinate rings
have cardinality $< \kappa$. The {\em $\arc$-topology} on $\mathrm{Sch}_{qcqs,
R,  < \kappa}$ is defined as the
Grothendieck topology where covering families $\{U_i \to U, i \in I\}$ are 
families of maps such that there exists a 
finite $I' \subset I$ such that
$\sqcup_{i \in I'} U_i \to U$ has the $\arc$-lifting property with respect to
valuation rings of size $< \kappa$, i.e., for any rank $\leq 1$ valuation ring
$V$ of size $< \kappa$ and a map $\mathrm{Spec}(V) \to U$, there exists an
extension $V \to W$ of rank $\leq 1$ valuation rings of size $< \kappa$ and a
lift $\mathrm{Spec}(W) \to \sqcup_{i \in I'} U_i$. Write
$\mathrm{Shv}_{\arc}(\mathrm{Sch}_{qcqs,R, < \kappa})$ for the category of
$\arc$-sheaves on $\mathrm{Sch}_{qcqs,R, < \kappa}$; this is a topos.
\end{definition}

\begin{remark}[Independence of $\kappa$ in the notion of coverings]
\label{IndKappa}
Since $\kappa$ is a strong limit cardinal, essentially all the constructions
appearing previously in this paper to construct valuation rings out of schemes
preserve the property of having size $< \kappa$. For example, each $X$ in
$\mathrm{Sch}_{qcqs, R, < \kappa}$ admits an $v$-hypercover $Y_\bullet \to X$ such
that each $Y_i$ is in $\mathrm{Sch}_{qcqs,R, <\kappa}$ and is the spectrum of a
product of absolutely integrally closed valuation rings (necessarily of size $<
\kappa$): this follows from bounding the cardinality of the construction in
Proposition~\ref{vcoverbyvaluation}, noting that if $\lambda < \kappa$, then a
product of $\leq \lambda$  sets of size $\leq \lambda$ each has size $\leq
\lambda^\lambda \leq 2^{\lambda \times \lambda} \leq 2^{2^{\lambda}} < \kappa$.
In particular, the equivalence class of each rank $\leq 1$ valuation on any $X
\in \mathrm{Sch}_{qcqs, R, < \kappa}$ admits a representative $\mathrm{Spec}(V) \to
X$ with $V$ being an absolutely integrally closed rank $\leq 1$ valuation ring
of size $< \kappa$. Using this observation as well as the closure of
$\mathrm{Sch}_{qcqs, R, < \kappa}$ under fiber products in all schemes, it is
easy to see that a map $f\colon X \to Y$ in $\mathrm{Sch}_{qcqs, R, < \kappa}$ is an $\arc$-cover in the sense of Definition~\ref{ArcCovKappa} if and only if it is an $\arc$-cover in the sense of Definition~\ref{DefArcTop}.
 \end{remark}

\begin{remark}
The category $\mathrm{Sch}_{qcqs,R, < \kappa}$ and the topos
$\mathrm{Shv}_{\arc}(\mathrm{Sch}_{qcqs,R, < \kappa})$ evidently depend on
$\kappa$. Nevertheless, for $\kappa' > \kappa$ a larger uncountable strong limit
cardinal, the inclusion $\mathrm{Sch}_{qcqs,R, < \kappa} \subset
\mathrm{Sch}_{qcqs,R,< \kappa'}$ commutes with limits and preserves the notion of
coverings by Remark~\ref{IndKappa}, so there is an induced map
$f\colon\mathrm{Shv}_{\arc}(\mathrm{Sch}_{qcqs,R,< \kappa'}) \to
\mathrm{Shv}_{\arc}(\mathrm{Sch}_{qcqs,R, < \kappa})$ of topoi. One can show most
``geometrically meaningful constructions'' do not depend on $\kappa$ in the
sense that they are compatible with $f$. For example,  sheafification and $f_*$
are compatible for presheaves of bounded size (such as $h_X$ for $X \in
\mathrm{Sch}_{qcqs,R, < \kappa}$); see \cite{Waterhouse} for similar statements in the fpqc topology. 
\end{remark}

We shall use the language of coherent objects in a topos and that of coherent
topoi in the rest of this section; we refer the reader to \cite[Expose VI]{SGA4}
as well as \cite[Sec.~A]{SAG} for the relevant background (including definitions of quasi-compact, quasi-separated, and coherent objects). The next lemma records some generalities on coherent topoi, specialized to our setting.

\begin{lemma}
\label{ArcCohTop}
\begin{enumerate}
\item The topos $\mathrm{Shv}_{\arc}(\mathrm{Sch}_{qcqs,R, < \kappa})$ is coherent. 
\item Write  $h_X^\sharp \in \mathrm{Shv}_{\arc}(\mathrm{Sch}_{qcqs,R, <
\kappa})$  for the $\arc$-sheaf associated to $X \in \mathrm{Sch}_{qcqs,R, <
\kappa}$ under the Yoneda map. Then $h_X^\sharp$ is a coherent object of
$\mathrm{Shv}_{\arc}(\mathrm{Sch}_{qcqs,R, < \kappa})$. 
\item  An object $F \in \mathrm{Shv}_{\arc}(\mathrm{Sch}_{qcqs,R, < \kappa})$ is
coherent if and only if there exists a surjection $h_X^\sharp \to F$ for some $X
\in \mathrm{Sch}_{qcqs,R, < \kappa}$ such that $h_X^\sharp \times_F h_X^\sharp$ is quasi-compact.
\item The topos $\mathrm{Shv}_{\arc}(\mathrm{Sch}_{qcqs,R, < \kappa})$ has enough points.
\end{enumerate}
\end{lemma}
\begin{proof}
Part (1) and (2) are \cite[Prop.~A.3.1.3]{SAG}. Part (3) is a consequence of
\cite[Exp.~VI, Cor.~1.17.1]{SGA4}. Part (4) is Deligne's completeness theorem \cite[Expose VI, Appendix]{SGA4}.
\end{proof}

\begin{warning}
The $\arc$-topology is {\em not} subcanonical, so the sheafification is necessary in Lemma~\ref{ArcCohTop} (2) and (3).
\end{warning}

The following criterion for detecting surjections between coherent objects will be useful. 

\begin{lemma}
\label{CohArcSheafIsomCrit}
Let $F \to G$ be a map of $\arc$-sheaves of sets on $\mathrm{Sch}_{qcqs, R, < \kappa}$. 
\begin{enumerate}
\item Assume that $G$ is coherent and $F$ is quasi-compact. Then $F \to G$ is
surjective if and only if it has the $\arc$-lifting property, i.e., for every
rank $\leq 1$ valuation ring $V$ of size $< \kappa$ with the structure of
$R$-algebra and every section $g \in G(V)$, there exists an extension $V \to W$ of rank $\leq 1$ valuation rings of size $< \kappa$ and a section $f \in F(W)$ lifting the image of $g$ in $G(W)$.
\item Assume that both $F$ and $G$ are coherent. Then $F
\xrightarrow{\sim} G$ if and only if $F(V) \xrightarrow{\sim} G(V)$ for a cofinal collection of rank $\leq 1$ valuation rings $V$ of size $< \kappa$.
\end{enumerate}
\end{lemma}

Note that $(1)$ above is trivially false without the quasi-compactness
hypothesis: the canonical map $h_{\spec( \mathbb{Q})}^{\sharp} \sqcup \bigsqcup_p h_{\mathrm{Spec}(\mathbb{Z}_{(p)})}^\sharp \to h_{\mathrm{Spec}(\mathbb{Z})}^\sharp \simeq \ast$ (where the coproduct is indexed by the prime numbers) has the $\arc$-lifting property but is not a surjection of $\arc$-sheaves. 

\begin{proof} In both cases, the ``only if'' direction is clear, so it suffices to prove the ``if'' direction.

We first prove (1). When $F$ and $G$ are representable, i.e., have the form
$h_X^\sharp$ and $h_Y^\sharp$, the desired claim is essentially a reformulation
of the definition of an $\arc$-covering combined with the observation that the
map from a presheaf to the associated sheaf has the $\arc$-lifting property by
definition of the topology. In general, one first chooses a surjection
$h_U^\sharp \to G$ (which is possible as $G$ is quasi-compact); one may then then
choose a surjection $h_V^\sharp \to h_U^\sharp \times_G F$ (which is possible as
$G$ is quasi-separated and $F$ is quasi-compact). By the stability of the
$\arc$-lifting property under fiber products (and the fact that it holds 
for surjections), then the map $h_V^\sharp \to h_U^\sharp$ is surjective by the representable case. But then $h_V^\sharp \to h_U^\sharp \to G$ is also surjective, whence $F \to G$ is surjective as it factors a surjection.

For (2), note that $F \to G$ is an isomorphism if and only if both $F \to G$ and
its diagonal $F \to F \times_G F$ are surjective. The claim now follows from (1)
and the stability of coherent objects under fiber products since the hypothesis $F(V) \simeq G(V)$ implies that both $F \to G$ and its diagonal $F \to F \times_G F$ have the $\arc$-lifting property. \qedhere
\end{proof}

\begin{lemma} 
Let $F_0 \to G_0$ be a map of presheaves of sets on
$\mathrm{Sch}_{qcqs,R, <\kappa}$. Suppose that the $\arc$-sheafifications $F_0^{\sharp},
G_0^{\sharp}$ are coherent in the $\arc$-topos of $\mathrm{Sch}_{qcqs,R, <
\kappa}$. 
Suppose that
for every rank $\leq 1$ valuation 
ring $V$ of size $< \kappa$ with the structure of $R$-algebra, there exists an
extension $V \to W$ of rank $\leq 1$ valuation rings of size $<\kappa$ with
$F_0(W)
\to G_0(W)$ surjective (resp.~bijective). Then $F_0^{\sharp} \to G_0^{\sharp}$
is surjective (resp.~bijective). 
\label{ArcPresheafSheafCrit}
\end{lemma} 
\begin{proof} 
We apply 
Lemma~\ref{CohArcSheafIsomCrit}. 
Consider the commutative diagram of presheaves on 
$\mathrm{Sch}_{qcqs, R, < \kappa}$, 
\[ \xymatrix{
F_0 \ar[d]  \ar[r] &  F_0^{\sharp} \ar[d]  \\
G_0 \ar[r] &  G_0^{\sharp}
}.\]
For (1), given a rank $\leq 1$ valuation ring  $V$ of cardinality
$< \kappa$ with the structure
of $R$-algebra, 
and
any class in $x\in G_0^{\sharp}(V)$, there exists an extension $V \to W$ of rank
$\leq 1$ valuation 
rings of size $<\kappa$
such that the image of $x$ in $G_0^{\sharp}(W)$ comes from
$\widetilde{x} \in G_0(W)$. 
By assumption, there exists a further extension $W \to W'$ of rank $\leq 1$
valuation rings of size $<\kappa$ such that 
the image of $\widetilde{x}$ in $G_0(W')$ comes from $F_0(W')$. 
Therefore, the image of $x$ in $G_0^{\sharp}(W')$ comes from $F_0^{\sharp}(W')$,
whence the surjectivity by 
Lemma~\ref{CohArcSheafIsomCrit}. The second claim 
follows from the first applied to the diagonal $F_0\to F_0 \times_{G_0} F_0$,
since sheafification commutes with finite limits. 
\end{proof} 

The following lemma shall be used to deduce coherence of certain pushouts:

\begin{lemma}
\label{PushoutCritCoh}
Consider a pushout square diagram
\[ \xymatrix{ A \ar[r] \ar[d] & B \ar[d] \\ C \ar[r] & D }\]
in $\mathrm{Shv}_{\arc}(\mathrm{Sch}_{qcqs,R,<\kappa})$. Assume $A \to B$ is injective, and that $A$, $B$, and $C$ are coherent. Then $D$ is coherent.
\end{lemma}
\begin{proof}
We apply the criterion recorded in Lemma~\ref{ArcCohTop} (3) to the cover $B \sqcup C \to D$. Since $B \sqcup C$ is trivially coherent, it remains to show that $B \times_D B$, $B \times_D C$, and $C \times_D C$ are all quasi-compact. 

Since $A \to B$ is injective, the same holds for its (co)base change $C \to D$, so $C \times_D C \simeq C$ (via the diagonal) is quasi-compact by assumption on $C$.

Next, the natural map $A \to B \times_D C$ is an isomorphism: this can be checked after taking stalks (using Deligne's theorem in Lemma~\ref{ArcCohTop} (4) to ensure there are enough points), where it reduces to a standard calculation with pushout squares of sets with one arrow injective. Thus, $B \times_D C$ is quasi-compact by assumption on $A$. 

Finally, we claim that the  map $B \sqcup (A \times_C A) \to B \times_D B$ given by the diagonal on the first summand $B$ of the source, and the natural map $A \times_C A \to B \times_D B$ is surjective; this will prove the lemma as $B$ is quasi-compact (by assumption), $A \times_C A$ is quasi-compact (as $A$ is quasi-compact and $C$ is coherent) and the image of a quasi-compact object is quasi-compact. As in the previous paragraph, this surjectivity can again be checked after taking stalks. Write $(-)'$ for a fixed stalk functor. The desired statement is now the following elementary assertion (that we leave to the reader to check): if $(x,y) \in B' \times_{D'} B'$, then $x=y$ is an element of $B' - A' \subset B'$ (whence $(x,y)$ lies in the image of the diagonal $B' \to B' \times_{D'} B'$), or both $x$ and $y$ live in $A'$ and are mapped to the same element of $C'$ (whence $(x,y)$ lies in $A' \times_{C'} A' \subset B' \times_{D'} B'$).
\end{proof}

We can now prove the promised statement,

\begin{proposition}
\label{ExcSquarePushout}
Let $(A,I) \to (B,J)$ be an excision datum with $A$ and $B$ having size
$<\kappa$, so one obtains a Milnor square as in \eqref{milnorsquare2}. Consider the associated square
\[ \xymatrix{ \spec(B/J) \ar[r] \ar[d] & \spec(B) \ar[d] \\
		  \spec(A/I) \ar[r] & \spec(A)}\]
in $\mathrm{Sch}_{qcqs,R, <\kappa}$. The square of $\arc$-sheaves of sets
obtained by applying $\arc$-sheafification to the above square is a pushout
square in both  the ordinary topos $\mathrm{Shv}_{\arc}(\mathrm{Sch}_{qcqs,R, <
\kappa})$ of $\arc$-sheaves of sets as well as the $\infty$-topos
$\mathrm{Shv}_{\arc,\mathcal{S}}(\mathrm{Sch}_{qcqs,R, < \kappa})$ of $\arc$-sheaves of spaces.
\end{proposition}

\begin{proof}
Let $Q$ denote the pushout of $\spec(B) \gets \spec(B/J) \to \spec(A/I)$ in the
$\infty$-category of presheaves of spaces. As $\spec(B/J) \to \spec(B)$ is a
closed immersion, it gives a monomorphism of presheaves of sets, so $Q$ is
discrete\footnote{Given a pushout diagram of sets where one of the maps being pushed out is injective, it is also a pushout in $\mathcal{S}$.} and hence we may regard it as a presheaf of sets. There is a natural map $\eta\colon Q \to \spec(A)$ of presheaves. The sheafification of $Q$ is coherent by Lemma~\ref{PushoutCritCoh}.  We must show that $\eta$ gives an isomorphism on $\arc$-sheafification. 

By Lemma~\ref{ArcPresheafSheafCrit}, it is  enough to show that $\eta(V)\colon
Q(V) \to \spec(A)(V)$ is bijective for any rank $\leq 1$ valuation ring $V$ of
size $< \kappa$. The surjectivity is immediate from
Lemma~\ref{excisionimpliesvleq1}. For injectivity, fix $x_1,x_2 \in Q(V)$ with
common image  $y \in \spec(A)(V)$. Since $\spec(B/J) \subset \spec(B)$ is the
preimage of $\spec(A/I) \subset \spec(A)$, we find that $Q \to
\spec(A)$ is an isomorphism after pullback to $\spec(A/I)$. We may thus assume
that $y \in \spec(A)(V)$ corresponds to a map $y^*\colon A \to V$ with $y^*(I)
\neq 0$. Then both $x_1$ and $x_2$ must lie in the image of $\spec(B)(V) \to
Q(V)$: if not, then one of them would give a factorization of $y^*$ through $A
\to A/I$, which we just ruled out. Thus, $x_1$ and $x_2$ come from two ring maps
$b_1,b_2\colon B \to V$ that factor $y^*$. Choose some $t \in I$ such that $y^*(t) \neq 0$. As $t  \in I$, we have $A[\frac{1}{t}] = B[\frac{1}{t}]$, and so $b_1[\frac{1}{t}] = b_2[\frac{1}{t}]$ as maps $B[\frac{1}{t}] \to V[\frac{1}{y^*(t)}]$. As $y^*(t) \neq 0$, the map $V \to V[\frac{1}{y^*(t)}]$ is injective, so we must have $b_1 = b_2$, whence $x_1 = x_2$ as wanted.
\end{proof}

\begin{corollary}
\label{ExcArcSheafPull}
Let $\mathcal{F}$ be a sheaf on $\mathrm{Sch}_{qcqs,R, < \kappa}$ valued in an
$\infty$-category $\mathcal{C}$ that has all small limits. Then $\mathcal{F}$ satisfies excision. In other words, $\mathcal{F}$ carries excision squares (i.e., the squares of schemes appearing in Proposition~\ref{ExcSquarePushout}) to pullback squares.
\end{corollary}
\begin{proof}
For $\arc$-sheaves of spaces, this follows from Proposition~\ref{ExcSquarePushout}. In general, for any object $x \in \mathcal{C}$, the assignment $U \mapsto \mathrm{Hom}_{\mathcal{C}}(x, \mathcal{F}(U))$ gives an $\arc$-sheaf of spaces and thus satisfies excision. The claim now follows by the Yoneda lemma in $\mathcal{C}$.
\end{proof}

\newpage
\section{Examples of $\arc$-sheaves}

In this section, we record several examples of functors which satisfy
$\arc$-descent. There are two classes of examples we consider: those arising
from \'etale cohomology, and those arising from perfect $\mathbb{F}_p$-schemes. 
\label{Etalesec}

\subsection{\'Etale cohomology}
\label{ss:EtaleCohArcSheaf}
Let $R$ be a fixed base ring, and let $\mathcal{G}$ be a torsion sheaf on 
the small \'etale site of $\mathrm{Spec}(R)$. 
In this section, we will consider the functor
\[ \mathrm{Sch}_{qcqs,R}^{op} \to \D( \mathbb{Z})  \]
which sends a qcqs scheme $X$
with structure map $f\colon  X \to \mathrm{Spec}(R)$ to the \'etale cohomology
object $R \Gamma(X_{\et}, f^*\mathcal{G})$. 
Recall that this functor is finitary,  cf.~\cite[Tag 03Q4]{stacks-project}. 
Our main result is that it satisfies $\arc$-descent
(Theorem~\ref{EtaleCohExcisive}). To begin with, we review the (classical)
result that it satisfies $v$-descent.  This comes from the theory of
cohomological descent \cite[Exp.~V-bis]{SGA4}, and also appears explicitly in
\cite[Prop. 5.3.3]{CD}.

\begin{lemma} 
\label{hensdetect}
Let $\Lambda$ be a ring, and let $\mathcal{F}\colon  \mathrm{Sch}_{qcqs,R}^{op}\to \D( \Lambda)^{\geq 0}$ be a finitary functor. Suppose
that $\mathcal{F}$ satisfies \'etale descent. 
If $f\colon  Y \to X$ is a map in $\mathrm{Sch}_{qcqs,R}$ such that the base changes of $f$ to the
strict henselizations of $X$ are of $\mathcal{F}$-descent, then $f$ is of $\mathcal{F}$-descent. 
\end{lemma} 
\begin{proof} 
This follows from the fact that the \'etale topology has enough points,
given by the spectra of strictly henselian local rings. 
Suppose that $f\colon  Y \to X $ is such that after base change to any strict
henselization of $X$, $f$ is of $\mathcal{F}$-descent. 
We need to see that the natural map 
\begin{equation}  \mathcal{F}(X) \to \varprojlim( \mathcal{F}( Y) \rightrightarrows \mathcal{F}(
 Y \times_X Y)  \triplearrows \dots ) \label{Fmapetale} \end{equation}
is an equivalence. 
To do this, we consider both sides as sheaves on the small \'etale site of $X$. For
instance, for 
any \'etale $X' \to X$, we consider $X' \mapsto \mathcal{F}( X' )$, and
similarly for the right-hand-side (i.e., via $X' \mapsto \varprojlim(
\mathcal{F}(X' \times_X Y) \rightrightarrows \mathcal{F}(X' \times_X Y \times_X
Y)  \triplearrows \dots )$). 
Both sides are finitary functors, since totalizations in $\D( \Lambda)^{\geq 0}$
commute with filtered colimits. 
Our assumption is that this map of
\'etale sheaves
becomes an equivalence on stalks for each  strict henselization of $X$;
therefore the map of sheaves is an equivalence, and so is \eqref{Fmapetale}.
\end{proof}

\begin{proposition} 
\label{vDescentEtaleCoh}
Let $\mathcal{G}$ be a torsion sheaf on $\mathrm{Spec}(R)_{\et}$. Let
$\mathcal{F}\colon  \mathrm{Sch}_{qcqs,R}^{op} \to \D( \Lambda)^{\geq 0}$ be the functor $(f\colon X \to \mathrm{Spec}(R)) \mapsto R \Gamma(X_{\et}, f^* \mathcal{G})$. Then $\mathcal{F}$ satisfies $v$-descent. 
\end{proposition} 
\begin{proof} 
Let $Y \to X$ be a $v$-cover of schemes, which we want to show is of universal
$\sF$-descent.
Since $Y, X$ are qcqs, we can write $Y$ as a filtered limit of a tower of
finitely presented
$X$-schemes $\left\{Y_\alpha\right\}$ with affine 
transition maps. Since \'etale cohomology turns such filtered limits to filtered
colimits, and since each $Y_\alpha \to X$ is a $v$-cover too, we may assume that $Y \to X$
is finitely presented (cf.~Proposition~\ref{finitevtopology}). 
The map $Y \to X$ admits a refinement which factors as a composite of a quasi-compact open covering
and a proper finitely presented surjection \cite[Th.~3.12]{Rydh}. 

Quasi-compact open coverings are of universal $\sF$-descent. We claim that proper surjections are too, thanks to proper base change. 
In fact, if $Y \to X$ is proper and surjective, 
with $X$ the spectrum of a strictly henselian ring, let $x \in X$ be the closed
point with residue field $\kappa(x)$. 
Let $Y_x$ be the fiber of $Y$ at $x$. 
Then $\mathcal{F}(Y) \simeq \mathcal{F}( Y_x)$ thanks to proper base change
\cite[Exp.~XII, Th.~5.1]{SGA4}, whereas the map $Y_x \to x$
admits a section after base change along the universal homeomorphism $\kappa(x) \to
\overline{\kappa(x)}$ (which does not change the value of $\mathcal{F}$). 
The analogous result holds for fiber products of $Y$ over $X$. 
It follows that 
$Y \to X$ is of $\mathcal{F}$-descent by comparison with $Y_x \to x$, which is
of universal $\sF$-descent by \Cref{Fsorite}. 
Thus any proper surjection $Y \to X$ with $X$ strictly henselian is of
$\mathcal{F}$-descent; by Lemma~\ref{hensdetect}, it follows that proper surjections are
of $\mathcal{F}$-descent. 
Since proper surjections are stable under base change, they are of universal
$\mathcal{F}$-descent. 
\end{proof}

Next, we show that $\mathcal{F}$ satisfies $\arc$-descent.

\begin{lemma} 
\label{aicstricthens}
An absolutely integrally closed valuation ring $V$ is strictly henselian. 
\end{lemma} 
\begin{proof} 
Let $\mathfrak{m} \subset V$ be the maximal ideal. 
Given a monic polynomial $p(x) \in V[x]$, we know that $p$ splits 
into linear factors; this forces $V$ to be strictly henselian, and 
the residue field to be algebraically closed. 
\end{proof}

\begin{theorem}[$\arc$-descent for \'etale cohomology] 
\label{EtaleCohExcisive}
Let $\mathcal{G}$ be a torsion sheaf on $\mathrm{Spec}(R)_{\et}$. Let
$\mathcal{F}\colon  \mathrm{Sch}_{qcqs,R}^{op} \to \D( \Lambda)^{\geq 0}$ be the functor $(f\colon X \to \mathrm{Spec}(R)) \mapsto R \Gamma(X_{\et}, f^* \mathcal{G})$.
Then $\sF$ satisfies $\arc$-descent. In particular, it satisfies excision. 
\end{theorem} 
\begin{proof} 
Proposition~\ref{vDescentEtaleCoh} already shows that $\mathcal{F}$ satisfies
$v$-descent. Also, $\mathcal{F}$ is finitary because \'etale cohomology commutes
with filtered colimits of rings. It is therefore enough to check the condition
of aic-$v$-excision from Theorem~\ref{mainthm}. Fix  an absolutely integrally
closed valuation ring $V$ that is an $R$-algebra and a prime ideal $\mathfrak{p} \subset V$. Note that any reduced quotient of a localization of $V$ is also an absolutely integrally closed valuation ring, and hence a strictly henselian local ring. In particular, as $V$ and $V/\mathfrak{p}$ are both strictly henselian with identical residue fields, we have $\mathcal{F}(V) \simeq \mathcal{F}(V/\mathfrak{p})$ by standard facts in \'etale cohomology. Similarly, we also have $\mathcal{F}(V_{\mathfrak{p}}) \simeq \mathcal{F}(\kappa(\mathfrak{p}))$, so the cartesianness of the square from Theorem~\ref{mainthm} (3) is clear.
\end{proof}

We now prove part (1) of Corollary~\ref{GabberLocal} from the introduction,
recovering many special cases of the Gabber--Huber rigidity theorem for the \'etale cohomology of henselian pairs \cite{Gabber, Huber}.
\begin{proof}[Proof of part (1) of Corollary~\ref{GabberLocal}: rigidity from excision]
Let $(A,I)$ be a henselian pair, where $A$ is an algebra over a henselian local ring
$k$. Let $\mathcal{F}$ be a torsion abelian sheaf on $\spec(A)_{\text{\'et}}$. We must show that the map $\eta_{\mathcal{F}}\colon R\Gamma(\spec(A),\mathcal{F}) \to R\Gamma(\spec(A/I), \mathcal{F})$ is an isomorphism. 

First, note that the collection of all sheaves $\mathcal{F}$ for which
$\eta_{\mathcal{F}}$ is an isomorphism satisfies the ``$2$-out-of-$3$'' property
and contains the sheaves pushed forward from $\spec(A/I)$.  We may therefore assume $\mathcal{F}$ has the form $j_! \mathcal{G}$ for some torsion \'etale sheaf $\mathcal{G}$ on $\spec(A) \setminus V(I)$, where $j\colon \spec(A) \setminus V(I) \hookrightarrow \spec(A)$ is the displayed open immersion.

Next, let us prove the claim when $A/I$ is a henselian local ring with
residue field $E$, which implies that $A$ is also a henselian local ring with
residue field $E$. By pro-(finite \'etale) descent, this reduces to the case
where $A$ is a strictly henselian local ring. But the global sections functor is
exact on such a ring, and $\Gamma(\spec(A), \mathcal{H}) \simeq \Gamma(\spec(E),
\mathcal{H}|_{\spec(E)})$ for any \'etale sheaf $\mathcal{H}$ on $\spec(A)$,
whence the claim in this case.

We will now use excision to reduce to the special case treated above. Let $k \to
A$ be the structure map. Set $B := k \times_{A/I} A$. Then, we can also view $I$
as an ideal $J$ of $B$, and we get an excision datum $(B,J) \to (A,I)$ with $B/J
\simeq k$ being a henselian local ring. 
By an observation of Gabber 
\cite{Gabberpair}, $(B, J)$ is a henselian pair. 
As $B \to A$ is an isomorphism after
inverting any element of $J$, the sheaf $\mathcal{G}$ can be viewed as a torsion
\'etale sheaf on $\spec(B) \setminus V(J)$. Write $\mathcal{F}' := j'_!
\mathcal{G}$, where $j'\colon \spec(B) \setminus V(J) \hookrightarrow \spec(B)$
is the displayed open immersion. Consider the functor $R\Gamma(-, \mathcal{F}')$
on the category of qcqs schemes over $\spec(B)$; we need to show that this functor
carries $A \to A/I$ to an equivalence. Theorem~\ref{EtaleCohExcisive} implies that this functor carries the Milnor square associated to $(B,J) \to (A,I)$ to a  pullback square.
We are thus reduced to showing that the functor carries $B \to B/J$ to an
equivalence. But then $B/J \simeq k$ is a henselian local ring, so we are done by the special case shown earlier.
\end{proof}

Next, we observe that excision for \'etale cohomology is also a quick
consequence of Gabber--Huber's affine analog of proper base change \cite{Gabber, Huber}. 

\begin{remark}[Excision from rigidity]
\label{r}
Suppose that we have an excision datum $f\colon  (A, I) \to (B, J)$. 
To see that the induced square on $R \Gamma_{}( \cdot, \Lambda)$ is
 cartesian, 
 note first that the datum of
an excision datum is preserved under flat base change in $A$. 
 Thus, we may localize on $A$ and assume that $A $ is a henselian
local ring with maximal ideal $\mathfrak{m} \subset A$ (cf.~the proof of
\Cref{hensdetect}). We need to see that the square 
\begin{equation} \label{lastetaleexc} \xymatrix{
R \Gamma( \mathrm{Spec}(A)_{\et} , \Lambda) \ar[d]  \ar[r] &  R \Gamma( \mathrm{Spec}(A/I)_{\et}, \Lambda)
\ar[d]  \\
R \Gamma( \mathrm{Spec}(B)_{\et}, \Lambda) \ar[r] &  R \Gamma( \mathrm{Spec}(B/J)_{\et}, \Lambda)
} \end{equation}
is a pullback. 
Then,  there are two cases: 
\begin{enumerate}
\item $I = A$, so $1 \in I$. In this case,  $J = B$ since $J$ is an ideal and
therefore $A \simeq B$. Therefore, the excision  assertion
 is evident. 
\item $I \subset \mathfrak{m}$. Then $(A, I)$ is a henselian pair. It
then follows by an observation of Gabber \cite{Gabberpair} that $(B, J)$ is a henselian pair. 
By the affine analog of proper base change, we have
$R \Gamma( \mathrm{Spec}(A)_{\et}, \Lambda) \simeq R \Gamma( \mathrm{Spec}(A/I)_{\et}, \Lambda)$ and $R
\Gamma( \mathrm{Spec}(B)_{\et}, \Lambda) \simeq R \Gamma( \mathrm{Spec}(B/J)_{\et}, \Lambda)$. Therefore,
excision holds in this case as well: the horizontal maps in
\eqref{lastetaleexc} are equivalences. 
\end{enumerate}
Of course, in our approach above, we do not need to invoke the affine analog of
proper base change at any point, but we need to assume that one is working over
a henselian local ring. 
\end{remark} 

\subsection{Constructible \'etale sheaves}
In this subsection we prove ``categorified'' versions of
Theorem~\ref{EtaleCohExcisive}.

We demonstrate the following result, which is due to
Rydh in the forthcoming work \cite{RydhII}; we thank him for
originally indicating the
result to us.  We give our own quick proof using 
Theorem~\ref{mainthm}. 

\newcommand{\cat}{\mathrm{Cat}}

\begin{theorem}[Rydh \cite{RydhII}] 
\label{rydhthmsheaves}
Consider the following three functors: 
\begin{enumerate}
\item  
The functor 
$\sF_0\colon  \mathrm{Sch}_{qcqs}^{op} \to \cat_1$ that sends a qcqs scheme $X$ to the category of
\'etale
qcqs $X$-schemes which are separated over $X$. 
\item 
The functor 
$\sF_1\colon  \mathrm{Sch}_{qcqs}^{op} \to \cat_1$ that sends a qcqs scheme $X$ to the category of
finite \'etale
$X$-schemes. 
\item The functor $\sF_2\colon  \mathrm{Sch}_{qcqs}^{op} \to \cat_1$ that sends a 
qcqs scheme $X$ to the category of constructible \'etale sheaves of sets on
$X$, or equivalently  the category of \'etale qcqs algebraic spaces over
$X$ (see \cite[Ch.~V, Th.~1.5]{Mil80} and \cite[Exp.~IX,
Cor.~2.7.1]{SGA4} for this equivalence).
\end{enumerate}
Then $\sF_0, \sF_1, \sF_2$  are  finitary $\arc$-sheaves. 
\end{theorem}

\begin{proposition} 
\label{etaleaic}
Let $V$ be an absolutely integrally closed valuation ring. Let $f\colon X \to \spec(V)$ be a quasi-compact separated \'etale
morphism. Then $X$ is a disjoint union of $V$-schemes of the form $\spec(
V[1/t])$, $t \in V$. That is, the category of separated \'etale $V$-schemes is equivalent
to the category of disjoint unions of quasi-compact opens of $\spec(V)$. 
\end{proposition} 
\begin{proof} 
This is a general fact about integral normal schemes with separably closed
function field, cf.~\cite[Tag 0EZN]{stacks-project}. For the convenience of the reader, we also sketch a direct proof.

Using Lemma~\ref{filtcolimitaic} and a standard noetherian approximation argument, we may assume $V$ has finite rank. We prove the claim by induction on the rank. When the rank is $0$, then $V$ is an algebraically closed field, so the claim is clear. In general, since $V$ is strictly henselian, we can use Zariski's main theorem to write $X := X' \sqcup U$ where $X' \to \mathrm{Spec}(V)$ is finite \'etale, and $U \to \mathrm{Spec}(V)$ does not hit the closed point. As $V$ is absolutely integrally closed, the map $X' \to \mathrm{Spec}(V)$ is a disjoint union of sections: the ring of functions on each component of $X'$ is a normal domain finite over $V$, but there are no such domains other than $V$ as the fraction field of $V$ is algebraically closed.  Also, $f(U) \subset \mathrm{Spec}(V)$ is a quasi-compact open subset that misses the closed point. As $V$ is an 
absolutely integrally closed 
valuation ring, each finitely generated ideal is principal, so $f(U) :=
\mathrm{Spec}(V[1/t])$ for some nonunit $t \in V$. In particular, $f(U)$ is the
spectrum of an
absolutely integrally closed 
valuation ring of rank strictly smaller than that of $V$. Applying the inductive hypothesis to the map $U \to f(U)$ now gives the claim.
\end{proof}

\begin{proof}[Proof of \Cref{rydhthmsheaves}] 
First, the functors $\sF_0, \sF_1$ are finitary by the general theory about limits
of qcqs schemes along affine transition maps (cf.~\cite[Tag
081C]{stacks-project}).  Similarly $\sF_2$ is finitary, as a presentation of
the algebraic space can be descended, cf.~\cite[Tags 07SJ, 084V]{stacks-project}.

The functors $\sF_0, \sF_1, \sF_2$ satisfy $v$-descent. 
Recall that since the functors are finitary, it suffices to check $v$-descent on
schemes of finite type over $\spec( \mathbb{Z})$, where the $v$-topology
reduces to the topology of universal
submersions. Rydh's result \cite[Cor.~5.18]{Rydh} implies that $\sF_0,
\sF_1, \sF_2$ satisfy
$v$-descent. 
To complete the argument, thanks to \Cref{mainthm},  it remains to show that $\sF_0, \sF_1, \sF_2$ satisfy
aic-$v$-excision. 
Let $V$ be an absolutely integrally closed valuation ring and let $\mathfrak{p} \subset V$ be a prime
ideal. 
Then $V/\mathfrak{p}, V_{\mathfrak{p}}, \kappa( \mathfrak{p})$ are absolutely integrally closed valuation rings as well.
\begin{enumerate}
\item  
For $\sF_1$, 
 the
categories of finite \'etale algebras over $V, V/\mathfrak{p},
V_{\mathfrak{p}}, \kappa( \mathfrak{p})$ are all equivalent to the category of
finite sets thanks to Lemma~\ref{aicstricthens}. Thus clearly $\sF_1$ satisfies
aic-$v$-excision. 

\item
For $\sF_0$, we see by Proposition~\ref{etaleaic}
that the category of quasi-compact, separated, and  \'etale $V$-schemes is equivalent
to the category of disjoint unions of quasi-compact open sets of $V$.  Similarly
for $V/\mathfrak{p}, V_{\mathfrak{p}}, \kappa( \mathfrak{p})$.

We need to see 
that $\sF_0( \spec(V)) \simeq \sF_0( \spec(V_{\mathfrak{p}}))
\times_{ \sF_0( \spec( \kappa( \mathfrak{p})))} \sF_0( \spec( V/\mathfrak{p}))$. 
We can reduce to the case where $V$ has finite rank by writing
$V$ as a filtered colimit of absolutely integrally closed valuation subrings
(Lemma~\ref{filtcolimitaic}).  Suppose $V$ has rank $n$, so 
$\spec(V)$ is the totally ordered set $\left\{1,2, \dots, n\right\}$ under
specialization, and
$\mathfrak{p}$ corresponds to the element $i \in [1, n]$. 
Then we observe that: 
\begin{enumerate}
\item The category of 
quasi-compact separated \'etale $V$-schemes is the category of functors $(1 \to 2 \to \dots
\to n)^{op} \to \mathrm{FinSet}_{\mathrm{inj}}$, the category of finite sets and
injections. 
\item The category of 
quasi-compact
separated \'etale $V_{\mathfrak{p}}$-schemes is the
category of functors $(1\to 2 \to \dots \to i)^{op} \to
\mathrm{FinSet}_{\mathrm{inj}}$. 
\item 
The category of 
quasi-compact
separated \'etale $(V/\mathfrak{p})$-schemes is the
category of functors $(i\to i+1 \to \dots \to n)^{op} \to
\mathrm{FinSet}_{\mathrm{inj}}$. 
\item The category of 
quasi-compact
separated \'etale $\kappa( \mathfrak{p})$-schemes 
is the category of functors $i^{op} \to \mathrm{FinSet}_{\mathrm{inj}}$. 
\end{enumerate}

In view of the above identifications,  
$\sF_0( \spec(V)) \simeq \sF_0( \spec(V_{\mathfrak{p}}))
\times_{ \sF_0( \spec( \kappa( \mathfrak{p})))} \sF_0( \spec( V/\mathfrak{p}))$
as desired. 
\item For $\sF_2$, the category of constructible sheaves of sets on $\spec(V)$
is just the category of functors $\left\{1, 2, \dots, n\right\}^{op} \to \mathrm{FinSet}$, again
via Proposition~\ref{etaleaic} which describes the \'etale site. 
One has a similar description for $V_{\mathfrak{p}}, V/\mathfrak{p}, \kappa(\mathfrak{p})$ and the
result follows. 
\end{enumerate}
Therefore, $\sF_0, \sF_1, \sF_2$ satisfy aic-$v$-excision and are finitary $v$-sheaves. 
In light of \Cref{mainthm} we can conclude. 
\end{proof}

Let $\Lambda$ be a \emph{finite} ring. 
Next, we prove that the functor which sends a qcqs scheme $X$ to the
$\infty$-category $\D^b_{\mathrm{cons}}(X_{\et}, \Lambda)$ of constructible sheaves 
satisfies $\arc$-descent. This is a refinement of
Theorem~\ref{EtaleCohExcisive} (which we recover by taking derived endomorphisms of the
constant sheaf).

\begin{notation}[Constructible sheaves]
Consider the functor $X \mapsto \D^b_{\mathrm{cons}}(X_{\et}, \Lambda),    \ \mathrm{Sch}_{qcqs}^{op} \to \cat_\infty$
which assigns to a qcqs scheme
$X$ the subcategory of the full derived $\infty$-category $\D(X_{\et}, \Lambda)$ of \'etale
sheaves of $\Lambda$-modules spanned by those objects which are bounded with constructible
cohomology. 

For $n \geq 0$, 
we let $\D^b_{\mathrm{cons}}(X_{\et}, \Lambda)^{[-n, n]} \subset \D^b_{\mathrm{cons}}(X_{\et}, \Lambda)$ be the subcategory
consisting of objects with amplitude in $[-n, n]$. 
For example, when $n = 0$ we recover the ordinary abelian category of
constructible sheaves of $\Lambda$-modules on $X$.  
\end{notation}

\begin{remark}
When $\Lambda$ is a field or a product of fields, the preceding definition
coincides with the standard one appearing in the definition of the $\ell$-adic
derived category, cf.~\cite[Tag 09C0]{stacks-project}. When $\Lambda$ has a
nontrivial Jacobson radical, 
one can also enforce finite $\Lambda$-Tor-dimension (which is necessary to
obtain something functorial in $\Lambda$). 
The results below also hold for the variant where we force finite
$\mathrm{Tor}$-dimension. 
\end{remark}

\newcommand{\sG}{\mathcal{G}}
\begin{proposition} 
\label{dboundedvsheaf}
The functor $X \mapsto \D^b_{\mathrm{cons}}(X_{\et}, \Lambda)$ is finitary. 
More generally, for each $n$, the functor $X \mapsto \D^b_{\mathrm{cons}}(X_{\et}, \Lambda)^{[-n, n]}$
is finitary.  
\end{proposition} 
\begin{proof} 
It suffices to prove that each 
functor $X \mapsto \D^b_{\mathrm{cons}}(X_{\et}, \Lambda)^{[-n, n]}$ is
finitary, since then we can take the union over $n$. 

Suppose that we can write $X = \varprojlim_{i \in I} X_i$ 
as a filtered inverse limit of a tower of qcqs schemes $X_i$ over a totally
ordered set $I$, such that the transition maps $X_j \to X_i$ are affine. 
Let $p_i\colon  X \to X_i$ and $p_{ji}\colon  X_j \to X_i$ (for $j \geq i$) denote the
transition maps. 

Let $0 \in I$. 
First, we show that if $\sF, \sG \in \D^b_{\mathrm{cons}}((X_0)_{\et}, \Lambda)$, then 
$\hom_{\D^b_{\mathrm{cons}}( X_{\et}, \Lambda)}( p_0^* \sF, p_0^* \sG)  
\simeq \varinjlim_{i \geq 0} \left( \hom_{\D^b_{\mathrm{cons}}( (X_i)_{\et},
\Lambda)}( p_{i0}^* \sF, p_{i0}^* \sG) \right)
$. This easily implies that the map 
$\varinjlim  \left( \D^b_{\mathrm{cons}}( (X_i)_{\et}, \Lambda) \right) \to \D^b_{\mathrm{cons}}( X_{\et}, \Lambda)$ is fully
faithful (and similarly with truncations attached). 
In fact, by a d\'evissage (first via the Postnikov tower, then using \cite[Tag
095M]{stacks-project}),
we reduce to the case where $\sF$ is obtained as the (discrete) sheaf
$j_!(M)$ for some finite $\Lambda$-module $M$ and $j\colon  U_0 \to X_0$ a separated
\'etale map and where $\sG$ is also (up to shift) discrete. 
We have then 
for each $i$,
$ \hom_{\D^b_{\mathrm{cons}}( (X_i)_{\et},
\Lambda)}( p_{i0}^* \sF, p_{i0}^* \sG) 
= 
\hom_{\D(\Lambda)}(M,  R \Gamma(U_0 \times_{X_0} X_i, q_i^* \mathcal{G}) )$
for $q_i \colon
U_0 \times_{X_0} X_i \to X_0$, and similarly for $X$. 
In this case, however, the result follows from the commutation of
\'etale cohomology with filtered colimits over the bases $\{U_0 \times_{X_0}
X_i\}$.

Next, we need to show that 
the functor
$\varinjlim  \left( \D^b_{\mathrm{cons}}( (X_i)_{\et}, \Lambda)^{[-n, n]}
\right) \to \D^b_{\mathrm{cons}}( X_{\et},
\Lambda)^{[-n, n]}$
is essentially surjective. Since the functor is already known to be fully faithful (indeed on the
whole bounded derived $\infty$-categories), it suffices by
d\'evissage to show the result in the case $n = 0$, where it follows from 
\cite[Tag 09YU]{stacks-project}. 
\end{proof}

\begin{proposition} 
\label{dbconsvsheaf}
The functor $X \mapsto \D^b_{\mathrm{cons}}(X, \Lambda)$ on $\mathrm{Sch}_{qcqs}$ is a
hypercomplete $v$-sheaf. 
More generally, for each $n$, the functor $X \mapsto \D^b_{\mathrm{cons}}(X, \Lambda)^{[-n, n]}$
is a hypercomplete $v$-sheaf. 
\end{proposition} 

Related results can be found in \cite{LaszloOlsson1}.

\begin{proof} 
Note that the assertions for 
$\D^b_{\mathrm{cons}}(X, \Lambda)$ and $\D^b_{\mathrm{cons}}(X, \Lambda)^{[-n, n]}$ (for each $n$) are equivalent,
since given a surjection of qcqs schemes $Y \to X$, amplitude in the derived
category on $X$ can be tested by pullback to $Y$; moreover, 
$\D^b_{\mathrm{cons}}(X, \Lambda) = \bigcup_{n \geq 0}
\D^b_{\mathrm{cons}}(X, \Lambda)^{[-n, n]}$ and this filtered union
commutes with the relevant homotopy limits. For this, compare the proof of
\cite[Th.~11.2]{BhattScholzeWitt}. 
Hence it suffices to show that 
$X \mapsto \D^b_{\mathrm{cons}}(X, \Lambda)^{[-n, n]}$ is a hypercomplete $v$-sheaf. 
Note that this functor
takes values in $\cat_{2n+1}$, which is compactly generated by cotruncated objects, so
hypercompleteness is automatic once we prove that it is a $v$-sheaf; this
follows from \cite[Lemma 6.5.2.9]{HTT} for sheaves of $n$-truncated spaces, and
by the Yoneda lemma (applied using the cotruncated compact generators) in general. The functor is an \'etale sheaf by construction, and we have seen in
\Cref{dboundedvsheaf} that it is finitary. 
Using the refinement result \cite[Th.~3.12]{Rydh}, together with 
\Cref{finitevtopology},
it now suffices to prove that for a proper surjection of
schemes $Y \to X$ finite type over $\spec(\mathbb{Z})$, we have
for each $n$,
\begin{equation}  \D^b_{\mathrm{cons}}(X_{\et}, \Lambda)^{[-n, n]}
\xrightarrow{\sim} \varprojlim \left( \D^b_{\mathrm{cons}}( Y_{\et},
\Lambda)^{[-n, n]}
\rightrightarrows \D^b_{\mathrm{cons}}( (Y \times_X Y)_{\et}, \Lambda)^{[-n, n]} \triplearrows \dots
 \right).  \label{derivedcatlimit} \end{equation}
Indeed, the map $Y \to X$ determines via the \v{C}ech nerve an augmented cosimplicial
scheme (over $X$) and the diagram \eqref{derivedcatlimit} is obtained 
by applying $\D^b_{\mathrm{cons}}( (\cdot)_{\et}, \Lambda)^{[-n, n]}$ to it. 

We verify the equivalence using the abstract result \cite[Cor.~4.7.5.3]{HA}
(applied to opposite $\infty$-categories); this is an $\infty$-categorical
version of the connection between the Barr--Beck theorem and descent,
cf.~\cite[Sec.~4.2]{Deligne90}. 
Note that the pullback 
$\D^b_{\mathrm{cons}}( X_{\et}, \Lambda) \to \D^b_{\mathrm{cons}}( Y_{\et}, \Lambda)$ is clearly conservative
since isomorphisms are checked on stalks. 
Since pullback is exact, 
$\D^b_{\mathrm{cons}}( X_{\et}, \Lambda)^{[-n, n]} \to \D^b_{\mathrm{cons}}( Y_{\et}, \Lambda)^{[-n, n]}$
preserves finite limits and hence totalizations (which behave as finite limits
since we are working with $(2n+1)$-categories). 
Moreover, for any cartesian diagram of qcqs schemes
\[ \xymatrix{
E' \ar[d] \ar[r] &  E \ar[d]  \\
F' \ar[r] &  F
}\]
along \emph{proper} maps,
the square
of $\infty$-categories
\[ \xymatrix{
\D^{b}_{\mathrm{cons}}( F_{\et}, \Lambda)^{[-n, n]} \ar[d]  \ar[r] &  \D^b_{\mathrm{cons}}( F'_{\et},
\Lambda)^{[-n, n]} \ar[d]  \\
\D^b_{\mathrm{cons}}( E_{\et}, \Lambda)^{[-n, n]} \ar[r] &  \D^b_{\mathrm{cons}}( E'_{\et}, \Lambda)^{[-n, n]}
}\]
is right adjointable 
in the sense of 
\cite[Def.~4.7.4.13]{HA}. The right adjoints are given by truncations of
derived pushforward and the adjointability follows from proper base change
\cite[Exp.~XII, Th.~5.1]{SGA4}. 
Therefore, 
we can apply 
\cite[Cor.~4.7.5.3]{HA} to conclude the desired equivalence
\eqref{derivedcatlimit}. 
\end{proof} 

\begin{proposition} 
Let $V$ be an absolutely integrally closed valuation ring of finite rank $m$, so $X = \spec(V)$ (under
specialization) is the
partially ordered set $0 \to 1 \to 2 \to \dots \to m$. 
Then $\D^b_{\mathrm{cons}}( X_{\et}, \Lambda)$ is the $\infty$-category of functors
$(0 \to 1 \to 2 \to \dots \to m)^{op} \to \D^b_{fg}( \Lambda)$ for $\D^b_{fg}( \Lambda)$
the bounded derived category of finitely generated $\Lambda$-modules.
Similarly 
$\D^b_{\mathrm{cons}}( X_{\et}, \Lambda)^{[-n, n]}$ is the $\infty$-category of functors
$( 0 \to 1 \to 2 \to \dots \to m)^{op} \to \D^b_{fg}( \Lambda)^{[-n, n]}$. 
\label{etalesheafonspec}
\end{proposition} 
\begin{proof} 
Recall (Proposition~\ref{etaleaic}) that the category of separated \'etale $V$-schemes is equivalent to the
category of finite disjoint unions of quasi-compact open subsets of $\spec(V)$. 
Every quasi-compact open subset is of the form
$\left\{0, 1, 2, \dots, i\right\}$ for some $i \leq m$, 
and 
the Grothendieck topology in this case is trivial. 
Thus, a sheaf 
on $\spec(V)_{\et}$
is simply a presheaf on the category
of open subsets, whence the result. 
\end{proof}

\begin{theorem}[$\arc$-descent for constructible sheaves] 
\label{arcConsDer}
The functor $X \mapsto \D^b_{\mathrm{cons}}(X_{\et}, \Lambda)$ is a hypercomplete $\arc$-sheaf. 
More generally, for each $n$, the functor $X \mapsto \D^b_{\mathrm{cons}}(X_{\et}, \Lambda)^{[-n, n]}$
is a hypercomplete $\arc$-sheaf. 
\end{theorem} 
\begin{proof} 
As in the proof of \Cref{dbconsvsheaf}, 
it suffices to prove 
the second assertion.
Moreover, the construction $X \mapsto \D^b_{\mathrm{cons}}(X_{\et}, \Lambda)^{[-n, n]}$ is
automatically hypercomplete if it is an $\arc$-sheaf because it takes values in
$\cat_{2n+1}$.

Since 
$X \mapsto \D^b_{\mathrm{cons}}(X_{\et}, \Lambda)^{[-n, n]}$
is already checked to be a $v$-sheaf (Proposition~\ref{dbconsvsheaf}), 
it suffices to verify aic-$v$-excisiveness by \Cref{mainthm}. Let $X =
\spec(V)$ for $V$ an absolutely integrally closed valuation ring and let $\mathfrak{p} \subset V$ be a
prime ideal. We need to show that 
$\D^b_{\mathrm{cons}}( \cdot_{\et}, \Lambda)^{[-n, n]}$ satisfies descent for the cover $\spec(
V/\mathfrak{p}) \sqcup \spec( V_{\mathfrak{p}}) \to \spec( V)$, i.e., that we
have a homotopy cartesian square
\[ \xymatrix{
\D^b_{\mathrm{cons}}( \spec(V)_{\et}, \Lambda)^{[-n, n]} \ar[d]  \ar[r] &  \D^b_{\mathrm{cons}}( \spec(
V/\mathfrak{p})_{\et}, \Lambda)^{[-n, n]} \ar[d]  \\
\D^b_{\mathrm{cons}}( \spec(V_{\mathfrak{p}})_{\et}, \Lambda )^{[-n, n]} \ar[r] &  \D^b_{\mathrm{cons}}( \spec(
\kappa(\mathfrak{p}))_{\et}, \Lambda)^{[-n, n]}
}.\]
Without loss of generality (in light of \Cref{dboundedvsheaf} and
\Cref{filtcolimitaic}), 
we can assume that $V$ has finite rank, so that $\spec(V) = \left\{1, 2,
\dots, m\right\}$ as a poset under specialization. Suppose $\mathfrak{p}$
corresponds to the element $i$; then 
$\spec(V_{\mathfrak{p}}) = \left\{1, 2, \dots, i\right\}$, $\spec(
V/\mathfrak{p}) = \left\{i, i+1, \dots, n\right\}$, and $\spec( \kappa(
\mathfrak{p})) = \left\{i\right\}$. Using the description of \'etale sheaves on
spectra of absolutely integrally closed valuation rings 
in Proposition~\ref{etalesheafonspec}, the result now follows. 
\end{proof} 

\subsection{Differential forms in characteristic zero}
In this subsection, we give an example of an $\arc$-sheaf in characteristic
zero. 
We thank Joseph Ayoub for raising this question. 

Let $k$ be a field of characteristic zero. 
Let $X$ be a finite type $k$-scheme. 
We consider the \emph{$h$-differential forms} of Huber--J\"order \cite{HJ14}, 
$R \Gamma_h(X, \Omega^i_h) \in \D(k)$; by definition, this is the cohomology in
the $h$-topology of the usual presheaf of differential $i$-forms.  
This defines a $\D(k)$-valued functor 
on schemes of finite type over $k$ which, by definition, is a sheaf for the
$h$-topology.
A basic comparison result \cite[Cor.~6.5]{HJ14} yields that 
for $X$ smooth, 
$R \Gamma_h(X, \Omega^i_h) \simeq R \Gamma_{\mathrm{Zar}}(X, \Omega^i)$: that
is, one can equally work in the Zariski topology for smooth schemes. 

Consider the functor
$\sF_i\colon \mathrm{Sch}_{qcqs,k}^{op} \to \D(k)$ given as the left Kan extension of $X \mapsto R
\Gamma_h( X,  \Omega^i_h)$ on finite type $k$-schemes.
Explicitly, given a qcqs $k$-scheme $X$, we write $X$ as a
filtered limit $\varprojlim_{j \in J} X_j$ of finite type $k$-schemes $X_j, j
\in J$ along affine transition maps; then 
$\sF_i(X) = \varinjlim_{j \in J} R \Gamma_h(X_j, \Omega^i_h)$.

\begin{theorem} 
\label{hdifferentialsarearc}
The functor $\sF_i\colon \mathrm{Sch}_{qcqs,k}^{op} \to \D(k)$
is an $\arc$-sheaf. 
\end{theorem} 
\begin{proof} 
We first recall that $\sF_i$ is a $v$-sheaf, since we started with an $h$-sheaf
on finite type $k$-schemes (\Cref{finitevtopology}). 
Now we verify the hypotheses of \Cref{mainthm}. 
 If $R$ is an $\mathrm{ind}$-smooth $k$-algebra, then 
$\sF_i(\spec(R))$ is represented by the (usual) 
differential $i$-forms $\Omega^i_{R/k}$, thanks to \cite[Cor.~6.5]{HJ14}. 
In particular, this holds for $R=V$ a valuation $k$-algebra, since valuation rings
are $\mathrm{ind}$-smooth as a consequence of local uniformization
\cite{Zar40} in characteristic zero; indeed, in the needed special case
of absolutely integrally closed
valuation rings, ind-smoothness also follows (in any characteristic) from
alterations \cite{dJ96}.

Let $V$ be an absolutely integrally closed valuation $k$-algebra and let $\mathfrak{p}$ be a prime ideal. 
To verify the homotopy pullback square \eqref{faic}, it suffices to show that
for each $i \geq 0$, one has a homotopy pullback square
\begin{equation}  \label{aicomega}
\begin{aligned}
\xymatrix{
\Omega^i_{V/k} \ar[d]  \ar[r] &  \Omega^i_{(V/\mathfrak{p})/k} \ar[d]  \\
\Omega^i_{V_{\mathfrak{p}}/k} \ar[r] &  \Omega^i_{\kappa(\mathfrak{p})/k}
}. \end{aligned} \end{equation}
For this, compare \cite[Lem.~3.12]{HK18}. 
Since $V/k$ is ind-smooth, $\Omega^1_{V/k}$ is a flat $V$-module. 
Moreover, we have that 
$\Omega^1_{(V/\mathfrak{p})/k} \simeq \Omega^1_{V/k} \otimes_V
(V/\mathfrak{p})$; by the transitivity triangle, this follows since
$\mathfrak{p} = \mathfrak{p}^2$ as $V$ is absolutely integrally closed. 
Therefore, \eqref{aicomega} follows from the pullback square of valuation rings
by tensoring with the flat $V$-module $\Omega^i_{V/k}$. 
\end{proof} 

A variant of this comes from the de Rham complex. 
For a finite type $k$-scheme $X$, let $R\Gamma_h(X, \Omega^\bullet)$ denote the ($h$-sheafified) \emph{algebraic de Rham cohomology} of $X$. 
As before, we left Kan extend this to all qcqs $k$-schemes to obtain a functor 
$\sF\colon \mathrm{Sch}_{qcqs,k}^{op} \to \D(k)$. 

\begin{corollary} 
The construction $X \mapsto \sF(X)$ defines an $\arc$-sheaf. 
\end{corollary} 
\begin{proof} 
This follows from 
the ($h$-sheafified) Hodge filtration and \Cref{hdifferentialsarearc}. 
\end{proof} 

\subsection{Perfect schemes of characteristic $p> 0$}
Using exactly the same strategy used to prove $\arc$-descent in \'etale cohomology, one also obtains $\arc$-hyperdescent for perfect complexes on perfect schemes, extending the analogous $v$-hyperdescent result in \cite[Theorem 11.2 (2)]{BhattScholzeWitt}.

\begin{theorem}[$\arc$-hyperdescent for perfect complexes on perfect schemes]
\label{archyp}
Fix a prime number $p$. Let $\mathcal{F}$ be the functor on
$\mathrm{Sch}_{qcqs,\mathbb{F}_p}$ carrying a qcqs scheme $X$ to the
$\infty$-category $\mathrm{Perf}(X_{\mathrm{perf}})$ of perfect complexes on the perfection $X_{\mathrm{perf}}$ of $X$. Then $\mathcal{F}$ is a hypercomplete $\arc$-sheaf.
\end{theorem}
\begin{proof}
For each $n$, let 
$\mathcal{F}_n$ denote the functor on $\mathrm{Sch}_{qcqs,\mathbb{F}_p}$ carrying $X$ to 
the subcategory of $\mathrm{Perf}(X_{\mathrm{perf}})$ consisting of objects with
$\mathrm{Tor}$-amplitude in $[-n, n]$. 
Then $\mathcal{F}_n$ takes values in the $\infty$-category 
$\mathrm{Cat}_{2n+1}$
of $(2n+1)$-categories, which is compactly generated by cotruncated objects. We have that
$\mathcal{F} = \bigcup_n \mathcal{F}_n$. 
As in the proof of \Cref{dbconsvsheaf}, it suffices to show that $\mathcal{F}_n$
is an $\arc$-sheaf (whence it is automatically hypercomplete). 

By \cite[Theorem 11.2 (2)]{BhattScholzeWitt} and its proof, we already know that
$\mathcal{F}_n$
is a hypercomplete $v$-sheaf. Also, $\mathcal{F}_n$ is clearly finitary. To show
that $\mathcal{F}_n$ is an $\arc$-sheaf, it suffices to check aic-$v$-excision as in
Theorem~\ref{mainthm} (3). Thus, fix an absolutely integrally closed valuation
ring $V$ of characteristic $p$ with a prime ideal $\mathfrak{p}$. We must check
that applying $\mathcal{F}_n$ to the Milnor square
\begin{equation}  \label{milnorsqperf}
\begin{aligned}
\xymatrix{
V \ar[d] \ar[r] &  V/\mathfrak{p} \ar[d] \\
V_{\mathfrak{p}} \ar[r] &  \kappa(\mathfrak{p})} \end{aligned} \end{equation}
gives a cartesian square of $(2n+1)$-categories; in fact, for this it suffices
to show that 
applying $\mathcal{F}$ to \eqref{milnorsqperf} gives a cartesian square of
$\infty$-categories. Note that all rings appearing
above are perfect. Moreover, the map $V \to B := V_{\mathfrak{p}} \times
V/\mathfrak{p}$ is descendable in the sense of \cite[Definition
11.14]{BhattScholzeWitt} 
or \cite[Definition~3.18]{MathewGalois} 
because the above square is cartesian, expressing $V$ as a
finite homotopy limit of objects in $\D(V)$ which admit the structure of
$B$-modules. By \cite[Theorem 11.15]{BhattScholzeWitt} (which is
\cite[Proposition 3.21]{MathewGalois}), we have $\mathcal{F}(V) \simeq
\varprojlim \mathcal{F}(B^\bullet)$, where $B^\bullet$ is the \v{C}ech nerve of
$V \to B$ in the $\infty$-category of $E_\infty$-rings. In fact, as both $V$ and
$B$ are perfect, the terms of $B^\bullet$ coincides with the \v{C}ech nerve of
$V \to B$ in the category of ordinary commutative rings by \cite[Lemma
3.16]{BhattScholzeWitt}. It is now easy to see, just as in
Proposition~\ref{1implies3}, that the statement $\mathcal{F}(V) \simeq
\varprojlim \mathcal{F}(B^\bullet)$ implies exactly that applying $\mathcal{F}$
(and hence $\mathcal{F}_n$) to the above Milnor square results in a cartesian
square. Therefore, we conclude that $\mathcal{F}_n$ is an
$\arc$-sheaf, whence the result. \end{proof}

Using the previous result, we show that on the category of qcqs perfect $\mathbb{F}_p$-schemes, the $\arc$-covers are precisely the universal effective epimorphisms (i.e., the covers for the canonical topology, see \cite[Tag 00WP]{stacks-project}).

\begin{theorem}[Universally effective epimorphisms of perfect schemes]
\label{PerfRepArcSheaf}
Fix a prime number $p$. Let $Y$ be a qcqs algebraic space over $\mathbb{F}_p$. Consider the functor $\sF$ on  $\mathrm{Sch}_{qcqs,\mathbb{F}_p}$ carrying $X$ to the set of maps $X_{\mathrm{perf}} \to Y$. Then  $\sF$ is an $\arc$-sheaf (of sets). 

Moreover, a map of qcqs perfect $\mathbb{F}_p$-schemes $Y \to X$ is
an $\arc$-cover if and only if it is a universally effective
epimorphism 
in the category of qcqs perfect $\mathbb{F}_p$-schemes. \end{theorem} 

Note that the functor $\sF$ is finitary if $Y$ is in addition of finite type over $\mathbb{F}_p$, and the proof shows that it satisfies excision even if $Y$ is only assumed qcqs (one can reduce to the finite type case via noetherian approximation \cite{RNoeth}). 

\begin{proof} 
\newcommand{\perf}{\mathrm{Perf}}
\newcommand{\fun}{\mathrm{Fun}}
For the sheafiness of $\sF$, we use the Tannaka duality in the form of \cite[Theorem 1.5]{B}. Given $X$, we let $\perf(X_{\mathrm{perf}})$ denote the symmetric monoidal $\infty$-category of perfect complexes on $X_{\mathrm{perf}}$ as before. Then  we have  
\[ \hom( X_{\mathrm{perf}}, Y) \simeq \fun^{\otimes}_{\mathrm{ex}}(\perf(Y), \perf(X_{\mathrm{perf}})),  \]
where the target denotes symmetric monoidal exact functors.  
Recall now that homotopy limits of symmetric monoidal, stable $\infty$-categories
are computed at the level of underlying $\infty$-categories. 
Since we have just
seen that $X \mapsto \perf(X_{\mathrm{perf}})$ satisfies $\arc$-descent
(Theorem~\ref{archyp}),  the result now follows.

The previous paragraph shows that $\arc$-covers of perfect qcqs
$\mathbb{F}_p$-schemes are universally effective epimorphisms 
in the category of perfect qcqs $\mathbb{F}_p$-schemes. For the converse,
it therefore suffices to show that if $X \to Y$ is a map of perfect qcqs schemes
that is a universally effective epimorphism, then it is also a cover for the
$\arc$-topology. As the property of being a universally effective epimorphism is
local, we may assume $Y := \mathrm{Spec}(V)$ is a rank $\leq 1$ valuation ring.
Moreover, as $\arc$-covers are universally effective epimorphisms, we may 
refine $X$ by an $\arc$-cover to assume $X$ has the form $\mathrm{Spec}(R)$
where $R := \prod_{i \in I} W_i$  is a product of  valuation rings
$W_i$ (cf.~\Cref{vcoverbyvaluation}). 
There are two cases to consider. 

Assume first that $V$ has rank $0$, so $V$ is a field. We must show that $X \neq
\emptyset$ or equivalently that $R \neq 0$. The sheaf property of
$\mathrm{Hom}(-, \mathbb{A}^1_{\mathrm{perf}})$ with respect to the map $V \to
R$ (assumed to be a universally effective epimorphism of qcqs perfect
$\mathbb{F}_p$-schemes) shows that $V$ is the equalizer of the two maps $R \to R \otimes_V R$, and so $R \neq 0$ since $V \neq 0$.

Assume now that $V$ has rank $1$. The previous paragraph shows that $I \neq
\emptyset$.  If one of the induced maps $V \to W_i$ is an injective local
homomorphism, then it is also faithfully flat, so we are done. Assume towards
contradiction then that each map $V \to W_i$ is either non-injective or
non-local. Any such map must factor over either the residue field $k$ (if
non-injective) or the fraction field $K$ (if non-local) of $V$. But then each map $V \to W_i$ factors over $V \to R' := K \times k$, and hence the same holds for $V \to R = \prod_i W_i$. In particular, $\spec(R') \to \spec(V)$ is also a canonical cover. But one easily checks that $R' \otimes_V R' \simeq R'$ via the multiplication map in all cases. Applying the sheaf axiom for the sheaf $\mathrm{Hom}(-, \mathbb{A}^1_{\mathrm{perf}})$ then shows that $V \simeq R'$, which is absurd since $R'$ is a product of fields while $V$ is a rank $1$ valuation ring.
\end{proof}

\begin{remark}
The canonical topology on qcqs perfect $\mathbb{F}_p$-schemes is \emph{not} quasi-compact. That is, a covering family in the canonical topology does not need to admit a  finite refinement.  Indeed, it is shown in \cite[Tag 0EUE]{stacks-project} that $\spec( \mathbb{Z})$ is not quasi-compact for the canonical topology on the category of all qcqs schemes. This example can be adapted to the setting of perfect $\mathbb{F}_p$-schemes by replacing $\spec(\mathbb{Z})$ with $\mathbb{A}^1_{\mathrm{perf}}$.
\end{remark}

\newpage
\section{Consequences of $\arc$-descent}

\subsection{Formal glueing squares}
Next, we prove that any functor satisfying $\arc$-descent also satisfies a
``formal glueing square.''
Recall the following assertion: if $A$ is a noetherian ring and  $t \in A$,
then we can form the {square} 
\begin{equation} \label{arithsquare}  \begin{aligned}\xymatrix{
A \ar[d]  \ar[r] & \hat{A}_t \ar[d]  \\
A[1/t] \ar[r] &  \hat{A}_t[1/t]
}, \end{aligned} \end{equation}
which is a pullback square; here $\hat{A}_t$ denotes the $t$-adic completion of
$A$. 
Given a functor on rings, one can ask whether it carries 
\eqref{arithsquare} to 
 a homotopy pullback square. 

\begin{example} 

A basic example is that nonconnective $K$-theory $\mathbb{K}$ does; that is, the
square
\[ \xymatrix{
\mathbb{K}(A) \ar[d]  \ar[r] &  \mathbb{K}( \hat{A}_t) \ar[d]  \\
\mathbb{K}(A[1/t]) \ar[r] &  \mathbb{K}( \hat{A}_t[1/t])
},\]
is homotopy cartesian. 
This follows from the theory of localization sequences: the fibers of the horizontal maps are given by the $\mathbb{K}$-theory of perfect
$t$-power torsion $A$-modules (resp.~$t$-power torsion $\hat{A}_t$)-modules,
cf.~\cite[Theorem 5.1]{TT}, and these categories are clearly equivalent. 
More generally one can formulate a statement 
for finitely generated ideals. 

\end{example} 
Here we prove an analogous result for finitary $\arc$-sheaves. First we need
some preliminaries. 

\begin{proposition} 
\label{completevaluationring}
Let $V$ be a rank $1$ valuation ring, and let $t$ be a pseudouniformizer. Then the
$t$-adic completion $\hat{V}_t$ is a rank $1$ valuation ring, and the map $V \to
\hat{V}_t$ is faithfully flat. 
\end{proposition} 
\begin{proof} 
In fact, if $K$ is the fraction field of $V$, then the rank $1$ valuation on $V$
defines a nonarchimedean absolute value $|\cdot | \colon  K \to \mathbb{R}_{\geq 0}$. The completion
$\hat{K}$ of $K$ with respect to the absolute value $|\cdot|$ also admits a
canonical
extension of the 
absolute value (denoted by $|\cdot|$ again). One checks that $\hat{V}_t$ is
precisely the rank $1$ valuation ring $\left\{x \in \hat{K}: |x| \leq 1\right\}
$ and that $V \to \hat{V}_t$ is an extension of valuation rings, hence
faithfully flat. 
\end{proof} 

Note in particular that if $V$ is a rank $1$ valuation ring, then the condition
that $V$ should be $t$-adically complete does not depend on the non-unit $t \neq
0$. 

\begin{proposition} 
\label{completeisarccover}
Let $R$ be a ring and let $I \subset R$ be a finitely generated ideal. Then
$\spec ( \hat{R}_I) \sqcup \spec(R) \setminus V(I) \to \spec(R)$ is an
$\arc$-cover. 
\end{proposition} 
\begin{proof} 
Let $V$ be a rank $\leq 1$ valuation ring. By Proposition~\ref{completevaluationring}, we may assume
that $V$ is complete with respect to any element in the maximal ideal
$\mathfrak{m}_V \subset V$. 
We then make a stronger claim: any map $\spec(V) \to \spec(R)$ factors through
the above cover. 

Consider a map $f\colon  R \to V$. If $f$ carries $I$ into the maximal ideal of $V$,
then $f(I) \subset tV$ for some pseudouniformizer $t \in \mathfrak{m}_V$ as $I$ is finitely generated. Clearly 
then $f$ factors over $\hat{R}_I$ since $V$ is $t$-adically complete. 
Conversely, if $f(I) $ generates the unit ideal in $V$, then $\spec(V) \to
\spec(R)$ factors through the open locus $\spec(R) \setminus V(I)$ and we are
done in this case too. 
\end{proof} 

In the following, we use the notion of a formal glueing square
(\Cref{DefFormalGlueing}). 
If $A$ is a ring (not
necessarily noetherian) and $I
\subset A$ is a finitely generated ideal, then the map $A \to \hat{A}_I$ induces
an equivalence modulo $I^n$ for all $n \geq 0$ (\cite[Tag 00M9]{stacks-project}).
In particular, the map $(A \to
\hat{A}_I, I)$ gives a formal glueing square. 

For the next result, we fix
a base ring $R$. 
Compare also 
Corollary~\ref{FormalGlueArcSheafGeneral} for a proof of the result with fewer
assumptions via $\arc$-sheafification. 

\begin{theorem}[Formal glueing squares for $\arc$-sheaves]
\label{arcarithm}
Let $\mathcal{C}$ be an  $\infty$-category that is compactly generated by cotruncated objects.
Let $\sF\colon  \mathrm{Sch}_{qcqs,R}^{op} \to \mathcal{C}$ be a finitary
$\arc$-sheaf. Then $\sF$ satisfies formal glueing, i.e., if  $(A \to B, I)$ is a
formal glueing datum of $R$-algebras in the sense of Theorem~\ref{FormalGlueing},   then  the natural square
\begin{equation}  \label{hcart4} \begin{aligned} \xymatrix{
\sF(\spec (A)) \ar[d]  \ar[r] &  \sF(\spec( B) ) \ar[d]  \\
\sF( \spec (A) \setminus V(I)) \ar[r] &  \sF( \spec( B) \setminus V(IB))
} \end{aligned} \end{equation}
is cartesian.
\end{theorem} 
\begin{proof} 
The argument follows a familiar pattern, cf.~the proof of
Proposition~\ref{arcsheafisexcisive}. 
We first consider the case where $A \to B$ is surjective (this may happen
in a non-noetherian case, e.g., $\mathbb{Z}_p \otimes_{\mathbb{Z}} \mathbb{Z}_p
\to \mathbb{Z}_p$ for $I = (p)$). 
Then assertion is local on $A$,  
so we may assume that $A$ is a rank $\leq 1$ valuation ring
(\Cref{detectequivfinitaryarc}). 
If $I = A$, then the assertion that \eqref{hcart4} is homotopy cartesian is trivial. If $I$ is contained in the maximal
ideal of $A$, then (since $I$ is finitely generated) the only way $A
\twoheadrightarrow B$ can induce an isomorphism $A/I^n A \simeq B/I^n B$ is for
$A = B$. In this case, too, \eqref{hcart4} is evidently homotopy cartesian. 

Next, suppose $A \to B$ admits a section $s\colon  B \to A$. 
In this case, 
$s\colon  B \to A$ also induces an equivalence modulo each power of $IB \subset B$, so
when we contemplate the diagram
$$
\xymatrix{
\sF(\spec (A)) \ar[d]  \ar[r] &  \sF(\spec( B) ) \ar[d]  
\ar[r]^{s^*} & \sF( \spec (A))  \ar[d] 
\\
\sF( \spec (A) \setminus V(I)) \ar[r] &  \sF( \spec( B) \setminus V(IB))
\ar[r]^{s^*} &
\sF( \spec(A) \setminus V(I)),
} $$
we get that the rightmost square is homotopy cartesian by the previous
paragraph. This also implies that the leftmost square is homotopy cartesian by
a 2-out-of-3. 

Now we consider the general case. 
The assertion is local on 
$A$, so we may assume that $A$ is 
a rank $\leq 1$ valuation ring which is complete, cf.~\Cref{completevariantrnk1}.
Then either $A$ is $I$-adically complete or $I$ is the unit ideal, since $I$ is
finitely generated. 
If $I$ is the unit ideal, then the above square \eqref{hcart4} is trivially homotopy cartesian. 
Suppose then that $A$ is $I$-adically complete, so then $A \to B$ admits a
section. In this case we are also done by the previous paragraph. 
\end{proof} 

\begin{example}[Formal glueing in \'etale cohomology] 
\label{arithetalesquare}
If $I$ is generated by a single element, then Theorem~\ref{arcarithm} shows that \'etale cohomology
(or any $\arc$-sheaf) admits ``formal glueing squares,'' i.e., carries diagrams of
shape \eqref{arithsquare} to homotopy pullbacks. 
\end{example}

Let us also sketch an essentially equivalent proof that is formulated slightly
differently and  applies to all $\arc$-sheaves (but uses $\arc$-sheafification);
this is analogous to the alternative proof of
Proposition~\ref{arcsheafisexcisive} given in
Proposition~\ref{ExcSquarePushout}. 
In particular, we can prove the formal glueing 
property of an $\arc$-sheaf without the assumption 
that that $\mathcal{F}$ should be finitary or take values in a $\mathcal{C}$
compactly generated by cotruncated objects. 
For the rest of this section, fix an
uncountable strong limit cardinal $\kappa$ and a base ring $R$ of cardinality
$< \kappa$. Consider the category $\mathrm{Sch}_{qcqs,R, <\kappa}$ equipped with the $\arc$-topology, as in Definition~\ref{ArcCovKappa}.

\begin{theorem}[Formal glueing via $\arc$-sheafification]
\label{FormalGlueArcSheaf}
For any formal glueing datum of $R$-algebras $(A \to B,I)$ as in Theorem~\ref{FormalGlueing} where $A$ and $B$ have size $< \kappa$, consider the square 
\[ \xymatrix{ \mathrm{Spec}(B) \setminus V(IB) \ar[r] \ar[d] & \mathrm{Spec}(B)
\ar[d] \\ \mathrm{Spec}(A) \setminus V(I) \ar[r] & \mathrm{Spec}(A). }\]
The square of $\arc$-sheaves obtained by applying $\arc$-sheafification to this
square is a pushout square in both the ordinary topos
$\mathrm{Shv}_{\arc}(\mathrm{Sch}_{qcqs,R, <\kappa})$ of $\arc$-sheaves of sets
as well as the $\infty$-topos
$\mathrm{Shv}_{\arc,\mathcal{S}}(\mathrm{Sch}_{qcqs,R, <\kappa})$ of $\arc$-sheaves of spaces.
\end{theorem}
\begin{proof}
We adopt the notation and set-theoretic conventions from \S \ref{ss:ExcSheaf}.  Before proceeding further, let us remark that if a ring has size $< \kappa$, then so does its completion with respect to any ideal, so the category of rings under consideration is closed under this operation.
 
Let $Q$ denote the pushout of $\spec(A) \setminus V(I) \gets \spec(B) \setminus
V(IB) \to \spec(B)$ in the $\infty$-category of presheaves of spaces. As in the
proof of Proposition~\ref{ExcSquarePushout}, $Q$ is discrete (i.e., equivalent
to a presheaf of sets) and its $\arc$-sheafification is coherent. There is an
evident map $\eta\colon Q \to \spec(A)$. Arguing as in
Proposition~\ref{ExcSquarePushout}, it is enough to prove that $\eta(V)\colon
Q(V) \to \spec(A)(V)$ is bijective for all complete rank $\leq 1$ valuation
rings $V$ of size $< \kappa$ with the structure of $R$-algebra. The surjectivity
of $\eta(V)$ follows from Proposition~\ref{completeisarccover}. For injectivity,
suppose we have two points $x_1,x_2 \in Q(V)$ giving the same point  $y \in \spec(A)(V)$. As $\spec(B) \setminus V(IB) \subset \spec(B)$ is the preimage of $\spec(A) \setminus V(I) \subset \spec(A)$, it formally follows that $\eta$ is an isomorphism after pullback to $\spec(A) \setminus V(I) \subset \spec(A)$. We may thus assume that the point $y \in \spec(A)(V)$ does not lie in $(\spec(A) \setminus V(I))(V)$, i.e., the corresponding map $y^*\colon A \to V$ carries $I$ into a nonunit ideal. But then both $x_1,x_2 \in Q(V)$ must come from $\spec(B)(V)$: if they came from $(\spec(A) \setminus V(I))(V)$, then $y^*(I)$ would generate the unit ideal, which is not possible. This means that the map $y^*\colon A \to V$ has two factorizations $x_1',x_2'\colon B \to V$ through $B$. But $V$ is $I$-adically complete  while $A$ and $B$ have the same $I$-adic completions, so $x_1' = x_2'$, whence $x_1 = x_2$.
\end{proof}

\begin{corollary}
\label{FormalGlueArcSheafGeneral}
Let $\mathcal{F}$ be an $\arc$-sheaf on $\mathrm{Sch}_{qcqs,< \kappa}$ valued in an
$\infty$-category $\mathcal{C}$ that has all small limits. Then $\mathcal{F}$ satisfies formal glueing as in Theorem~\ref{arcarithm}. 
\end{corollary}
\begin{proof}
This follows from Theorem~\ref{FormalGlueArcSheaf}, just as Corollary~\ref{ExcArcSheafPull} followed from Proposition~\ref{ExcSquarePushout}.
\end{proof}

Let us give an example illustrating why working locally in the $\arc$-topology is essential for Theorem~\ref{arcarithm} and Theorem~\ref{FormalGlueArcSheaf}, and the $v$-topology does not suffice.

\begin{example}
\label{ExFG}
Let $V$ be a rank $2$ valuation ring. Write $\mathfrak{p}$ for the height $1$
prime. Choose a nonunit $f \in V - \mathfrak{p}$, and let $I = (f)$, so
$\sqrt{I}$ is the maximal ideal. Write $\hat{V}_I$ for the $I$-adic completion
of $V$. Consider the formal glueing datum $(V \to \hat{V}_I, I)$. We claim that
the corresponding map  $(\spec(V) \setminus V(I)) \sqcup \spec(\hat{V}_I) \to
\spec(V)$ is not a $v$-cover. It suffices to show that there is no extension $V
\to W$ of valuation rings such that the $\spec(W) \to \spec(V)$ factors over
$(\spec(V) \setminus V(I)) \sqcup \spec(\hat{V}_I) \to \spec(V)$.  In fact, as
$\spec(W) \to \spec(V)$ is surjective with $\spec(W)$ connected, the only
possibility is that $V \to W$ factor as $V \to \hat{V}_I \to W$. But this is
impossible as $V \to W$ is injective while $V \to \hat{V}_I$ contains
$\mathfrak{p}$ in its kernel: we have $\mathfrak{p} \subset I^n$  for all $n
\geq 0$ since $f^n \notin \mathfrak{p}$ for all $n \geq 0$. Note that this is
not a problem if we are allowed to work $\arc$-locally on $V$ as the map $V \to
\widehat{V/\mathfrak{p}} \times V_{\mathfrak{p}}$ is an $\arc$-cover which factors over $(\spec(V) \setminus V(I)) \sqcup \spec(\hat{V}_I) \to \spec(V)$ on spectra.
\end{example}

\subsection{GAGA for rigid \'etale cohomology of affinoids}

In this subsection, we collect various results related to the \'etale
cohomology of ``affinoids,'' though we formulate them algebraically. 
Specifically, these results are related to the \'etale cohomology of rings
obtained by inverting an element $t$ on $t$-adically complete rings. 

Namely, we
reprove the Fujiwara--Gabber theorem (Theorem~\ref{GabberFujiwarathm}). 
Next, we show that \'etale cohomology of rings of the form $\hat{A}_t[1/t]$
satisfies descent with respect to a variant of the $\arc$-topology where the
element $t$ is
taken into account (Corollary~\ref{BerkDescent}). Finally, we record a
K\"unneth theorem (Proposition~\ref{kunnethformula}).

\begin{definition} 
We say that a functor $\sF$ on $R$-algebras is \emph{rigid} if 
for every henselian pair $(A, J)$ with $A$ an $R$-algebra, we have $\sF(A) \simeq
\sF(A/J)$. 
\end{definition}

Rigid $\arc$-sheaves exhibit some additional rigidity: they are insensitive to
completion, even after ``passage to the generic fiber.''

\begin{corollary}
\label{GabberFujiwara}
Let $R$ be a ring which is henselian along a finitely generated ideal $I \subset R$. 
Let $\sF\colon  \mathrm{Sch}_{qcqs,R}^{op} \to \D( \Lambda)^{\geq 0}$ (for some coefficient
ring $\Lambda$) be a finitary
$\arc$-sheaf which is  rigid on $R$-algebras. 
Then the map 
$\sF( \spec(R) \setminus V(I)) \to \sF( \spec( \hat{R}_I) \setminus V(I
\hat{R}_I))$ is an
equivalence. 
\label{critequiv}
\end{corollary} 
\begin{proof} 
This follows from the fiber square
\eqref{hcart4} applied with $S = \hat{R}_I$. By assumption, the top arrow is an
equivalence via rigidity; therefore, the bottom arrow is too. 
\end{proof} 

By the affine analog of proper base change \cite{Gabber, Huber}, it follows that \'etale cohomology
with torsion coefficients is rigid. 
We thus recover the following result of Fujiwara--Gabber. 
\begin{theorem}[{Fujiwara--Gabber, cf.~\cite[Cor.~6.6.4]{Fujiwara}}] 
\label{GabberFujiwarathm}
Let $(R, I)$ be a henselian pair with $I \subset R$ finitely generated, and let
$f \colon R \to \hat{R}_I$ denote the $I$-adic completion. Then for any torsion
\'etale sheaf $\sG$ on $\spec(R) \setminus V(I)$, the map 
\[ R \Gamma( \spec(R) \setminus V(I),\sG) \to 
R \Gamma( \spec ( \hat{R}_I) \setminus V( I \hat{R}_I), f^* \sG)
\]
is an equivalence. 
\end{theorem} 
\begin{proof} 
Note first that $\sG$ extends to a torsion \'etale sheaf $\sG_1$ on $\spec(R)$, for example, by extension by $0$.
We consider the functor 
$\sF\colon  
\mathrm{Sch}_{qcqs,R}^{op} \to \D(\mathbb{Z})^{\geq 0}$ which sends $p\colon  Y \to \spec(R)$ to $R
\Gamma( Y_{\et}, p^* \sG_1)$. 
By Theorem~\ref{EtaleCohExcisive} and the affine analog of proper base change 
\cite{Gabber, Huber} (which we reproved in \S \ref{ss:EtaleCohArcSheaf} in the
special case where $R/I$ is an algebra over a henselian local ring), it follows that 
the hypotheses of Corollary~\ref{GabberFujiwara} 
apply to the functor $\sF$, and hence we can conclude. 
\end{proof}

\begin{remark} 
In \cite{Fujiwara}, the above result is proved under the additional
hypothesis that $R$ is noetherian. The non-noetherian case of the Fujiwara--Gabber theorem 
is due to Gabber, and is outlined in 
\cite[Exp.~XX, Sec.~4.4]{Travaux}. 
\end{remark} 
\begin{remark} 
Corollary~\ref{critequiv} relies on the fact that equivalences in the derived
$\infty$-category $\D( \Lambda)$ can be checked after pullback. In particular, the
analogous result does not obviously apply to functors taking values in sets or
categories. For instance, it does not obviously apply to the functor which
sends a qcqs scheme to its category of finite \'etale schemes (which we have
seen satisfies $\arc$-descent, and which is rigid). Nevertheless, the conclusion
of Corollary~\ref{critequiv} is valid for this functor, by some results in the
literature which we now recall. 

In the noetherian case,  Elkik \cite{Elkik} proved that if $(R, I)$ is a henselian pair with $R$ noetherian,
then finite \'etale schemes of $\spec(R) \setminus V(I)$ and $\spec( \hat{R}_I)
\setminus V(I \hat{R}_I)$ are identified. This statement has been generalized
to the non-noetherian case in the important special case where $I =
(t)$ is principal with $t$ a nonzerodivisor, by Gabber--Ramero \cite[Prop.~5.4.53]{GabberRamero} and 
in the general finitely generated case by Gabber \cite[Exp.~XX, Th\'eor\`eme
2.1.2]{Travaux}. 
More generally, one has a rigidity result for arbitrary
sheaves of sets or sheaves of ind-finite groups on $\spec(R) \setminus V(I)$, that is,
$H^0$ (resp.~$H^0, H^1$) are identified over $\spec(R) \setminus V(I)$ and
$\spec( \hat{R}_I) \setminus V(I \hat{R}_I)$. 
Compare also \cite[Theorem 6.4]{HallRydh} for another proof of the case of sets. 
 \end{remark}

For our next application, we work (implicitly) in the category of $t$-adically complete $\mathbb{Z}[t]$-algebras $R_0$ (endowed with the $t$-adic topology). Recall that each such $R_0$ yields a Banach $\mathbb{Z}((t))$-algebra $R := R_0[\frac{1}{t}]$ with unit ball $R_0$ and thus corresponds to an affinoid rigid space $\mathrm{Spa}(R, R^\circ)$ over $\mathbb{Z}((t))$; conversely, every affinoid rigid space has such a form as one may simply take $R_0$ to be a ring of definition. Our next result is roughly that the assignment $\mathrm{Spa}(R,R^\circ) \mapsto R\Gamma(\spec(R), \Lambda)$ on affinoids satisfies descent with respect to the \'etale topology on affinoid rigid spaces; in fact, we prove descent with respect to a much finer topology. In particular, it follows that the purely algebraically defined \'etale cohomology groups $H^*(\spec(R), \Lambda)$ may serve as meaningful \'etale cohomology groups in rigid geometry. To avoid the language of rigid geometry (and, in particular, delicate sheafiness questions), we formulate our statements purely algebraically. 

\begin{definition} 
\label{DefArct}
We say that a map of $\mathbb{Z}[t]$-algebras $R \to S$ is an {\em $\arc_{t}$-cover}
if for every rank $ 1$ valuation ring $V$ over $\mathbb{Z}[t]$ where $t$ is
a pseudouniformizer and every map $R \to V$, there is an extension 
of rank $\leq 1$ valuation rings  $V \to W$ such that the map $R \to V \to W$
extends over $S$. 
\end{definition} 

We translate the above definition into the theory of adic spaces.

\begin{example}
Let $f\colon (A,A^+) \to (B,B^+)$ be a map of Tate $(\mathbb{Z}((t)), \mathbb{Z}[[t]])$-algebras, so $t$ is a pseudouniformizer. Then $f$ defines a map $\mathrm{Spa}(f)\colon \mathrm{Spa}(B,B^+) \to \mathrm{Spa}(A,A^+)$ on adic spectra. Then the map $A^+ \to B^+$ is an $\arc_t$-cover if and only if the map $\mathrm{Spa}(f)$ is surjective on generic points (or equivalently that the map on associated Berkovich spaces is surjective). Indeed, this follows immediately from the definitions: the generic points of the adic spectrum $\mathrm{Spa}(A,A^+)$ are in bijective correspondence with maps $A^+ \to V$ where $V$ is a rank $1$ valuation ring with $t$ being a pseudouniformizer (up to refinements of $V$). In particular, the $\arc_t$-topology on affinoid rigid spaces (as defined via pullback along $(A,A^+) \mapsto A^+$) is finer than the $v$-topology of \cite[Definition 8.1]{ScholzeDiamonds}.
\end{example}

\begin{cons}[\'Etale cohomology of the rigid generic fiber]
\label{etcohrigidgenericfib}
Fix a torsion abelian group $\Lambda$. 
Consider the functor $F$ on $\mathbb{Z}[t]$-algebras with values in $\D(
\mathbb{Z})^{\geq 0}$
that sends a
$\mathbb{Z}[t]$-algebra $R$ to $R \Gamma(  \spec(\hat{R}_t[1/t])_{\et},
\Lambda)$. 
By Theorem~\ref{GabberFujiwarathm}, one can replace the $t$-adic completion
with the $t$-henselization; therefore, $F$ commutes with filtered colimits. 
\label{rigidcoh}
\end{cons}

\begin{corollary}[$\arc_t$-descent for \'etale cohomology of affinoids]
\label{BerkDescent}
\label{BerkDesc2}
Let $A$ be a ring with a distinguished element $t \in A$. 
Let $\sG$ be a torsion sheaf on $\spec(A)_{\et}$. Consider the functor 
$F$ which sends an $A$-algebra $R$ to $F(R) := R \Gamma( \spec(\hat{R}_t[1/t]),
\sG)$. Then $F$ is a
finitary $\arc_t$-sheaf on $A$-algebras.
\end{corollary} 
\begin{proof} 
The statement that $F$ is finitary follows from
the discussion in 
\Cref{etcohrigidgenericfib}.
We consider the fiber square 
(Example~\ref{arithetalesquare}) 
\begin{equation} \label{fibsquare5} \xymatrix{
& R \Gamma( \spec(R)_{\et}, \mathcal{G}) \ar[d]  \ar[r] & R \Gamma( \spec(R[1/t])_{\et},
\mathcal{G}) \ar[d]  \\
R \Gamma( \spec( R/t)_{\et}, \mathcal{G})
& \ar[l]_{\sim} R \Gamma( \spec( \hat{R}_t)_{\et}, \mathcal{G}) \ar[r] & R \Gamma(
\spec(\hat{R}_t[1/t])_{\et} , \mathcal{G})
}, \end{equation}
where the identification in the bottom left  follows from the affine analog of
proper base change.

If $R \to S$ is an $\arc_t$-cover, then $R \to S \times R/t \times R[1/t]$ is an
$\arc$-cover: given a map $R \to V$ with $V$ a rank $\leq 1$ valuation ring, the
image of $t$ in $V$ can either pseudouniformizer or zero or a unit, and these
three possibilities correspond to factoring (up to extensions) through each of
the three terms of the product. Moreover, the desired statement for $S$ is equivalent to the
statement for $S \times R/t \times R[1/t]$ (since $t$-adic completion followed
by inverting $t$ kills both any $R/t$-module and any $R[1/t]$-module). Thus, it
suffices to show that $F$ is an $\arc$-sheaf.
We use the fiber square \eqref{fibsquare5}. 
To see that $F$ is an $\arc$-sheaf, it suffices to see that all the other terms
in \eqref{fibsquare5} are $\arc$-sheaves. This follows because
\'etale cohomology is an $\arc$-sheaf (\Cref{EtaleCohExcisive}) and $\arc$-covers are stable
under base change. \end{proof} 

In the language of adic spaces, Corollary~\ref{BerkDescent} implies that the purely algebraically defined \'etale cohomology groups of affinoid adic spaces satisfy descent for the analytic \'etale topology; in this form, it is equivalent to Huber's affinoid comparison theorem \cite[Corollary 3.2.2]{HuberBook}. Let us sketch why this result, together with proper base change, also give a GAGA result for proper adic spaces. A more general statement can be found in \cite[Theorem 3.2.10]{HuberBook}.

\begin{corollary}[GAGA for rigid \'etale cohomology in the proper case]
Let $A$ be a ring with a distinguished element $t$ such that $(A,(t))$ is a
henselian pair. Assume either $A$ is noetherian or that $\widehat{A}_t[1/t]$ (where the completion is $t$-adic) is a strongly noetherian Tate ring, so the theory in \cite{HuberBook} applies. Let $X$ be a proper $A[1/t]$-scheme, and write $X^{ad}$ for the associated adic space over $\mathrm{Spa}(A[1/t],A)$. Then for any torsion \'etale sheaf $\sF$ on $X$ with pullback $\sF^{ad}$ to $X^{ad}$, the natural map gives an identification
\[ R\Gamma(X,\sF) \simeq R\Gamma(X^{ad},\sF^{ad})\]
between algebraic and analytic \'etale cohomologies.
\end{corollary}
\begin{proof}
By Nagata compactification, we can find a proper $A$-scheme $\mathfrak{X}$ extending $X$, i.e.,with an identification $\mathfrak{X} \times_{\spec(A)} \spec(A[1/t]) \simeq X$. Let $\mathfrak{F}$ be a torsion \'etale sheaf on $\mathfrak{X}$ extending $\sF$. Write $\widehat{\mathfrak{X}}$ for the $t$-adic completion of $\mathfrak{X}$. Consider the commutative square 
\begin{equation}
\label{RigidGAGA}
\xymatrix{ X^{ad} \ar[r] \ar[d] & \widehat{\mathfrak{X}} \ar[d] \\
		X \ar[r] & \mathfrak{X} }
\end{equation}
	of morphisms of locally ringed topoi (where each vertex is given the \'etale topology and the usual structure sheaf). The sheaf $\mathfrak{F}$ defines via pullback an \'etale sheaf on each of the topoi above, and we abusively denote this pullback by $\mathfrak{F}$ as well. We shall show that applying $R\Gamma(-,\mathfrak{F})$ to \eqref{RigidGAGA} gives a cartesian square; this implies the corollary as $R\Gamma(\mathfrak{X}, \mathfrak{F}) \simeq R\Gamma(\widehat{\mathfrak{X}}, \mathfrak{F})$ by the proper base change theorem 
	\cite[Exp.~XII, Th.~5.1]{SGA4}. 

To prove that applying $R\Gamma(-,\mathfrak{F})$ to \eqref{RigidGAGA} gives a Cartesian square, fix an affine open cover $\{U_i\}$ of the $A$-scheme $\mathfrak{X}$. Via pullback, this defines compatible open covers of $X$ by affine schemes, of $\widehat{X}$ by affine formal schemes, and of $X^{ad}$ by affinoid adic spaces. Using these covers to compute cohomology, and thanks to the affinoid comparison theorem \cite[Corollary 3.2.2]{HuberBook} (of which Corollary~\ref{BerkDescent} is an algebraic variant), it is therefore enough to prove that applying $R\Gamma(-,\mathfrak{F})$ to the analog of \eqref{RigidGAGA} for each $U_i$  gives a Cartesian square. But this follows  from formal glueing for \'etale cohomology (Proposition~\ref{arcarithm}).
\end{proof}

For future reference we note the following consequences 
of \Cref{BerkDescent}
for the functor $R
\mapsto R \Gamma( \spec(\hat{R}_t[1/t]), \sF)$.

\begin{definition}[$\arc_t$-equivalences] 
Let $A$ be a base ring containing an element $t$. 
A map of $A$-algebras $R \to R'$ is said to be an
\emph{$\arc_t$-equivalence} if for every $t$-adically complete rank $\leq 1$
absolutely integrally closed valuation ring $V$ where $t$ is a pseudouniformizer
with the structure of $A$-algebra, we have that
$\hom_A(R, V) \simeq \hom_A(R', V)$. 
\end{definition}

\begin{example} 
\label{examplesofarctequiv}
Let us give some examples of $\arc_t$-equivalences. Let $R$ be an $A$-algebra. The following maps are $\arc_t$-equivalences:
\begin{enumerate}
\item The map $R \to R/( t^\infty\text{-torsion})$. 
\item The map $R \to \hat{R}_t$. 
\item If $R^+$ denotes the integral closure of $R$ in $R[\frac{1}{t}]$, then $R \to R^+$ is an $\arc_t$-equivalence. Indeed, if $V$ is any valuation ring over $A$ with $t \neq 0$ in $V$, then $V$ is integrally closed in $V[\frac{1}{t}]$. 
\item 
Assume $t \in R$ is a nonzerodivisor. 
Let $R^{tic} := \{x \in R[1/t] \mid t^c x^{\mathbb{N}} \in R \text{ for some } c \geq 0\}$; this is called the \emph{total integral closure} of $R$ in $R[1/t]$. It is easy to see that $R^{tic}$ is a subring of $R[1/t]$ that contains the integral closure of $R$ in $R[1/t]$, and equals it when $R$ is noetherian. Then $R \to R^{tic}$ is an $\arc_t$-equivalence. This follows just as in (3) since rank $\leq 1$ valuation rings are totally integrally closed in their fraction fields. 
\end{enumerate}
\end{example}

\begin{proposition}[Invariance under $\arc_t$-equivalences] 
\label{arctinvariance}
Fix a pair $(A, t \in A)$ and let $\sF$ be a torsion sheaf on $\spec(A)_{\et}$. 
Let $R \to R'$ be an $\arc_t$-equivalence of $A$-algebras. 
Then $R \Gamma( \spec(\hat{R}_t[1/t], \sF)) \simeq 
R \Gamma( \spec(\widehat{R'}_t[1/t], \sF))$. 
\end{proposition} 
\begin{proof} 
Let $F$ be the functor on $A$-algebras given by $F(S) = R \Gamma( \spec(
\hat{S}_t[1/t]), \sF)$. 
We have just seen (Corollary~\ref{BerkDesc2}) that $F$ satisfies
$\arc_t$-descent. 
Now $R \to R'$ is an 
$\arc_t$-cover, so \begin{equation} \label{limauxdiag} F(R) \simeq \varprojlim ( F(R') \rightrightarrows F(R'
\otimes_R R') \triplearrows \dots ). \end{equation}Moreover, 
$R' \otimes_R R'  \to R'$ is an $\arc_t$-cover as well, and taking the \v{C}ech
nerve of $R' \otimes_R R' \to R'$, we find that 
$F( R' \otimes_R R') \simeq F(R')$, and similarly $F(R' \otimes_R R' \otimes_R
R') \simeq F(R')$, etc. Substituting this back into 
\eqref{limauxdiag} gives the claim. 
\end{proof}

Finally, as an application of Corollary~\ref{BerkDescent}, we obtain the following K\"unneth formula. 

\begin{proposition}[K\"unneth formula] 
\label{kunnethformula}
Let $V$ be an absolutely integrally closed valuation ring of rank $1$ with
pseudouniformizer $\pi$. Fix a prime number $\ell$ prime to the characteristic of
the residue field of $V$. Consider the functor $\widetilde{F}\colon 
\mathrm{Ring}_V \to \D( \mathbb{F}_\ell)^{\geq 0}$ given by $\widetilde{F}(R) = R
\Gamma( \spec( \hat{R}_\pi[1/\pi]), \mathbb{F}_\ell)$. 
Then $\widetilde{F}$ is a symmetric monoidal functor. That is, for any
$V$-algebras $R, R'$, the map
\[  R
\Gamma( \spec( \hat{R}_\pi[1/\pi]), \mathbb{F}_\ell) 
\otimes_{\mathbb{F}_\ell} R
\Gamma( \spec( \widehat{R'}_\pi[1/\pi]), \mathbb{F}_\ell)
\to 
R \Gamma( \spec( \widehat{(R\otimes_V R')}_\pi[1/\pi]), \mathbb{F}_\ell) 
\]
is an equivalence. 
\end{proposition} 
\begin{proof} 
Since the functor $\widetilde{F}$ on $V$-algebras is a finitary
$\arc_\pi$-sheaf (\Cref{BerkDescent}), 
we can work locally on $R, R'$ separately to reduce to the case where $R,
R'$ are themselves absolutely integrally closed rank $\leq 1$ valuation rings
(under $V$), thanks to Corollary~\ref{detectequivfinitaryarc}.\footnote{Here we
implicitly use that in $D(\mathbb{F}_l)^{\geq 0}$, the tensor product commutes
with totalizations in each variable. This follows since the tensor product
commutes with finite limits in each variable, and a totalization in any range
of degrees is given as a finite limit.} 
If $\pi = 0$ or $\pi$ is a unit in either of $R, R'$, then both terms vanish, so we may assume that $\pi$ is a pseudouniformizer in
each of $R, R'$. 
In this case, $\widetilde{F}(R), \widetilde{F}(R') = \mathbb{F}_\ell$ since the
\'etale cohomology of $\hat{R}_\pi[1/\pi]$ is equal to that of the
algebraically closed field $R[1/\pi]$, and similarly for $R'$. 
It remains to determine $\widetilde{F}( R \otimes_V R')$, i.e., the \'etale cohomology of 
$\widehat{(R\otimes_V R')}_\pi[1/\pi]$.

For this, we use a formal glueing square. Let $K$ denote the fraction field
of $V$ and let $R_K, R'_K$ denote the associated base changes.  
In view of the square \eqref{fibsquare5}, we obtain a homotopy pullback diagram
\[ \xymatrix{
R \Gamma( \spec( R \otimes_V R'), \mathbb{F}_\ell) \ar[d]  \ar[r] & 
R \Gamma( \spec (R_K \otimes_K R'_K), \mathbb{F}_\ell) \ar[d]  \\
R \Gamma( \spec( (R/\pi) \otimes_{(V/\pi)} (R'/\pi)), \mathbb{F}_\ell) \ar[r] &  
\widetilde{F}(R \otimes_V R')
}. \]
Now the top right and bottom left squares are just $\mathbb{F}_l$, thanks to
the K\"unneth formula in \'etale cohomology for qcqs schemes over a separably
closed field \cite[Cor.~1.11]{finitude} and a limit argument.  
The top left square is $\mathbb{F}_\ell$ by a result of Huber, \cite[Cor.~4.2.7]{HuberBook}. 
Therefore, $\widetilde{F}( R \otimes_V R') \simeq \mathbb{F}_\ell$ and the natural map 
$\widetilde{F}(R) \otimes_{\mathbb{F}_\ell} \widetilde{F}(R')\to \widetilde{F}(R
\otimes_V R')$ is an equivalence. 
This proves that $\widetilde{F}$ is symmetric monoidal. 
\end{proof} 

\begin{remark} 
The assumption that $\ell$ is invertible on the residue field of $V$ in Proposition~\ref{kunnethformula} is necessary. Indeed, even if $K$ is a nonarchimedean algebraically closed field of characteristic zero with residue  characteristic $p$,  the analog of Proposition~\ref{kunnethformula} fails for the functor $R \mapsto R\Gamma(\spec(\widehat{R}[1/p]), \mathbf{F}_p)$, already on the category of smooth $\mathcal{O}_K$-algebras. Nevertheless, one can show that this functor can be realized as the $\varphi$-fixed points on another functor (the global sections of the tilted structure sheaf, as in \cite{ScholzeRigid}, on the rigid space $\mathrm{Spa}(\widehat{R}[1/p])$ attached to $\mathrm{Spf}(\widehat{R}$)) that does satisfy a K\"unneth formula in a certain completed sense. 
\end{remark} 

\subsection{A variant of excision}
We end with a slight  generalization of excision; the statement is formulated in a fashion that could be potentially useful in relating the \'etale cohomology of different formal schemes giving the same Berkovich space.  First, we record the following simple observation. 
\begin{proposition} 
\label{univhomeo}
Let $Y \to X$ be a universal homeomorphism of qcqs $R$-schemes, for $R$ a base
ring. If $\sF$ is a
$v$-sheaf on $\mathrm{Sch}_{qcqs, R}$ with values in any $\infty$-category that
has all small limits, then  $\sF(X) \to \sF(Y)$ is an equivalence. 
\end{proposition} 
\begin{proof} 
This is similar to Proposition~\ref{arctinvariance}. 
Since $Y \to X$ is a universal homeomorphism, it is clearly a $v$-cover (as we
can lift specializations). Thus, we have
\begin{equation} \label{sFXexpr} \sF(X) \simeq \varprojlim ( \sF(Y) \rightrightarrows \sF(Y \times_X Y)
\triplearrows \dots ).   \end{equation}

Moreover, the map $\Delta\colon  Y \to Y \times_X Y$ is a universal homeomorphism too, and
hence a $v$-cover,
so we recover an expression for $\sF(Y \times_X Y)$ in terms of 
$\sF(Y), \sF(Y \times_{Y \times_X Y} Y), \dots$; since the 
map $\Delta$ is an immersion, this simplifies to $\sF(Y) \simeq \sF(Y \times_X
Y)$. This also holds for the iterated fiber products of $Y$ over $X$. 
Returning to \eqref{sFXexpr}, we find also that $\sF(X) \simeq \sF(Y)$ as
desired. 
\end{proof}

\begin{theorem}
\label{TICarc}
Let $A \to B$ be a map of commutative rings, and let $t \in A$ be an element
which is a nonzerodivisor on both $A$ and $B$ and such that $A^{tic} \simeq
B^{tic}$, where $A^{tic}$ denotes the total integral closure of $A$ in $A[1/t]$
(\Cref{examplesofarctequiv}), and similarly for $B^{tic}$. 
Let $\mathcal{C}$ be an $\infty$-category that has all small limits. 
Let $\mathcal{F}$ be a $\mathcal{C}$-valued  $\arc$-sheaf on $\mathrm{Ring}_A$ (such as the functor from Proposition~\ref{vDescentEtaleCoh}). Then the square
\[ \xymatrix{\mathcal{F}(A) \ar[r] \ar[d] & \mathcal{F}(A/tA) \ar[d] \\
		  \mathcal{F}(B) \ar[r] & \mathcal{F}(B/tB) }\]
is cartesian. 
\end{theorem}

The conditions on $A \to B$ appearing above are satisfied, for example, if $A \to B$ is an integral map of $t$-torsionfree rings such that $A[1/t] \simeq B[1/t]$. 

\begin{proof}
Set $A' := B \times_{B/tB} A/tA$, so we have a factorization $(A,tA) \to (A',tB)
\to (B,tB)$ of maps of pairs with the second map being an excision datum. As
$\mathcal{F}$ is excisive by Theorem~\ref{mainthm} and
Corollary~\ref{ExcArcSheafPull}, it is enough to show that $\mathcal{F}(A) \simeq \mathcal{F}(A')$.  By Proposition~\ref{univhomeo}, it suffices to show that $\mathrm{Spec}(A') \to \mathrm{Spec}(A)$ is a universal homeomorphism. 

We first check that $A \to A'$ is integral. As $A'$ is an extension of $A/tA \simeq A'/tB$ by $tB$, we can write any $x \in A'$ as $a + tb$ where $a \in A$ and $tb \in tB \subset A'$. To show integrality of $x$ over $A$, it is enough to show the integrality of $tb$ over $A$. But $A^{tic} \simeq B^{tic}$, so the element $t^n b^n \in B$ actually lies in $A$ (in fact, in $t^{n-c}A$ for some constant $c$) for $n \gg 0$, which proves integrality. 

As integral maps are universally closed, to see that $\mathrm{Spec}(A') \to
\mathrm{Spec}(A)$ is a universal homeomorphism, it suffices to show that
$f\colon A[1/t] \to A'[1/t]$ and $g\colon A/tA \to A'/tA'$ are universal
homeomorphisms on $\mathrm{Spec}(-)$. The claim for $f$ follows from $f$ being an isomorphism. For $g$, note that $A/tA \to A'/tB$ is an isomorphism by construction, so it is enough to check that the surjection $A'/tA' \to A'/tB$ has nilpotent kernel; equivalently, we must show that for any $tb \in tB$, we have $(tb)^n \in tA'$ for $n \gg 0$, which was already shown in the previous paragraph. 
\end{proof}

\newpage

\section{An application: Artin--Grothendieck vanishing in rigid geometry}

The classical Artin--Grothendieck vanishing theorem in algebraic geometry gives a bound on the cohomological dimension of affine varieties.

\begin{theorem}[{Artin--Grothendieck \cite[Cor.~3.2, Exp.~XIV]{SGA4}}] 
Let $k$ be a separably closed field and let $\ell$ be a prime number invertible in $k$.
Let $A $ be a finitely generated $k$-algebra of dimension
$d$. Let $\sF$ be an $\ell$-power torsion \'etale sheaf on $\spec(A)_{\et}$. Then $H^j(\spec(A), \sF) = 0$ for $j > d$.
\end{theorem} 

The assumption that $\ell \neq \mathrm{char}(k)$ is not
necessary: if $\ell= \mathrm{char}(k)$, we even have vanishing above degree $1$
(and even above degree $0$ if $\dim(A) = 0$). For this reason, in the sequel we focus on the interesting case discussed above.

In this section, we prove an analog of the classical Artin--Grothendieck
vanishing theorem in rigid analytic geometry, strengthening recent results of
Hansen \cite{HansenArtin}. We use the  basic objects of rigid analyic geometry, 
 cf.~\cite[Ch.~6--7]{BGR} for a general reference. 

\renewcommand{\O}{\mathcal{O}}

\begin{notation}
Throughout this section, we  let $K$ be a  complete, separably closed
nonarchimedean field with (nontrivial) absolute value $|\cdot|\colon  K \to
\mathbb{R}_{\geq 0}$. We let $\O_K \subset K$ be the ring of integers, and $\pi
\in \mathcal{O}_K$ a pseudouniformizer. We fix a prime number $\ell$ which is
different from the characteristic of the residue field of $K$. Write
$\mathcal{O}_K \langle X_1,...,X_n \rangle$ for the $\pi$-adic completion of the
polynomial ring $\mathcal{O}_K[X_1,...,X_n]$ and $T_n := K \langle X_1,...,X_n \rangle := \mathcal{O}_K \langle X_1,...,X_n \rangle[\frac{1}{\pi}]$
for the $n$-variable Tate algebra. 
In this section, all ring theoretic completions are $\pi$-adic ones unless otherwise specified.
We freely use the theory of \emph{topologically finite type} $K$-algebras (i.e.,
quotients of $T_n$ for some $n$) and their
completed tensor products and 
rational localizations, as treated in \emph{loc.~cit}. 
\end{notation}

We will consider the \'etale cohomology of topologically finite type $K$-algebras, considered as
abstract rings. It is known
that these agree with the \'etale cohomology of rigid analytic varieties (e.g., \cite[Theorem
3.2.1]{HuberBook}), but we will try to minimize the use of  this language for simplicity; nevertheless, this comparison provides the proper context for many of the statements that follow.

Our main result about \'etale cohomology of topologically finite type $K$-algebras is the following theorem. 

\begin{theorem} 
\label{rigidartin}
Let $K$ be a complete, algebraically closed nonarchimedean field. 
Suppose that $A$ is a $K$-algebra which is topologically finite type and let $d
=\mathrm{dim}(A)$. Let $\sF$ be a torsion abelian sheaf on $\spec(A)_{\et}$. 
Then:
\begin{enumerate}
\item We have  $H^i( \spec(A), \sF) = 0$ for $i > d+ 1$. 
\item If $\sF$ is an $\ell$-power torsion sheaf ($\ell$ prime to the
characteristic of the residue field of $K$), then we have $H^i( \spec(A), \sF) = 0$ for $i > d$. 
\end{enumerate}
\end{theorem}

This extends recent results of Hansen \cite{HansenArtin}. In particular, in
Theorem 1.3 of {\em loc.\ cit.} part (2) of  the above result is proved in the case
when $A$ descends to a discretely valued field and $K$ has characteristic
zero.\footnote{Hansen \cite{HansenArtin} works with Zariski-constructible
sheaves on the \'etale site of the associated rigid analytic space.} 
Our result confirms Hansen's Conjecture 1.2 when $K$ has characteristic zero,
in view of the comparison \cite[Theorem 1.7]{HansenArtin}. 

\begin{remark} 
We would expect that part (2) of Theorem~\ref{rigidartin} holds for arbitrary torsion sheaves. By part (1), this amounts to showing the following: if $A$ is a topologically finite type $K$-algebra of dimension $d$ and $p$ is the residue characteristic of $K$, then $H^{d+1}(\spec(A), \sF) = 0$ for all $p$-torsion sheaves $\sF$ on $\spec(A)_{\et}$. Two special cases of this expectation are within reach: 
\begin{enumerate}
\item If $K$ itself has characteristic $p$, then this assertion is straightforward: the $\mathbb{F}_p$-\'etale cohomological dimension of any affine $\mathbb{F}_p$-scheme is $\leq 1$ (thanks essentially to the Artin-Schreier sequence). 
\item The algebraization method used to prove Lemma~\ref{specialcaseartin} below can be adapted to prove this statement when $A$ is smooth (in the sense of rigid spaces) and $\sF$ is constant, thanks to \cite[Theorem 7 and Remark 2 on page 587]{Elkik}.
\end{enumerate}
Nevertheless, if $K$ has characteristic $0$ with its residue field having
characteristic $p$, the general case remains out of reach by our methods. One difficulty is
that that over such $p$-adic fields, $p$-adic \'etale cohomology behaves quite
differently on affinoids than its $\ell$-adic counterpart. For example, it is
almost never finite dimensional (unlike the $\ell$-adic case), even though it
does take finite dimensional values for proper rigid analytic varieties
\cite{ScholzeRigid}, \cite[Theorem 3.17]{Scholzesurvey}. More crucially, the
tensor product trick used in Proposition~\ref{rigidartinconstant} to improve the
bound from $d+1$ to $d$ in the $\ell$-adic case is not available in the $p$-adic
case. However, we have been informed by O.~Gabber that the result should be
provable using the partial algebrization techniques used in \cite{GO}. 
\end{remark}

We review some facts about the \'etale cohomology of affinoid varieties. 
First we need the K\"unneth formula. 

\begin{proposition}[K\"unneth formula]
\label{Kunnethaffinoid}
Given topologically finite type $K$-algebras $A, B$, the natural map
\[ R \Gamma( \spec(A) , \mathbb{F}_\ell) \otimes_{\mathbb{F}_\ell}  R \Gamma(
\spec(B), \mathbb{F}_\ell) \to R \Gamma( \spec(A \hat{\otimes}_K B),
\mathbb{F}_\ell) \]
is an equivalence in $\D(\mathbb{F}_\ell)$. 
\end{proposition} 
\begin{proof} 
Let $A_0 \subset A, B_0 \subset B$ be open bounded subrings. 
Then $A_0, B_0$ are $\pi$-adically complete, and  $A \hat{\otimes}_K B$ is obtained by inverting $\pi$ in $A_0 \widehat{\otimes}_{\mathcal{O}_K} B_0$. The result now follows from the K\"unneth formula of
Proposition~\ref{kunnethformula}. 
\end{proof}

\newcommand{\spm}{\mathrm{Spm}}

Next, we need to observe that (algebraic) \'etale cohomology satisfies descent in the
analytic topology. 
We treat a special case of this. 

\begin{proposition}[Descent in the analytic topology] 
\label{analyticdescent}
Let $A$ be a topologically finite type $K$-algebra and let $f, g \in A$ generate the unit ideal. 
Then for any torsion abelian sheaf $\sF$ on $\spec(A)$, we have a pullback square
\[ \xymatrix{
R \Gamma( \spec(A), \sF) \ar[d] \ar[r] & 
R \Gamma( \spec(A\left \langle \frac{f}{g}\right\rangle), \sF) . \ar[d]  \\
R \Gamma( \spec(A\left \langle \frac{g}{f}\right\rangle), \sF) \ar[r] & 
 R \Gamma( \spec(A\left \langle \frac{f}{g}, \frac{g}{f}\right\rangle), \sF)
}\]
\end{proposition} 

In the language of rigid spaces, this square captures the Mayer--Vietoris sequence for the affinoid space $X := \mathrm{Spa}(A,A^\circ)$ for $A$ attached to the open cover $X(\frac{f}{g}) := \{x \in X : |f(x)| \leq |g(x)|\}$ and $X(\frac{g}{f})$. In particular, the proposition follows immediately from the comparison between analytic and algebraic \'etale cohomology of affinoids. Alternately, we can argue directly as follows using the machinery of this paper.

\begin{proof} 
We consider the functor $F$ on $\mathcal{O}_K$-algebras given by $F(S) = R
\Gamma( \spec(\widehat{S}[1/\pi]), \sF)$, which we have seen satisfies
$\arc_\pi$-descent (\Cref{BerkDescent}). 

Let $A_0 \subset A$ be an open bounded subring; rescaling $f,g$ we can assume
that $f, g \in A_0$, and they generate an open ideal of $A_0$. 
We consider the $A_0$-algebras $S_1 =  A_0[T]/(fT - g) , S_2 = A_0[U]/(gU -
f)$ and $S= S_1 \times S_2$. Then the map $A_0 \to S$ is an $\arc_\pi$-cover and
the maps $S_1
\otimes_{A_0} S_1 \to S_1, \  S_2 \otimes_{A_0} S_2 \to S_2$ are $\arc_\pi$-equivalences. 
Form the \v{C}ech nerve of $A_0 \to S$ and apply the functor $F$ to it,
obtaining a limit diagram; in
light of the above observation and Proposition~\ref{arctinvariance}, this
translates to a pullback square
\[ \xymatrix{
F(A_0) \ar[d]  \ar[r] &  F(S_1) \ar[d]  \\
F(S_2) \ar[r] &  F(S_1 \otimes_{A_0} S_2). 
}\]
This is precisely the claimed pullback square. 
\end{proof}

Next, we review (what amounts to) a special case of Huber's quasi-compact base
change theorem, \cite[Sec.~4.4]{HuberBook}. More precisely, we prove a continuity property for \'etale cohomology of affinoids that roughly says that the \'etale cohomology of a Zariski closed subset of an affinoid can be calculated as the filtered colimit of the \'etale cohomology of rational subsets that contain it.

\begin{proposition} 
\label{qcbase}
Let $R$ be a topologically finite type $K$-algebra. Let $f \in R$. This gives an inductive sequence of topologically finite type $K$-algebras $\{R\left \langle f/\pi^r\right\rangle\}$ with transition maps $R\left \langle f/\pi^r\right\rangle \to R\left \langle f/\pi^{r+1}\right\rangle \to \cdots$. For any torsion sheaf $\sF$ on $\spec(R)_{\et}$, there is an equivalence 
\begin{equation} \label{qceq} \varinjlim_r \left(R \Gamma( \spec( R\left \langle f/\pi^r\right\rangle )
, \sF) \right) \simeq R \Gamma(\spec(R/f), \sF). \end{equation}
\end{proposition} 
\begin{proof} 
Let $R_0 \subset R$ be an open bounded $\mathcal{O}_K$-subalgebra which is topologically finite type over $\mathcal{O}_K$. 
Rescaling $f$, we may assume $f \in R_0$ as well. 

We consider the functor $F\colon  \mathrm{Ring}_{R_0} \to \D(\mathbb{Z})^{\geq 0}$ given by
$F(T) = R \Gamma( \spec( \widehat{T}[1/\pi]), \sF)$. For each $r$, we consider the algebra $S_r := R_0[U_r]/( \pi^r U_r - f)$. 
We have a sequence of maps 
$S_1 \to S_2 \to \dots \to S_r \to S_{r+1} \to \dots$, where the map $S_r \to
S_{r+1}$ sends $U_r \mapsto \pi U_{r+1}$.  We have that 
$\widehat{S_r}[1/\pi] = R\left \langle f/\pi^r\right\rangle $. Also, $R_0/f$ is
a topologically finite type $\mathcal{O}_K$-algebra and hence $\pi$-adically complete.
Therefore, thanks to \Cref{BerkDescent} and \Cref{arctinvariance}, it suffices to show that the natural map 
\[ \varinjlim_r S_r \to R_0/f  \]
 is an $\arc_\pi$-equivalence. In other words, for any rank $1$ valuation ring $V$ (over $\mathcal{O}_K$) with pseudouniformizer $\pi$, we must show that every map $\varinjlim_r S_r \to V$ extends uniquely to maps $R_0/f \to V$. The uniqueness is immediate from the surjectivity of $\varinjlim_r S_r \to R_0/f$. For existence, we must show that $f \in R_0$ and each $U_r \in S_r$ maps to $0$ in $V$ under any map $\varinjlim_r S_r \to V$. The image of $f$ in $\varinjlim_r S_r$ is divisible by all powers of $\pi$, and hence must map to $0$ in $V$ as $V$ is $\pi$-adically separated. But then $U_r = 0$ since $\pi^r U_r = f$ in $S_r$ and $V$ is also $\pi$-torsionfree.
\end{proof}

Next, we prove a special case of Artin--Grothendieck vanishing. 
\begin{lemma} 
\label{specialcaseartin}
Let $f_1, \dots, f_m, g \in  T_n$ generate the unit ideal. 
Let $A$ be an algebra which is finite \'etale over $T_n \left \langle
 \frac{f_1, \dots, f_m}{g} \right\rangle$. Then 
 for any torsion abelian group $\Lambda$, 
 we have 
$R \Gamma( \spec(A), \Lambda) \in \D( \mathbb{Z})^{\leq n}$ (i.e.,
Artin--Grothendieck
vanishing holds for $A$).  
 \end{lemma} 
\begin{proof} 
Note that $\mathrm{min}( |f_1|, \dots, |f_m|, |g|)$ is uniformly bounded below on the unit ball
of $K^n$ since $f_1, \dots, f_m, g$ generate the unit ideal. 
Therefore, we can assume that $f_i, g \in K[X_1, \dots, X_n] \subset T_n$ (for
each $i$) without
changing the rational subset $\left\{x: |f_i(x)| \leq |g(x)|, \forall i\right\}$ in the
unit ball. 
Rescaling the $f_i$ and $g$ and possibly enlarging the set, we may assume furthermore that
they belong to $\mathcal{O}_K[X_1,
\dots, X_n]$ and that $f_1$ is a power of $\pi$, cf.~\cite[Prop.~3.9]{Hubercts}. 

Now  $R_0 = \O_K [X_1, \dots, X_n, U_1, \dots, U_m]/( gU_1 - f_1, \dots, g U_m - f_m)$ is a ring of finite type over
$\mathcal{O}_K$ such that $T_n \left \langle \frac{f_1, \dots, f_m}{g}
\right\rangle = \widehat{R_0}[1/\pi]$.  Note that $R_0[1/\pi] \simeq K[X_1,
\dots, X_n, 1/g]$: imposing the equation $gT_1 - f_1$ forces $g$ to become
invertible on inverting $\pi$ (as $f_1$ is a power of $\pi$), thus trivializing
the other equations. In particular, $R_0[1/\pi]$ has Krull dimension $n$.
Furthermore, we claim that the $\pi$-power torsion in $R_0$ is bounded. 
Let $J_0 \subset R_0$ be the ideal of $\pi$-power torsion. 
Then $R_0/J_0$ is a finitely presented $\O_K[X_1, \dots, X_n, U_1, \dots,
U_m]$-module thanks to \cite[Tag 053E]{stacks-project}, which forces $J_0$ to be
a finitely generated module. 

Let $R_0^h$ denote the henselization of $R_0$ along $\pi$. 
By  \cite[Proposition 5.4.54]{GabberRamero} (which is the non-noetherian
version of a result of Elkik \cite{Elkik}; it applies to the quotient of $R_0$
by its bounded $\pi$-power torsion) 
finite \'etale covers of $R_0^h[1/\pi]$ and $\widehat{R_0}[1/\pi]$ are identified. It follows that there exists a 
finite \'etale $R_0^h[1/\pi]$-algebra $A'$ such that
\[ A \simeq A' \otimes_{R_0^h[1/\pi]} \widehat{R_0}[1/\pi] .  \]
By the Fuijwara-Gabber theorem (Theorem~\ref{GabberFujiwarathm}), 
it follows that $A$ and $A'$ have the same \'etale cohomology.  But $A'$ is finite \'etale over $R_0^h[1/\pi]$.  Now $R_0^h$ is a filtered colimit of \'etale $R_0$-algebras, 
so $R_0^h[1/\pi]$ is a filtered colimit of \'etale $R_0[1/\pi] \simeq K[X_1, \dots, X_n, 1/g]$-algebras and therefore so is $A'$. 
In particular, it is a filtered colimit of finite type $K$-algebras of dimension $n$. 
We can apply classical Artin--Grothendieck vanishing (and the fact that \'etale cohomology of
rings commutes with filtered colimits) to conclude that $R \Gamma( \spec(A),
\Lambda) = R \Gamma( \spec(A'), \Lambda) \in \D(\mathbb{Z})^{\leq n}$,
as desired. 
\end{proof} 

Next, we prove 
Theorem~\ref{rigidartin} in the case of a constant sheaf. 
\begin{proposition} 
\label{rigidartinconstant}
Let $R$ be a topologically finite type $K$-algebra of dimension $d$. Then: 
\begin{enumerate}
\item For all torsion abelian groups $\Lambda$, we have $R \Gamma( \spec(R),
\Lambda) \in \D (\mathbb{Z})^{\leq d+1}$. 
\item For all $\ell$-power torsion abelian groups $\Lambda$, we have $R \Gamma( \spec(R),
\Lambda) \in \D (\mathbb{Z})^{\leq d}$. 
\end{enumerate}
\end{proposition} 
\begin{proof}
Consider the following assertion:

\begin{itemize}[leftmargin=0.8in, rightmargin=0.8in]
\item[$(\mathrm{Art}_d):$] For all topologically finite type $K$-algebras $S$ of dimension $d$ and all
torsion abelian groups $\Lambda$, we have $R \Gamma(\spec(S), \Lambda)  \in \D(\mathbb{Z})^{\leq d + 1}$.
\end{itemize}

We shall prove $(\mathrm{Art}_d)$ by induction on $d$; this will imply part (1) of the Proposition.  The case $d  = 0$ is trivial.  Suppose that we know $(\mathrm{Art}_{d-1})$.  Fix a $d$-dimensional topologically finite type $K$-algebra $R$. 
Without loss of generality, we may assume that $R$ is reduced. 
By Noether normalization \cite[Sec.~6.1.2, Cor.~2]{BGR}, there is a finite map 
\[ T_d \to R.  \]
Since $T_d$ is an integral domain, if $\mathrm{char}(K) = 0$, then  there exists
an $f \in T_d$ such
that $R[1/f]$ is finite \'etale over $T_d[1/f]$. 

If $\mathrm{char}(K) = p> 0$, 
we have to work slightly harder since the map $T_d \to R$ given by Noether
normalization may have a nonreduced geometric generic fiber. 
To fix this, for each $t$, 
consider the $t$th iterated Frobenius $\phi^t\colon T_d \to T_d$ and the new
$T_d$-algebra $R_t = (R
\otimes_{T_d, \phi^t} T_d)_{\mathrm{red}}$. 
Each of these $T_d$-algebras $R_t$ comes with a universal homeomorphism $R
\to R_t$, so it suffices to prove the result 
for any $R_t$. But for $t \gg 0$, the map $T_d \to R_t$ has reduced geometric
generic fiber. 
Therefore, up to replacing $R$ by some $R_t$, we find that there exists $f \in
T_d$ such that $T_d[1/f] \to R[1/f]$ is finite \'etale.

For each $r  > 0$, we consider the following rings: 
\begin{enumerate}

\item   $R_r^{(1)} = R \hat{\otimes}_{T_d} \left(T_d\left \langle
\pi^r/f\right\rangle\right) = R\left \langle \frac{\pi^r}{f}\right\rangle$. 

\item $R_r^{(2)} = R \hat{\otimes}_{T_d} T_d\left \langle f/\pi^r\right\rangle = R \left \langle \frac{f}{\pi^r}\right\rangle$. 

\item  $R_r^{(3)} = R \hat{\otimes}_{T_d} T_d\left \langle \pi^r/f,  f/\pi^r\right\rangle = R\left \langle \frac{f}{\pi^r}, \frac{\pi^r}{f}\right\rangle$. 

\end{enumerate}

We have a pullback square
from Proposition~\ref{analyticdescent}, 
\[ \xymatrix{
R \Gamma( \spec(R) , \Lambda) \ar[d]  \ar[r] &  R \Gamma( \spec(R_r^{(1)}),
\Lambda) \ar[d]  \\
R \Gamma( \spec(R_r^{(2)}), \Lambda) \ar[r] &  R \Gamma( \spec(R_r^{(3)}),
\Lambda)
}.\]
Since $R\left \langle \pi^r/f\right\rangle$ is finite \'etale over $T_d\left
\langle \pi^r/f\right\rangle$ (note that the completed tensor products in the definitions
of $R_r^{(i)}, i = 1, 2, 3$ can be replaced by tensor products), and similarly $R\left \langle \pi^r/f,
f/\pi^r\right\rangle$ is finite \'etale over $T_d \left \langle \pi^r/f,
f/\pi^r\right\rangle$, by Lemma~\ref{specialcaseartin} we find that the top right and bottom right
corners belong to $\D(\mathbb{Z})^{\leq d}$, for any $r$. 

Suppose that $x \in H^{j}( \spec(R), \Lambda)$ for $j \geq d + 2$ is nonzero. Using the long
exact sequence in cohomology from the above fiber square, we find that 
$x$ must map to a nonzero class in $H^j( \spec( R \left \langle
f/\pi^r\right\rangle, \Lambda)$ for each $r > 0$. This contradicts the fact that
the fact that the colimit of these groups is zero by the inductive hypothesis
and Proposition~\ref{qcbase}. 
This completes the proof of the claim $(\mathrm{Art}_d)$ and, by induction, the
first half of the theorem. 

For the second half of the theorem, we may assume that $\Lambda = \mathbb{F}_\ell$
and use the ``tensor power trick.'' 
For any topologically finite type $K$-algebra $A$ which is $d$-dimensional, we have seen that 
$R \Gamma( \spec(A), \mathbb{F}_\ell) \in \D(\mathbb{F}_\ell)^{\leq d+1}$. 
Since $A \hat{\otimes}_K A$ is $2d$-dimensional, we get that 
$R \Gamma ( \spec( A \hat{\otimes}_K A) , \mathbb{F}_\ell)  \in \D(
\mathbb{F}_\ell)^{\leq 2d + 1}$ again by the first half of the theorem. Using the K\"unneth formula
(Proposition~\ref{Kunnethaffinoid}), we have 
$R \Gamma ( \spec( A \hat{\otimes}_K A) , \mathbb{F}_\ell)  \simeq 
R \Gamma ( \spec( A ) , \mathbb{F}_\ell)^{\otimes 2}$, 
and
this forces $R \Gamma( \spec(A), \mathbb{F}_\ell) \in \D(\mathbb{F}_\ell)^{\leq d}$. 
\end{proof}

We now explain the proof of the full result. 
\begin{proof}[Proof of Theorem~\ref{rigidartin}] 
For simplicity we treat the $\ell$-power torsion case (i.e., the second half of
the theorem); the other case is analogous. By a filtering $\sF$ in terms of its constructible subsheaves \cite[Tag 03SA]{stacks-project} and in terms of the $\ell$-adic filtration, it suffices to prove the following statement: 

\begin{itemize}[leftmargin=0.8in, rightmargin=0.8in]
\item[$(\ast_d):$] If $A$ is a topologically finite type $K$-algebra of dimension $d$ and $\sF$ is a constructible \'etale $\mathbb{F}_\ell$-sheaf on $X := \spec(A)$, then $H^i(X, \sF) = 0$ for $i > d$. 
\end{itemize}

We already know $(\ast_d)$ for all $d$ when $\sF$ is constant
(\Cref{rigidartinconstant}). We shall reduce the general case to this one by a
standard devissage procedure. Note that if $A$ is a topologically finite type $K$-algebra of dimension
$d$, then any finite $A$-algebra (such as a quotient) is also a topologically finite type $K$-algebra
and has dimension $\leq d$. Indeed, this follows from the analogous statement for topologically finite type 
$\mathcal{O}_K$-algebras: if $R_0$ is topologically finite type over $\mathcal{O}_K$ and $S_0/R_0$ is
finite, then $S_0$ is a quotient of some $\mathcal{O}_K\left \langle
U_1, \dots, U_m\right\rangle$ and hence topologically finite type as desired.

\newcommand{\sL}{\mathcal{L}}
Let us prove $(\ast_d)$ by induction on $d$. When $d=0$, the statement is clear
as $A$ is a finite product of copies of $K$. Assume that $(\ast_{d-1})$ holds
and fix $X := \spec(A)$ and $\sF$ as in $(\ast_d)$. There exists a dense open
$j \colon U \hookrightarrow X$ such that $\sL = \sF|_U$ is locally constant. As $j$ is
dense open, the cokernel of the canonical injective map $j_! \sL \to \sF$ is
supported on some closed subset of dimension $< d$. By induction and the long
exact sequence, we may thus assume $\sF = j_! \sL$. Moreover, as $\sF$
decomposes as a finite direct sum of its restrictions to each connected
component of $U$, we may replace $\sF$ by a direct summand to assume that $U$ is
connected (though it might no longer be dense in $X$). By \cite[Tag
0A3R]{stacks-project}, there exists a finite \'etale morphism $\pi\colon V \to U$ of
degree prime to $\ell$ such that $\pi^*(\sL)$ is an iterated extension of
constant sheaves. As $\pi$ has degree prime to $\ell$, the ``m\'ethod de la
trace'' \cite[Tag 03SH]{stacks-project} coupled with the assumption on $\pi^*
\sL$ shows that we may assume that $\sL = \pi_* \mathbb{F}_\ell$. Zariski's main theorem then gives a commutative diagram
\[ \xymatrix{ V \ar[r]^-{j'} \ar[d]^-{\pi} & Y \ar[d]^-{\overline{\pi}} \\
		U \ar[r]^-{j} & X}\]
with $j'$ an open immersion and $\overline{\pi}$ finite. In particular, $Y$ is
the spectrum of an algebra topologically finite type  over $K$ of dimension $\leq d$.  By replacing $Y$ with the closure of $V$ if necessary, we may also assume $j'$ has dense image. Now $H^i(X, j_! \pi_* \mathbb{F}_\ell) \simeq H^i(Y, j'_! \mathbb{F}_\ell)$ as $\overline{\pi}_! \simeq \overline{\pi}_*$ and similarly for $\pi$ (and because these functors have no higher derived functors). We are thus reduced to showing that $H^i(Y, j'_! \mathbb{F}_\ell) = 0$ for $i > d$.  If $i\colon Z \to Y$ is the complementary embedding, then we have a short exact sequence
\[ 0 \to j'_! ( \mathbb{F}_\ell) \to \mathbb{F}_\ell \to i_* (\mathbb{F}_\ell) \to 0\]
of sheaves on $Y$. Taking the long exact sequence and using the result for constant sheaves (Proposition~\ref{rigidartinconstant}) gives the claim since $\dim(Z) < d$.
\end{proof}

\newpage

\bibliographystyle{amsalpha}
\bibliography{affine} 

\end{document}